\documentclass[12pt,a4paper]{amsart}
\usepackage{amsfonts}
\usepackage{amsthm}
\usepackage{amsmath}
\usepackage{amscd}
\usepackage[english]{babel}
\usepackage[latin2]{inputenc}
\usepackage{csquotes}
\usepackage{graphicx}
\usepackage[square,numbers]{natbib}
\usepackage{tikz}
\numberwithin{equation}{section}
\usepackage[margin=2.9cm]{geometry}
\usepackage{dsfont}
\usepackage{enumitem}
\usepackage{color}
\usepackage{lineno,hyperref}

\allowdisplaybreaks

\makeatletter
\newcommand{\xRightarrow}[2][]{\ext@arrow 0359\Rightarrowfill@{#1}{#2}}
\makeatother

\makeatletter
\newcommand*{\rom}[1]{\expandafter\@slowromancap\romannumeral #1@}
\makeatother

\allowdisplaybreaks

\theoremstyle{plain}
\newtheorem{theorem}{Theorem}[section]
\newtheorem{lemma}[theorem]{Lemma}
\newtheorem{corollary}[theorem]{Corollary}

\theoremstyle{definition}

\newtheorem{remark}[theorem]{Remark}
\newtheorem{example}[theorem]{Example}
\newcounter{aaa}[theorem]
\newenvironment{aaalist}{\begin{list}{%
\rm (\alph{aaa})   \hfill        \it}{\usecounter{aaa} %
\topsep0mm \partopsep0mm \parsep0mm \itemsep0mm %
\leftmargin2em \labelwidth2em \labelsep0em}}{\end{list}}

\newcounter{one}[theorem]
\newenvironment{onelist}{\begin{list}{%
\rm (\arabic{one}) \hfill           }{\usecounter{one} %
\topsep0mm \partopsep0mm \parsep0mm \itemsep0mm %
\leftmargin2em \labelwidth2em \labelsep0em}}{\end{list}}

\newenvironment{hylist}{\begin{list}{%
--               \hfill            }{%
\topsep0mm \partopsep0mm \parsep0mm \itemsep0mm %
\leftmargin2em \labelwidth2em \labelsep0em}}{\end{list}}

\newcounter{rom}

\newcommand{\I}{{\mathbb I}}
\newcommand{\N}{{\mathbb N}}
\newcommand{\R}{{\mathbb R}}

\newcommand{\CC}{{\mathcal C}}

\newcommand{\aaa}{{\bf a}}

\newcommand{\sss}{{\bf s}}
\newcommand{\ttt}{{\bf t}}
\newcommand{\uuu}{{\bf u}}
\newcommand{\vvv}{{\bf v}}

\newcommand{\xxx}{{\bf x}}

\newcommand{\XXX}{{\bf X}}

\newcommand{\nul}{{\bf 0}}

\newcommand*{\bigtimes}{\mathop{\raisebox{-.5ex}{\hbox{\huge{$\times$}}}}}

\begin{document}

	\title[]{How simplifying and flexible is the simplifying assumption in pair-copula constructions - analytic answers 
in dimension three and a glimpse beyond}
	\author[]{Thomas Mroz, Sebastian Fuchs, Wolfgang Trutschnig}
	\address{Department for Mathematics, University of Salzburg \\
Hellbrunnerstrasse 34, A-5020 Salzburg, Austria}
	\email{thomasmroz@a1.net \\ sebastian.fuchs@sbg.ac.at \\ wolfgang.trutschnig@sbg.ac.at}

%\begin{abstract} The aim of this paper is to provide ......................
%\end{abstract}

\maketitle
\vspace*{-0.6cm}
\section*{Keywords}
Pair copula,
simplifying assumption,
conditional distribution,
Markov kernel,
dependence

\begin{abstract}
Motivated by the increasing popularity and the seemingly broad applicability of pair-copula constructions 
underlined by numerous publications in the last decade, in this contribution we tackle the unavoidable question on how 
flexible and simplifying the commonly used `simplifying assumption' is from an analytic perspective and provide 
answers to two related open questions posed by Nagler and Czado in 2016.  
Aiming at a simplest possible setup for deriving the main results we first focus on the three-dimensional setting.  
We prove that the family of simplified copulas is flexible in the sense that it is dense in the 
set of all three-dimensional co\-pulas with respect to the uniform metric $d_\infty$ - considering stronger 
notions of convergence like the one induced by the metric $D_1$, by weak conditional convergence, 
by total variation, or by Kullback-Leibler divergence, 
however, the family even turn out to be nowhere dense and hence insufficient for any kind of flexible approximation.  
Furthermore, returning to $d_\infty$ we show that the partial vine copula is never the optimal 
simplified copula approximation of a given, non-simplified copula $C$, and derive examples 
illustrating that the corresponding approximation error can be strikingly large and extend to more than 
28\% of the diameter of the metric space. Moreover, the mapping $\psi$ assigning each three-dimensional copula 
its unique partial vine copula turns out to be discontinuous with respect to $d_\infty$ 
(but continuous with respect to $D_1$ and to weak conditional convergence), 
implying a surprising sensitivity of partial vine copula approximations.    
The afore-mentioned main results concerning $d_\infty$ are then extended to the general 
multivariate setting. 
\end{abstract}

%%%%%%%%%%%%%%%%%%%%%%%%%%%%%%%%%%%%%%%%%%%%%%%%%%%%%%%%%%%%%%%%%%%%%%%%%%%%%%%%%%%%%%%%%%%%%%%%%%%%%%%%%%%%%%%%%%%%%%%%%%%%%%%%%%%%%%%%%%%%%%%%%%%%%%%%%%%%%%%%%%%%%%%%%%%%%%%%%%%%%%%%%%%%%%%%%%%%%%%%%%%%%%%%%%%%%%%%%%%%%%%%%%%%%%%%%%%%%%%%%%%%%%%%%%%%%%%%%%%%%%%%%%%%%%%%%%%%%%%%%%%%%%%%%%%%%%%%%%%%%%%%%%%%%%%%%%%%%%%%%%%%%%%%%%%%%%%%%%%%%%%%%%%%%%%%%%%%%%%%%%%%%%%%%%%%%%%%%%%%%%%%%%%%%%%%%%%%%%%%%%%%%%%%%%%%%%%%%%%%%%%%%%%%%%%%%%%%%%%%%%%%%%%%%%%%%%%%%%%%%%%%%%%%%%%%%%%%%%%%%%%%%%%%%%%%%%%%%%%%%%%%%%%%%%%%%%%%%%%%%%%%%%%%%%%%%%%%%%%%%%%%%%%%%%%%%%%%%%%%%%%%%%%%%%%%%%%%%%%%%%%%%%%%%%%%%%%%%%%%%%%%%%%%%%%%%%%%%%%%%%%%%%%%%%%%%%%%%%%%%%%%%%%%%%%%%%%%%%%%%%%%%%%%%%%%%%%%%%%%%%%%%%%%%%%%%%%%%%%%%%%%

\section{Introduction}

Pair-copula constructions (most well-known in the context of vine copulas) are a very popular bottom-up approach for 
constructing high-dimensional copulas out of several bivariate ones; they have a handy graphical representation 
and can be considered as an ordered sequence of trees.
Aiming at a significant reduction of complexity it is usually assumed that the so-called \emph{simplifying assumption}, 
saying that the copulas of the conditional distribution functions do not depend on the conditioning variables, holds.  

Considering the enormous number of scientific contributions working with and applying simplified pair-copulas 
(see, e.g., \cite{biller2011, chevallier2019, dalla2016, rui2019, xiong2014, zhi2020, zhang2020})
it is quite surprising that, apart from a few critical voices (see, e.g., \cite{acar2012,derumigny2017,gijbels2017}),
no analytic and systematic study on the approximation quality and flexibility of these concepts seems to have been published so far.

After an extensive literature research it seems that the publication coming closest to such a study was written by 
\citet{spanhel2019} who focus mainly on partial vine copulas (special simplified pair-copulas whose conditional distribution functions follow a certain intuitive construction principle) and show that partial vine copulas are optimal w.r.t. Kullback-Leibler divergence if the minimization is performed sequentially, but not necessarily if the estimation is done jointly. 
As stated in \cite{spanhel2019}, this \grqq{}implies that it may not be optimal to specify the true copulas in the first 
tree" of a simplified pair-copula approximation.

Motivated by the broad applicability of pair-copula constructions, in this contribution we study flexibility and the extent of simplification imposed by the simplifying assumption from an analytic perspective. For the sake of generality of the construction we do not directly assume absolute continuity and work with densities but build the analysis on 
conditional distributions (Markov kernels). Although most results are 
established in the three-dimensional setting we also sketch possible extensions 
to the general multivariate case. We first introduce and discuss the 
somewhat more general concept of \emph{simplified copulas}, i.e., copulas satisfying the simplifying assumption but do not necessarily follow the hierarchical bottom-up approach.
We show that, on the one hand, simplified copulas are very flexible in the sense that they are dense in the family of all three-dimensional copulas with respect to the uniform metric $d_\infty$ - this flexibility, however, 
gets lost when finer topologies like the one induced by the metric $D_1$, by 
weak conditional convergence, by total variation metric or by the Kullback-Leibler divergence are considered.  
In fact, we prove that the family of simplified copulas is even \emph{nowhere dense} with respect to either of these four topologies, and, thereby provide an answer to one of the questions posed by \citet{nagler2016}.

%We then show that simplified copulas are very flexible in the sense that they are dense in the family of all three-dimensional copulas with respect to the uniform metric $d_\infty$ 
%and hence provide an answer to one of the questions posed by \citet{nagler2016}. 

Returning to $d_\infty$ we then show that the partial vine copula of a given, non-simplified copula $C$ is \emph{never} the best-possible simplified copula approximation of $C$ (with respect to $d_\infty$).
%As a direct consequence, however, we then prove that the partial vine copula of a given, non-simplified copula $C$ is
%\emph{never} the best-possible simplified copula approximation of $C$ (with respect to $d_\infty$). 
More importantly, the error made by approximation via partial simplified vines may be strikingly large: 
in the worst case scenario the distance between a three-dimensional copula and its assigned partial vine copula is at least $3/16$ which corresponds to $28.125 \%$ of the diameter of the metric space. 
An analogous result holds in arbitrary dimensions, in this case the worst case distance is at least $1/8$. 
With these results we answer the question on \grqq{}how far off can we be by assuming a simplified model?" also 
posed by \citet{nagler2016}.    

Sticking to the analytic perspective we moreover focus on continuity properties of the
mapping $\psi$ assigning each three-dimensional copula its unique partial vine copula and show (among other things) 
that this mapping is \emph{not} continuous with respect to $d_\infty$. In other words: if $d_\infty(A,B)$ is small then in 
general we can not infer that $d_\infty(\psi(A),\psi(B))$ is small too. As a direct consequence, although simplified 
pair-copulas are \grqq{}highly flexible" (\cite{killiches2017})
and partial vine copulas \grqq{}can yield an approximation that is superior to competing approaches" (\cite{spanhel2019}), 
approximations in terms of partial vine copulas can be of very poor quality and lead to wrong conclusions.

%\bigskip
The rest of this paper is organized as follows: 
Section \ref{Sect.Prelim.} gathers preliminaries and notations that will be used in the sequel.
In Section \ref{Sect.Simplified} we introduce simplified copulas in dimension $d=3$,  
prove that the family of these copulas is dense in the metric space of all copulas with respect to $d_\infty$ 
(Corollary \ref{S.Dense}) and then show that with respect to either of the afore-mentioned four notions of convergence
the family is very small in the sense that it is nowhere dense.  
In Section \ref{Sect.Pair} we then focus on partial vine copulas and study the afore-mentioned mapping $\psi$ 
assigning each copula its simplified approximation.  
We discuss continuity of $\psi$ with respect to different notions of convergence (some lemmata and proofs are moved to the Supplementary to facilitate reading) and provide the afore-mentioned lower bound for the worst-case scenario 
(Sections \ref{Sect.Partial.Opt.} and \ref{Sect.Partial.Cont.}).
To avoid unnecessary complexity, in the first few sections we proceed as \cite{acar2012,hobaek2010,killiches2017,portier2018,spanhel2016} and restrict ourselves to the three-dimensional setting.
To underline generality of our findings, however, in Section \ref{Arb.Dim.} we extend some of our main results to 
the general multivariate setting and discuss the notion of so-called universally simplified copulas.  
Various examples and graphics illustrate both the obtained results and the ideas underlying the proofs.

%%%%%%%%%%%%%%%%%%%%%%%%%%%%%%%%%%%%%%%%%%%%%%%%%%%%%%%%%%%%%%%%%%%%%%%%%%%%%%%%%%%%%%%%%%%%%%%%%%%%%%%%%%%%%%%%%%%%%%%%%%%%%%%%%%%%%%%%%%%%%%%%%%%%%%%%%%%%%%%%%%%%%%%%%%%%%%%%%%%%%%%%%%%%%%%%%%%%%%%%%%%%%%%%%%%%%%%%%%%%%%%%%%%%%%%%%%%%%%%%%%%%%%%%%%%%%%%%%%%%%%%%%%%%%%%%%%%%%%%%%%%%%%%%%%%%%%%%%%%%%%%%%%%%%%%%%%%%%%%%%%%%%%%%%%%%%%%%%%%%%%%%%%%%%%%%%%%%%%%%%%%%%%%%%%%%%%%%%%%%%%%%%%%%%%%%%%%%%%%%%%%%%%%%%%%%%%%%%%%%%%%%%%%%%%%%%%%%%%%%%%%%%%%%%%%%%%%%%%%%%%%%%%%%%%%%%%%%%%%%%%%%%%%%%%%%%%%%%%%%%%%%%%%%%%%%%%%%%%%%%%%%%%%%%%%%%%%%%%%%%%%%%%%%%%%%%%%%%%%%%%%%%%%%%%%%%%%%%%%%%%%%%%%%%%%%%%%%%%%%%%%%%%%%%%%%%%%%%%%%%%%%%%%%%%%%%%%%%%%%%%%%%%%%%%%%%%%%%%%%%%%%%%%%%%%%%%%%%%%%%%%%%%%%%%%%%%%%%%%%%%%%
%\newpage

\section{Notation and preliminaries}
\label{Sect.Prelim.}

Throughout this paper we will write $\I := [0,1]$ and let $d\geq 2$ be an integer, which will be kept fixed.
Bold symbols will be used to denote vectors, e.g., $\mathbf{x}=(x_1,\ldots,x_d) \in \mathbb{R}^d$.  
The $d$-dimensional Lebesgue measure will be denoted by $\lambda^d$, in case of $d=1$ we will also simply write $\lambda$.
We will let $\mathcal{C}^d$ denote the family of all $d$-dimensional copulas, 
$M$ will denote the comonotonicity copula, 
$\Pi$ the independence copula
and, for $d=2$, $W$ will denote the countermonotonicity copula
(we omit the index indicating the dimension since no confusion will arise).  
For every $C \in \mathcal{C}^d$ the corresponding $d$-stochastic measure will be denoted by $\mu_C$, 
i.e. $\mu_C([\nul,\uuu]) = C(\uuu)$ for all $\uuu \in \I^d$, and $\mathcal{P}_\mathcal{C}$ will denote the 
family of all $d$-stochastic measures. 
  For more background on copulas and $d$-stochastic measures we refer to \cite{durante2016,nelsen2006}. 
  For every metric space $(S,\delta)$ the Borel $\sigma$-field on $S$ will be denoted by $\mathcal{B}(S)$.

  In what follows Markov kernels will play a prominent role:
	A \emph{Markov kernel} from $\R$ to $\mathcal{B}(\mathbb{R}^{d-1})$ is a mapping 
$K: \mathbb{R}\times\mathcal{B}(\mathbb{R}^{d-1}) \rightarrow \I$ such that for every fixed 
$E\in\mathcal{B}(\mathbb{R}^{d-1})$ the mapping 
$y\mapsto K(y,E)$ is (Borel-)measurable and for every fixed $y\in\mathbb{R}$ the mapping 
$E\mapsto K(y,E)$ is a probability measure. 
\\
  Given a real-valued random variable $Y$ and a real-valued $(d-1)$-dimensional random vector ${\bf X}$ 
on a probability space $(\Omega, \mathcal{A}, \mathbb{P})$ 
we say that a Markov kernel $K$ is a \emph{regular conditional distribution} of ${\bf X}$ given $Y$ if 
\begin{align*}
	K \big( Y(\omega), E \big)
	= \mathbb{E} \big( \mathds{1}_E \circ {\bf X} \,|\, Y \big) (\omega) 
\end{align*}
holds $\mathbb{P}$-almost surely for every $E\in \mathcal{B}(\mathbb{R}^{d-1})$. 
  It is well-known that for each random vector $({\bf X},Y)$ a regular conditional distribution 
$K$ of ${\bf X}$ given $Y$ always exists 
and is unique for $\mathbb{P}^Y$-a.e. $y\in\mathbb{R}$. 
  If $({\bf X},Y)$ has distribution function $H$ 
(in which case we will also write $({\bf X},Y) \sim H$ and let $\mu_H$ denote the corresponding probability measure on $\mathcal{B}(\mathbb{R}^d)$) we will let 
$K_H$ denote (a version of) the regular conditional distribution of ${\bf X}$ given $Y$ and simply refer to it as \emph{Markov kernel of $H$}. 
  If $C \in \mathcal{C}^d$ is a copula then we will consider the Markov kernel of $C$ automatically as mapping 
$K_C: \I \times \mathcal{B}(\I^{d-1}) \rightarrow \I$.
  Defining the $v$-section of a set $G\in\mathcal{B}(\I^{d})$ as 
$G_v := \lbrace {\bf u} \in \mathbb{R}^{d-1}: ({\bf u},v) \in G \rbrace$ 
the so-called disintegration theorem yields 
\begin{align}\label{eq:di}
	\int\limits_{\I} K_C(v,G_v) \; \mathrm{d}\lambda(v) 
	= \mu_C(G)
\end{align}
so, in particular, in case of $G =\times_{i=1}^{d-1} G_i$ with $G_i = \I$ for all $i \neq j$ we have 
\begin{align*}
	\int\limits\limits_{\I} K_C(v,G) \; \mathrm{d}\lambda(v) = \lambda(G_j).
\end{align*}
 For more background 
on conditional expectation and general disintegration we refer to \cite{kallenberg1997, klenke2007}.

  We call a copula $C \in \mathcal{C}^d$ \emph{completely dependent} (w.r.t. the last coordinate) 
if there exist $\lambda$-preserving transformations 
$h_1, \dots, h_{d-1}: \I \to \I$ 
(i.e., transformations fulfilling $\lambda(h_i^{-1}(F))=\lambda(F)$ for every $F \in \mathcal{B}(\I)$) such that
$$ K(y,E)
	:= \mathds{1}_E (h_1(y), \dots, h_{d-1}(y)) $$
is a Markov kernel of $C$. 
  Since the collection of all completely dependent copulas contains all shuffles of Min, 
it is dense in $(\CC^d,d_\infty)$ (also see \cite{mikusinski2010}).
  For more properties of complete dependence we refer to \cite{lancaster1963} as well as to \cite{trutschnig2015} 
  and the references therein.

  Markov kernels can be used to define metrics stronger than the standard \emph{uniform metric} $d_\infty$, defined by
\begin{align*}
	d_\infty(C_1,C_2) := \max\limits_{\uuu \in \I^d} |C_1(\uuu) - C_2(\uuu)|
\end{align*}
on $\mathcal{C}^d$. 
  It is well known that the metric space $(\mathcal{C}^d, d_\infty)$ is compact and that pointwise and uniform convergence 
of a sequence of copulas $(C_n)_{n\in \mathbb{N}}$ are equivalent 
(see \cite{durante2016}). 
  Following \cite{trutschnig2015} and defining
\begin{eqnarray*}
	D_1(C_1,C_2) 
	& := & \int\limits_{\I^{d-1}}\int\limits_{\I} \big| K_{C_1} (v,[{\bf 0},\uuu]) - K_{C_2} (v,[{\bf 0},\uuu]) \big| 
				 \; \mathrm{d}\lambda(v) \mathrm{d}\lambda^{d-1}(\uuu) 
	\\
	D_2(C_1,C_2) 
	& := & \int\limits_{\I^{d-1}}\int\limits_{\I} \big( K_{C_1} (v,[{\bf 0},\uuu]) - K_{C_2} (v,[{\bf 0},\uuu]) \big)^2
				 \; \mathrm{d}\lambda(v) \mathrm{d}\lambda^{d-1}(\uuu) 
	\\
	D_\infty(C_1,C_2) 
	& := & \sup_{\uuu \in \I^{d-1}} \int\limits_{\I} \big| K_{C_1} (v,[{\bf 0},\uuu]) - K_{C_2} (v,[{\bf 0},\uuu]) \big|  
	       \; \mathrm{d}\lambda(v)
\end{eqnarray*}
it can be shown that $D_1,D_2$ and $D_\infty$ are metrics generating the same topology on $\mathcal{C}^d$ and that 
the family of completely dependent copulas is closed with respect to these three metrics. 
  In the sequel we will mainly work with $D_1$ and refer to \cite{trutschnig2015} for more information on $D_2$ and $D_\infty$. 
  The metric space $(\mathcal{C}^d, D_1)$ is complete and separable but not compact.

  Viewing copulas in terms of their conditional distributions and considering weak convergence gives rise to what we refer to as weak conditional convergence in the sequel: 
  Consider a sequence of copulas 
$(C_n)_{n \in \N}$ and a copula $C$ and let 
$(K_{C_n})_{n \in \N}$ and $K_C$ be (versions of) the corresponding Markov kernels. 
  We will say that $(C_n)_{n \in \mathbb{N}}$ 
converges \emph{weakly conditional} (w.r.t. the last coordinate) 
to $C$ if and only if for $\lambda$-almost every $v \in \I$ 
we have that the sequence $(K_{C_n}(v,\cdot))_{n \in \mathbb{N}}$ of probability measures on 
$\mathcal{B}(\I^{d-1})$ converges weakly to the probability measure $K_{C}(v,\cdot)$.
  In the latter case we will write $C_n \xrightarrow{\text{wcc}} C$ (where 'wcc' stands for 'weak conditional convergence'). 
\\	
  According to Lemma $5$ in \cite{trutschnig2015}
weak conditional convergence of $(C_n)_{n \in \mathbb{N}}$ to $C$ implies convergence w.r.t. $D_1$
but not vice versa (see Example \ref{D1.Not.Wcc} below), 
and convergence w.r.t. $D_1$ implies convergence in $d_\infty$
but not vice versa.

%\smallskip
\begin{example}{} \label{D1.Not.Wcc} 
  For $d \geq 3$, $m \in \N$ and $k \in \{1,\dots,2^m\}$
define 
$J_{m,k}$ \linebreak $:=\big( (k-1)2^{-m}, k2^{-m}\big]$, set $n=2^m+k-2$ 
and consider the sequence of generalized EFGM copulas $(C_n)_{n \in \N}$ given by
$$
  C_n (\uuu,v)
	:= v \, \prod_{i=1}^{d-1} u_i + f_n(v) \, \prod_{i=1}^{d-1} u_i (1-u_i)
$$
where $f_{2^m+k-2}(v) := \int_{[0,v]} \mathds{1}_{J_{m,k}} (t) \, \mathrm{d} \lambda(t)$.
  Then, for every $n \in \N$, the identity 
$$
  K_{C_n}(v, [{\bf 0},\uuu])
	= \prod_{i=1}^{d-1} u_i + f^\prime_n(v) \, \prod_{i=1}^{d-1} u_i (1-u_i)
$$
holds for all $\uuu \in \I^{d-1}$ and almost all $v \in \I$.
Thus, the sequence $(K_{C_n}(v,\cdot))_{n \in \mathbb{N}}$ fails to converge weakly to 
$K_{\Pi}(v,\cdot)$ for $\lambda$-almost all $v \in \I$,
and it follows that $(C_n)_{n \in \N}$ does not converge weakly conditional to $\Pi$.
On the other hand, considering
\begin{eqnarray*}
  \lefteqn{\lim_{m \to \infty} \sup_{\uuu \in \I^{d-1}} 
					 \int\limits_{\I} 
					 \big| K_{C_{2^m+k-2}}(v,[{\bf 0},\uuu]) - K_{\Pi}(v,[{\bf 0},\uuu]) \big| 
					 \; \mathrm{d} \lambda(v)}
	\\
	& = & \lim_{m \to \infty} \sup_{\uuu \in \I^{d-1}} 
				\int\limits_{\I}
				\left| f^\prime_{2^m+k-2}(v) \, \prod_{i=1}^{d-1} u_i (1-u_i) \right| 
				\; \mathrm{d} \lambda(v)
	\\
	& = & \lim_{m \to \infty} \lambda (J_{m,k}) \; 
				\sup_{\uuu \in \I^{d-1}} \prod_{i=1}^{d-1} u_i (1-u_i)
	\\
	& = & 0
\end{eqnarray*}
so $ \lim_{n \to \infty} D_1 (C_n, \Pi) = 0 $.
For a counterexample in the case $d=2$ we refer to \cite{kasper2020}.
\end{example}{}

For any subset $J = \{j_{1},...,j_{|J|}\} \subseteq \{1, \dots, d\} $ with $2 \leq |J| \leq d$
such that $ j_{k} < j_{l} $ for all $ k,l \in \{1,...,|J|\} $ with $ k < l $ we let $C_J$ denote the \emph{marginal} 
copula of $C$ with respect to the coordinates in $J$. If $J$ only contains two indices $i,j$ then we will sometimes also 
write $C_{ij}$ instead of $C_{\{i,j\}}$ (no confusion will arise). 
Weak conditional convergence of a sequence of copulas transfers to marginal copulas:

\begin{theorem}{} \label{wcc.margins}
Suppose that $C,C_1,C_2,\ldots$ are $d$-dimensional copulas.
Then $C_n \xrightarrow{\text{wcc}} C$ implies
$$
  (C_n)_{J \cup \{d\}} \xrightarrow{\text{wcc}} C_{J \cup \{d\}}
$$
for every $J \subseteq \{1, \dots, d-1\} $ with $1 \leq |J| \leq d-1$.
\end{theorem}{}

\begin{proof}{}
  Consider $J \subseteq \{1, \dots, d-1\} $ with $1 \leq |J| \leq d-1$ and w.l.o.g. assume that $J = \{1, \dots, |J|\}$.
  Disintegration implies that for every copula $C \in \CC^d$ there exists some Markov kernel $K_C$ such that $C$ can be expressed as  
$$
  C(\uuu,v)
	= \int\limits_{[0,v]} K_C (t, [\nul,\uuu]) \; \mathrm{d}\lambda(t)
$$
for all $(\uuu,v) \in \I^{d-1} \times \I$ and some Markov kernel $K_{C_{J \cup \{d\}}}$ such that we have 
$$
  C_{J \cup \{d\}}(\sss,v)
	= \int\limits_{[0,v]} K_{C_{J \cup \{d\}}} (t, [\nul,\sss]) \; \mathrm{d}\lambda(t)
$$
for all $(\sss,v) \in \I^{|J|} \times \I$.
Thus
\begin{eqnarray} \label{Markov.kernel.Id.}
  K_{C_{J \cup \{d\}}} (t, [\nul,\sss]) 
	& = & K_C \big( t, [\nul,\sss] \times \I^{d-1-|J|} \big)
\end{eqnarray}
holds for all $\sss \in \I^{|J|}$ and $\lambda$-almost all $t \in \I$. \\
Suppose now that $C,C_1,C_2,\ldots$ are as in the theorem.  Since projections are continuous, the 
Continuous Mapping Theorem and the previous identity imply that for $\lambda$-almost every $v \in \I$ 
weak convergence of the sequence $(K_{C_n}(v,\cdot))_{n \in \mathbb{N}}$ to $K_{C}(v,\cdot)$ 
implies 
weak convergence of the sequence $(K_{(C_n)_{J \cup \{d\}}}(v,\cdot))_{n \in \mathbb{N}}$ to $K_{C_{J \cup \{d\}}}(v,\cdot)$, which 
proves the assertion.
\end{proof}{}

We complete this section with two additional notions of convergence considered, e.g., in \citet{spanhel2019}, 
the Kullback-Leibler divergence (distance) KL and the total variation metric TV, and describe their relationship 
with $D_1$ and $d_\infty$. Defining $TV$ on $\CC^d$ by 
$$
  TV(C_1,C_2) = \sup_{G \in \mathcal{B}(\I^d)} \vert \mu_{C_1} (G) - \mu_{C_2}(G) \vert,
$$
convergence with respect to $TV$ implies convergence with respect to $D_1$:
\begin{theorem}{} \label{Thm.D1.TV}
The inequalities
$$ 
  D_1(C_1,C_2) 
	\leq D_\infty(C_1,C_2) 
	\leq 2 \, TV(C_1,C_2)
$$
hold for all copulas $C_1,C_2 \in \CC^d$. 
In particular, convergence w.r.t. $TV$ implies convergence w.r.t. $D_1$ and $D_\infty$.
\end{theorem}{}
\begin{proof}
Fix $C_1,C_2 \in \CC^d$. 
For every $\mathbf{u} \in \I^{d-1}$ setting 
$$
  \Lambda_\mathbf{u}
	:= \{v \in \I \, : \, K_{C_1} (v,\mathbf{[0,u]}) > K_{C_2} (v,\mathbf{[0,u]})\} \in \mathcal{B}(\I)
$$
we get ($\Lambda_\mathbf{u}^c:=\I \setminus \Lambda_\mathbf{u}$) 
\begin{eqnarray*}
  0 
	& \leq & %\Phi_{C_1,C_2}(\mathbf{u}) 
	%\\
	%&  :=  & 
	\int_{\I} \vert K_{C_1}(v,\mathbf{[0,u]})- K_{C_2}(v,\mathbf{[0,u]}) \vert \; \mathrm{d} \lambda(v) 
	\\
  &   =  & \int_{\Lambda_\mathbf{u}} K_{C_1}(v,\mathbf{[0,u]})- K_{C_2}(v,\mathbf{[0,u]}) \; \mathrm{d} \lambda(v)  
	         + \int_{\Lambda_\mathbf{u}^c} K_{C_2}(v,\mathbf{[0,u]})- K_{C_1}(v,\mathbf{[0,u]}) \; \mathrm{d} \lambda(v) 
	\\
  &   =  & \mu_{C_1} (\Lambda_\mathbf{u} \times \mathbf{[0,u]}) - \mu_{C_2}(\Lambda_\mathbf{u} \times \mathbf{[0,u]}) 
	         + \mu_{C_2} (\Lambda_\mathbf{u}^c \times \mathbf{[0,u]}) - \mu_{C_1}(\Lambda_\mathbf{u}^c \times \mathbf{[0,u]}) 
	\\
  &   =  & \vert \mu_{C_1} (\Lambda_\mathbf{u} \times \mathbf{[0,u]}) - \mu_{C_2} (\Lambda_\mathbf{u} \times \mathbf{[0,u]}) \vert
           + \vert \mu_{C_1} (\Lambda_\mathbf{u}^c \times \mathbf{[0,u]}) - \mu_{C_2} (\Lambda_\mathbf{u}^c \times \mathbf{[0,u]}) \vert 
	\\
  & \leq &\, 2 TV(C_1,C_2)
\end{eqnarray*}
from which the desired inequalities follow immediately. 
The first inequality has already been proved in \cite[Lemma 3]{trutschnig2015}.
%Since the second assertion is a direct consequence the proof if complete. 
%and using the notation from $D_1$-multi paper [15] 
\end{proof}

It is well-known that KL divergence 
(which is not a metric and only well-defined for absolutely continuous copulas whose density is positive $\lambda^d$-almost everywhere) is stronger than TV (see the generalized Pinsker inequality in, e.h., \cite{reid2009}).
Altogether we have the following interrelation, where $a \Longrightarrow b$ indicates the convergence with respect to $a$
 implies convergence with respect to $b$ (and the first implication is restricted to those 
 copulas for which KL divergence is well-defined): 
$$
  KL
	\Longrightarrow
	TV
	\Longrightarrow
	D_1 
	\Longleftrightarrow
	D_\infty
	\Longrightarrow
	d_\infty
$$

%%%%%%%%%%%%%%%%%%%%%%%%%%%%%%%%%%%%%%%%%%%%%%%%%%%%%%%%%%%%%%%%%%%%%%%%%%%%%%%%%%%%%%%%%%%%%%%%%%%%%%%%%%%%%%%%%%%%%%%%%%%%%%%%%%%%%%%%%%%%%%%%%%%%%%%%%%%%%%%%%%%%%%%%%%%%%%%%%%%%%%%%%%%%%%%%%%%%%%%%%%%%%%%%%%%%%%%%%%%%%%%%%%%%%%%%%%%%%%%%%%%%%%%%%%%%%%%%%%%%%%%%%%%%%%%%%%%%%%%%%%%%%%%%%%%%%%%%%%%%%%%%%%%%%%%%%%%%%%%%%%%%%%%%%%%%%%%%%%%%%%%%%%%%%%%%%%%%%%%%%%%%%%%%%%%%%%%%%%%%%%%%%%%%%%%%%%%%%%%%%%%%%%%%%%%%%%%%%%%%%%%%%%%%%%%%%%%%%%%%%%%%%%%%%%%%%%%%%%%%%%%%%%%%%%%%%%%%%%%%%%%%%%%%%%%%%%%%%%%%%%%%%%%%%%%%%%%%%%%%%%%%%%%%%%%%%%%%%%%%%%%%%%%%%%%%%%%%%%%%%%%%%%%%%%%%%%%%%%%%%%%%%%%%%%%%%%%%%%%%%%%%%%%%%%%%%%%%%%%%%%%%%%%%%%%%%%%%%%%%%%%%%%%%%%%%%%%%%%%%%%%%%%%%%%%%%%%%%%%%%%%%%%%%%%%%%%%%%%%%%%%%
%\newpage

\section{Simplified copulas} 
\label{Sect.Simplified}

In this section we introduce three-dimensional so-called simplified copulas, i.e., copulas for which the conditional 
copulas do not depend on the conditioning variable. 
The enormous importance of this type of copulas is underlined by the fact that every copula 
can be approximated arbitrarily well with respect to $d_\infty$ by simplified copulas (see Corollary \ref{S.Dense}). 	
On the other hand, we will show that simplified pair-copula constructions may fail to approximate a given 
dependence structure w.r.t. $d_\infty$ reasonably well (see Example \ref{SVC.PInd.2}).
Additionally, we will see that the afore-mentioned denseness gets lost entirely when 
finer topologies or stronger metrics are considered, and prove that for $D_1$ (Theorem \ref{thm:nowhere.dense}), 
for the total variation metric TV (Theorem \ref{Non.Dense.TV}), and the Kullback-Leibler (KL) divergence 
(Theorem \ref{Non.Dense.KL}) the family is even nowhere dense.

With very few exceptions, in literature pair-copula constructions are introduced by working with copula densities, i.e., all
copulas are assumed to be absolutely continuous. Ensuring that no key idea of the underlying concept is left out 
and aiming at a setting as general as possible we deviate from this approach and work with Markov kernels instead.

In this and the subsequent three sections all conditioning will be done with respect to the last coordinate, notice that this does not impose any restriction
(as can be seen from 
Theorem \ref{Non.Dense.TV},
Theorem \ref{Non.Dense.KL},
Remark \ref{rem.cond.} 
and Section \ref{Arb.Dim.}).
%Theorem \ref{D.Emp.Subset.S}, 
%Corollary \ref{D.S.Dense}, 
%Theorem \ref{D.PVC.Gen.Dist.} 
%and Corollary \ref{discont.everywhere.delta})}.

  According to disintegration for every copula $C \in \CC^3$ there exists some Markov kernel 
  $K_C$ such that $C$ can be expressed as  
$$
  C(\uuu,v)
	= \int\limits_{[0,v]} K_C (t, [\nul,\uuu]) \;\mathrm{d}\lambda(t)
$$
for all $(\uuu,v) \in \I^2 \times \I$.
Since $K_C$ is a Markov kernel, for every $\uuu \in \I^2$ the mapping $ t \mapsto K_C (t, [\nul,\uuu]) $ is 
measurable and for almost every $t \in \I$ the mapping $ \uuu \mapsto K_C (t, [\nul,\uuu]) $ is a bivariate distribution function with (\emph{conditional}) univariate marginal distribution functions 
$F_{1|3}(\cdot|t)$ and $F_{2|3}(\cdot|t)$ (conditional on $t$).
Sklar's Theorem implies that for almost every $t \in \I$ there exists some (\emph{conditional}) 
bivariate copula $C_{12;3}^t$ (conditional on $t$) satisfying
$$
  K_C (t, [\nul,\uuu])
	= C_{12;3}^t \big( F_{1|3} (u_1|t), F_{2|3} (u_2|t) \big) 
$$
for all $\uuu \in \I^2$ such that the identity
\begin{equation} \label{Eq.Cop.1}
  C(\uuu,v)
	= \int\limits_{[0,v]} C_{12;3}^t \big( F_{1|3}(u_1|t), F_{2|3}(u_2|t) \big) \;\mathrm{d}\lambda(t)
\end{equation}
holds for all $(\uuu,v) \in \I^2 \times \I$.

%\smallskip
\begin{remark}{} \leavevmode
\begin{onelist}
\item
Since the (conditional) univariate marginal distribution functions satisfy 
$F_{1|3}(1|t)=1=F_{2|3}(1|t)$ for every $t \in \mathbb{I}$ the bivariate marginal copulas $C_{13}$ and $C_{23}$ of $C$
satisfy
$$
  C_{13} (u_1,v)
	= \int\limits_{[0,v]} C_{12;3}^t \big( F_{1|3}(u_1|t), F_{2|3}(1|t) \big) \;\mathrm{d}\lambda(t)
	= \int\limits_{[0,v]} F_{1|3}(u_1|t) \;\mathrm{d}\lambda(t)
$$
as well as $C_{23} (u_2,v) = \int_{[0,v]} F_{2|3}(u_2|t) \,\mathrm{d}\lambda(t)$
for all $(\uuu,v) \in \I^2 \times \I$ and their corresponding Markov kernels fulfill 
\begin{eqnarray} \label{Eq.Cop.Margins.1}
  K_{C} (t, [0,u_1] \times \I)
	& = K_{C_{13}} (t, [0,u_1])
	& = F_{1|3}(u_1|t)
	\\
	K_{C} (t, \I \times [0,u_2])
	& = K_{C_{23}} (t, [0,u_2])
	& = F_{2|3}(u_2|t)
\end{eqnarray}
for all $\uuu \in \I^2$ and $\lambda$-almost all $t \in \I$ (compare with Equation (\ref{Markov.kernel.Id.})).
\item 
 Notice that we choose this different notation for the (conditional) univariate distribution functions on purpose
 since this facilitates comprehending what follows.    

\item 
  For the copulas corresponding to the conditional bivariate distribution functions $K_C (t, .)$ we write 
  $C_{12;3}^t$ instead of $C_{12|3}^t$ and hence adopt the notation used in the literature (see, e.g., \cite{spanhel2019}).
\end{onelist}
\end{remark}{} 
%\smallskip

\noindent The following two observations concerning Equation (\ref{Eq.Cop.1}) are key:
\begin{onelist}
\item[(O1) ] 
the (conditional) bivariate copulas $C_{12;3}^t$ may depend on $t$;

\item[(O2) ]
since the (conditional) univariate marginal distribution functions $F_{1|3}(.|t)$ and $F_{2|3}(.|t)$ 
may fail to be continuous the (conditional) bivariate copulas $C_{12;3}^t$ are not unique in general.
\end{onelist} 
To the best of the authors' knowledge, the second observation has not yet been addressed in the literature which is somehow not surprising considering the fact that pair-copula constructions are usually focused on absolutely continuous copulas. 

In the sequel we will study copulas $C$ for which (O1) is not true, i.e., copulas for which the (conditional) copulas 
$C_{12;3}^t$ do not depend on $t$.
We will refer to a copula $C \in \CC^3$ as \emph{generalized simplified} (with respect to the third coordinate) 
if there exists some bivariate copula $A \in \CC^2$ such that the identity
\begin{equation} \label{Eq.GS.1}
  C(\uuu,v) 
	= \int\limits_{[0,v]} A \, \big( F_{1|3}(u_1|t), F_{2|3}(u_2|t) \big) \;\mathrm{d}\lambda(t)
\end{equation}
holds for all $(\uuu,v) \in \I^2 \times \I$. In the sequel $\CC^3_{\rm GS}$ will denote the family 
of all three-dimensional generalized simplified copulas.

The following first results (Theorem \ref{ComplDep.Subset.GS} and Corollary \ref{GS.Dense}) imply that the family of generalized simplified copulas is very flexible.

%\smallskip
\begin{theorem}{} \label{ComplDep.Subset.GS}
  Every completely dependent three-dimensional copula is genera\-lized simplified.
\end{theorem}{}

\begin{proof}{}
  Let $C \in \CC^3$ be a completely dependent copula, i.e., assume that there exist  
  $\lambda$--preserving functions $h_1,h_2: \I \to \I$ such that
$K_C (v, E):= 1_{E} (h_1(v), h_2(v))$ is a Markov kernel of $C$. 
Considering  
$$
  K_C (v, [\nul,\uuu])
	= 1_{[0,u_1] \times [0,u_2]} (h_1(v), h_2(v))
  = 1_{[h_1(v),1]} (u_1) \, 1_{[h_2(v),1]} (u_2)
$$
as well as
$F_{1|3}(u_1|v) = 1_{[h_1(v),1]} (u_1), F_{2|3}(u_2|v) = 1_{[h_2(v),1]} (u_2) \in \{0,1\}$ it follows that
for every copula $A \in \CC^2$ the identity
$$
  C(\uuu,v)
	= \int\limits_{[0,v]} A \big( F_{1|3}(u_1|t), F_{2|3}(u_2|t) \big) \;\mathrm{d}\lambda(t)
$$
holds for all $(\uuu,v) \in \I^2 \times \I$. This yields $C \in \CC^3_{\rm GS}$. 
\end{proof}{} 
%\smallskip

Note that completely dependent copulas are generalized simplified in the broadest sense since  
Equation (\ref{Eq.GS.1}) does not only hold for one or some copulas, it holds for every $A \in \CC^2$.

Since the collection of all completely dependent copulas is dense in $(\CC^3, d_{\infty})$ 
Theorem \ref{ComplDep.Subset.GS} has the following consequence:

%\smallskip
\begin{corollary}{} \label{GS.Dense}
  The collection of all generalized simplified copulas is dense in $(\CC^3, d_{\infty})$.
\end{corollary}{} 
%\smallskip

Returning to observation (O2) in what follows we will mainly restrict ourselves to the family of 
copulas $C \in \CC^3$ for which almost all (conditional) univariate marginal distribution functions 
$F_{1|3}(.|t)$ and $F_{2|3}(.|t)$ are continuous and let $\CC^3_{\rm c}$ denote the family of all these copulas.
According to Sklar's theorem, for every copula $C \in \CC^3_{\rm c}$ the (conditional) bivariate copulas 
$C_{12;3}^t$ are unique for almost all $t \in \I$. Obviously the family of all absolutely continuous copulas 
$\CC^3_{\rm ac}$ is a subset of $\CC^3_{\rm c}$, so for absolutely continuous copulas the conditional copulas are unique. 

We will let $\CC^3_{\rm S} := \CC^3_{\rm GS} \cap \CC^3_{\rm c}$
denote the collection of all \emph{simplified} copulas, i.e., the class of all three-dimensional copulas $C$ 
which are generalized simplified \emph{and} have continuous (conditional) univariate marginal distribution functions 
$F_{1|3}(.|t)$ and $F_{2|3}(.|t)$. In this case the copula $A \in \CC^2$ in Equation \ref{Eq.GS.1} is unique 
and equals $C_{12;3}^t$ for almost all $t \in \I$.
\\
Before proceeding we illustrate the above \emph{simplifying assumption} in terms of 
the (Fr{\'e}chet) class of all three-dimensional copulas $C$ fulfilling that coordinates $1 \& 3$ as well as 
$2 \& 3$ are independent:

%\clearpage
\begin{example}{} \label{SVC.PInd}
(Class $\mathcal{F}^3_\Pi$ of three-dimensional copulas $C$ satisfying $C_{13} = \Pi = C_{23}$) 

\noindent For $C \in \mathcal{F}^3_\Pi$ we have  
$F_{1|3}(u_1|t)=u_1$ and $F_{2|3}(u_2|t)=u_2$ for all $\uuu \in \I^2$ and almost all $t \in \I$ implying 
$\mathcal{F}^3_\Pi \subseteq \mathcal{C}^3_c$. 
If $D \in \mathcal{F}^3_\Pi$ is simplified then there exists some unique bivariate copula $A \in \CC^2$ such that 
$$
  D (\uuu,v) 
	= \int\limits_{[0,v]} A \, \big( F_{1|3}(u_1|t), F_{2|3}(u_2|t) \big) \;\mathrm{d}\lambda(t)
	= \int\limits_{[0,v]} A (u_1, u_2) \;\mathrm{d}\lambda(t)
	= A (\uuu) \, v
$$
holds for all $(\uuu,v) \in \I^2 \times \I$. 

\begin{onelist}
\item 
  The independence copula $\Pi \in \mathcal{F}^3_\Pi$ satisfies
$$
  \Pi (\uuu,v) 
	= \int\limits_{[0,v]} \Pi \, \big( F_{1|3}(u_1|t), F_{2|3}(u_2|t) \big) \;\mathrm{d}\lambda(t)
$$
for all $(\uuu,v) \in \I^2 \times \I$.
Thus, $\Pi$ is simplified, obviously $\Pi (\uuu,v)= \Pi(\uuu) \, v$ holds for all $(\uuu,v) \in \I^2 \times \I$.

\item
  The EFGM copula $C^{\rm EFGM} \in \mathcal{F}^3_\Pi$, given by 
$$
  C^{\rm EFGM}(\uuu,v)
	:= \Pi(\uuu,v) + u_1 (1-u_1) \, u_2 (1-u_2) \, v (1-v) 
$$
satisfies 
$$
  C^{\rm EFGM}(\uuu,v)
	= \int\limits_{[0,v]} (C^{\rm EFGM})_{12;3}^t \, \big( F_{1|3}(u_1|t), F_{2|3}(u_2|t) \big) \;\mathrm{d}\lambda(t)
$$
for all $(\uuu,v) \in \I^2 \times \I$,
where 
$$(C^{\rm EFGM})_{12;3}^t(\uuu) = u_1 u_2 + (1-2\,t) \, u_1 (1-u_1) \, u_2 (1-u_2)$$
for all $\uuu \in \I^2$ and almost all $t \in \I$.
Thus, $C^{\rm EFGM}$ is non-simplified.

\item
  The copula $C^{\rm Cube} \in \mathcal{F}^3_\Pi$ which distributes mass uniformly within the four cubes
\begin{eqnarray*}
  \big(0,\tfrac{1}{2}\big) \times \big(0,\tfrac{1}{2}\big) \times \big(0,\tfrac{1}{2}\big)
	& \quad & \big(\tfrac{1}{2},1\big) \times \big(\tfrac{1}{2},1\big) \times \big(0,\tfrac{1}{2}\big)
	\\
	\big(0,\tfrac{1}{2}\big) \times \big(\tfrac{1}{2},1\big) \times \big(\tfrac{1}{2},1\big)
	& \quad & \big(\tfrac{1}{2},1\big) \times \big(0,\tfrac{1}{2}\big) \times \big(\tfrac{1}{2},1\big)
\end{eqnarray*}
and has no mass outside these cubes satisfies 
$$
  C^{\rm Cube}(\uuu,v)
	= \int\limits_{[0,v]} (C^{\rm Cube})_{12;3}^t \, \big( F_{1|3}(u_1|t), F_{2|3}(u_2|t) \big) \;\mathrm{d}\lambda(t)
$$
for all $(\uuu,v) \in \I^2 \times \I$,
where 
$(C^{\rm Cube})_{12;3}^t = A^1$ for almost all $t \in \big(0,\tfrac{1}{2}\big)$ and 
$(C^{\rm Cube})_{12;3}^t = A^2$ for almost all $t \in \big(\tfrac{1}{2},1\big)$, 
and the copulas $A^1$ and $A^2$ are checkerboard copulas (see \cite{durante2016} for a general definition)
whose density is depicted in Figure \ref{fig.cube.}.
%\begin{hylist}
%\item 
%$A^1: \I^2 \to \I$ is defined as the shuffle of $\Pi$ for the shuffling structure   
%$ \{[\aaa_i,\bbb_i]\}_{i\in\{1,2\}} $   
%with 
%$$\begin{array}{ccccccc}
% \aaa_1 &=& (\phantom{00}0,\phantom{00}0)  &&  \bbb_1 &=& (1/2,1/2)  \\
% \aaa_2 &=& (1/2,1/2)  &&  \bbb_2 &=& (\phantom{00}1,\phantom{00}1)  
%\end{array}$$
%\item 
%$A^2: \I^2 \to \I$ is defined as the shuffle of $\Pi$ for the shuffling structure   
%$ \{[\aaa_i,\bbb_i]\}_{i\in\{1,2\}} $   
%with 
%$$\begin{array}{ccccccc}
% \aaa_1 &=& (\phantom{00}0,1/2)  &&  \bbb_1 &=& (1/2,\phantom{00}1)  \\
% \aaa_2 &=& (1/2,\phantom{00}0)  &&  \bbb_2 &=& (\phantom{00}1,1/2)  
%\end{array}$$
%\end{hylist}
As a direct consequence $C^{\rm Cube}$ is non-simplified. 
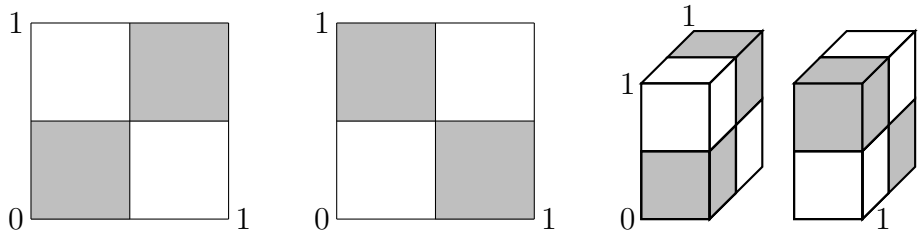
\begin{figure}[h]
\centering
\begin{tikzpicture}[xscale=0.65,yscale=0.65] 
\node (AA) at (-0.3,4.) {$1$};
\node (BA) at (-0.3,0.) {$0$};
\node (CA) at (4.3,0.) {$1$};
\fill [lightgray] (0,0) -- (0,2) -- (2,2) -- (2,0) -- (0,0);            
\fill [lightgray] (2,2) -- (2,4) -- (4,4) -- (4,2) -- (2,2);    
\draw [thin]  (0,0) -- (0,4);     
\draw [thin]  (2,0) -- (2,4);     
\draw [thin]  (4,0) -- (4,4);    
\draw [thin]  (0,0) -- (4,0);     
\draw [thin]  (0,2) -- (4,2);     
\draw [thin]  (0,4) -- (4,4);    
\end{tikzpicture}
\quad
\begin{tikzpicture}[xscale=0.65,yscale=0.65] 
\node (AB) at (9.7,4.) {$1$};
\node (BB) at (9.7,0.) {$0$};
\node (CB) at (14.3,0.) {$1$};          
\fill [lightgray] (10,2) -- (10,4) -- (12,4) -- (12,2) -- (10,2);            
\fill [lightgray] (12,0) -- (12,2) -- (14,2) -- (14,0) -- (12,0);  
\draw [thin]  (10,0) -- (10,4); 
\draw [thin]  (12,0) -- (12,4); 
\draw [thin]  (14,0) -- (14,4); 
\draw [thin]  (10,0) -- (14,0); 
\draw [thin]  (10,2) -- (14,2); 
\draw [thin]  (10,4) -- (14,4);            
\end{tikzpicture}
\quad
\begin{tikzpicture}[xscale=0.9,yscale=0.9] 
\node (A) at (-0.3,0.) {$1$};
\node (B) at (-1.2,-1.) {$1$};
\node (C) at (-1.2,-3.) {$0$};
\pgfmathsetmacro{\cube}{1}
\draw[thick,fill=lightgray] (0,-2,0) -- ++(-\cube,0,0) -- ++(0,-\cube,0) -- ++(\cube,0,0) -- cycle;
\draw[thick,fill=lightgray] (0,-2,0) -- ++(0,0,-\cube) -- ++(0,0-\cube,0) -- ++(0,0,0+\cube) -- cycle;
\draw[thick] (0,-1,0) -- ++(-\cube,0,0) -- ++(0,-\cube,0) -- ++(\cube,0,0) -- cycle;
\draw[thick] (0,-2,0) -- ++(-\cube,0,0) -- ++(0,-\cube,0) -- ++(\cube,0,0) -- cycle;
%\draw[thick,fill=lightgray] (0,0,-1) -- ++(0,0,-\cube) -- ++(0,0-\cube,0) -- ++(0,0,0+\cube) -- cycle;
%\draw[thick,fill=lightgray] (0,0,-2) -- ++(0,0,-\cube) -- ++(0,0-\cube,0) -- ++(0,0,0+\cube) -- cycle;
\draw[thick] (0,-1,0) -- ++(0,0,-\cube) -- ++(0,0-\cube,0) -- ++(0,0,0+\cube) -- cycle;
\draw[thick,fill=lightgray] (0,-1,-1) -- ++(0,0,-\cube) -- ++(0,0-\cube,0) -- ++(0,0,0+\cube) -- cycle;
%\draw[thick,fill=lightgray] (0,-1,-2) -- ++(0,0,-\cube) -- ++(0,0-\cube,0) -- ++(0,0,0+\cube) -- cycle;
\draw[thick] (0,-2,0) -- ++(0,0,-\cube) -- ++(0,0-\cube,0) -- ++(0,0,0+\cube) -- cycle;
\draw[thick] (0,-2,-1) -- ++(0,0,-\cube) -- ++(0,0-\cube,0) -- ++(0,0,0+\cube) -- cycle;
%\draw[thick,fill=lightgray] (0,-2,-2) -- ++(0,0,-\cube) -- ++(0,0-\cube,0) -- ++(0,0,0+\cube) -- cycle;
%\draw[thick,fill=lightgray] (0,0,0) -- ++(-\cube,0,0) -- ++(0,0,-\cube) -- ++(\cube,0,0) -- cycle;
\draw[thick] (0,-1,0) -- ++(-\cube,0,0) -- ++(0,0,-\cube) -- ++(\cube,0,0) -- cycle;
\draw[thick,fill=lightgray] (0,-1,-1) -- ++(-\cube,0,0) -- ++(0,0,-\cube) -- ++(\cube,0,0) -- cycle;
\end{tikzpicture}
\;
\begin{tikzpicture}[xscale=0.9,yscale=0.9] 
\node (D) at (0.3,-3.) {$1$};
\pgfmathsetmacro{\cube}{1}
\draw[thick] (0,-2,0) -- ++(-\cube,0,0) -- ++(0,-\cube,0) -- ++(\cube,0,0) -- cycle;
\draw[thick] (0,-2,0) -- ++(0,0,-\cube) -- ++(0,0-\cube,0) -- ++(0,0,0+\cube) -- cycle;
\draw[thick,fill=lightgray] (0,-1,0) -- ++(-\cube,0,0) -- ++(0,-\cube,0) -- ++(\cube,0,0) -- cycle;
\draw[thick] (0,-2,0) -- ++(-\cube,0,0) -- ++(0,-\cube,0) -- ++(\cube,0,0) -- cycle;
%\draw[thick,fill=lightgray] (0,0,-1) -- ++(0,0,-\cube) -- ++(0,0-\cube,0) -- ++(0,0,0+\cube) -- cycle;
%\draw[thick,fill=lightgray] (0,0,-2) -- ++(0,0,-\cube) -- ++(0,0-\cube,0) -- ++(0,0,0+\cube) -- cycle;
\draw[thick,fill=lightgray] (0,-1,0) -- ++(0,0,-\cube) -- ++(0,0-\cube,0) -- ++(0,0,0+\cube) -- cycle;
\draw[thick] (0,-1,-1) -- ++(0,0,-\cube) -- ++(0,0-\cube,0) -- ++(0,0,0+\cube) -- cycle;
%\draw[thick,fill=lightgray] (0,-1,-2) -- ++(0,0,-\cube) -- ++(0,0-\cube,0) -- ++(0,0,0+\cube) -- cycle;
\draw[thick] (0,-2,0) -- ++(0,0,-\cube) -- ++(0,0-\cube,0) -- ++(0,0,0+\cube) -- cycle;
\draw[thick,fill=lightgray] (0,-2,-1) -- ++(0,0,-\cube) -- ++(0,0-\cube,0) -- ++(0,0,0+\cube) -- cycle;
%\draw[thick,fill=lightgray] (0,-2,-2) -- ++(0,0,-\cube) -- ++(0,0-\cube,0) -- ++(0,0,0+\cube) -- cycle;
%\draw[thick,fill=lightgray] (0,0,0) -- ++(-\cube,0,0) -- ++(0,0,-\cube) -- ++(\cube,0,0) -- cycle;
\draw[thick,fill=lightgray] (0,-1,0) -- ++(-\cube,0,0) -- ++(0,0,-\cube) -- ++(\cube,0,0) -- cycle;
\draw[thick] (0,-1,-1) -- ++(-\cube,0,0) -- ++(0,0,-\cube) -- ++(\cube,0,0) -- cycle;
\end{tikzpicture}
\caption{Mass distribution of the copulas $A^1$, $A^2$ and $C^{\rm Cube}$ from Example \ref{SVC.PInd}.}
\label{fig.cube.}
\vspace{-0.3cm}
\end{figure}
\end{onelist}
\end{example}{}
%\smallskip

In contrast to the afore-mentioned class, some copula families only contain simplified copulas:\pagebreak

%\smallskip
\begin{example}{} \cite{hobaek2010,stoeber2013} 
\begin{onelist}
\item
  All three-dimensional Gaussian and Student $t$-copulas are simplified.

\item 
  The only three-dimensional Archimedean copulas that are simplified are those of Clayton type.
\end{onelist}
\end{example}{}
%\smallskip

We now focus on empirical copulas, show that they are simplified and then conclude that $\CC^3_{\rm S}$ is dense in 
$(\CC^3, d_{\infty})$ (Corollary \ref{S.Dense}).

Consider a random vector $(\XXX,Y)$ with continuous univariate marginals and suppose that 
$(\mathbf{X}_1,Y_1),\ldots, (\mathbf{X}_n,Y_n)$ is a sample from $(\mathbf{X},Y)$. 
  Since the univariate marginals are continuous w.l.o.g. we can assume that there are no ties.
  Let $\hat{C}_n$ denote the empirical copula (by which we mean the unique copula determined by trilinear interpolation of the empirical subcopula).
%\begin{figure}[h]
%\centering
%\includegraphics[height=60mm,width=120mm]{ecop_ex.pdf}
%\caption{Sample of size $n=8$ (left panel). Pseudo observations and mass distribution of the empirical copula (right panel).}
%\label{fig.emp.}
%\vspace{-0.3cm}
%\end{figure}
  Then there exist two permutations $\sigma_1,\sigma_2$ of $\{1,\ldots,n\}$ such that the density $\hat{c}_n$ of $\hat{C}_n$ is given by (uniform distribution on $n$ cubes of volume $\frac{1}{n^3}$)
$$
  \hat{c}_n (u_1,u_2,v) 
	= n^2 \sum_{i=1}^n \mathds{1}_{I_i^1}(u_1) \mathds{1}_{I_i^2}(u_2) \mathds{1}_{V_i}(v)
$$
where 
$I^1_i=(\frac{\sigma_1(i)-1}{n},\frac{\sigma_1(i)}{n}]$, 
$I^2_i=(\frac{\sigma_2(i)-1}{n},\frac{\sigma_2(i)}{n}]$ and
$V_i=(\frac{i-1}{n},\frac{i}{n}]$, 
so the Markov kernel of $\hat{C}_n$ fulfills 
\begin{equation}\label{eq:empcopkern}
  K_{\hat{C}_n}(v,[0,u_1] \times [0,u_2]) 
	= n^2 \sum_{i=1}^n
    \left(\int_{[0,u_1]}\mathds{1}_{I_i^1}(t) \; \mathrm{d} \lambda(t) \int_{[0,u_2]}\mathds{1}_{I_i^2}(s) \; \mathrm{d} \lambda(s) \right)
	  \mathds{1}_{V_i}(v). 
\end{equation}

%\smallskip
\begin{theorem}{} \label{Emp.Subset.S}
  Every three-dimensional empirical copula is simplified.
\end{theorem}{}
\begin{proof}{} 
Considering that the (conditional) univariate marginal distribution functions 
$(\hat{F}_n)_{1 \vert 3} (\cdot \vert v)$, $(\hat{F}_n)_{2 \vert 3} (\cdot \vert v)$ of $\hat{C}_n$ 
are continuous and given by
\begin{eqnarray*}
  (\hat{F}_n)_{1 \vert 3} (u_1 \vert v) 
	& = & n \sum_{i=1}^n \left( \int_{[0,u_1]}\mathds{1}_{I_i^1}(t) \; \mathrm{d} \lambda(t) \right) \mathds{1}_{V_i}(v) 
	\\
	(\hat{F}_n)_{2 \vert 3} (u_2 \vert v) 
	& = & n \sum_{i=1}^n \left( \int_{[0,u_2]}\mathds{1}_{I_i^2}(t) \; \mathrm{d} \lambda(t) \right) \mathds{1}_{V_i}(v) 
\end{eqnarray*}
using Equation (\ref{eq:empcopkern}) it follows immediately that $K_{\hat{C}_n}(v,[0,u_1] \times [0,u_2])$ can be 
expressed as
$$
  K_{\hat{C}_n}(v,[0,u_1] \times [0,u_2]) 
	= \Pi \big( (\hat{F}_n)_{1 \vert 3} (u_1 \vert v), (\hat{F}_n)_{2 \vert 3} (u_2 \vert v) \big)
$$
from which it follows that $\hat{C}_n$ is simplified.
\end{proof}{}
%\smallskip

Since the collection of all empirical copulas is dense in $(\CC^3, d_{\infty})$ (see \cite[Proposition 3.2]{durantefs2010}),  
Theorem \ref{Emp.Subset.S} has the following consequence
(for a stronger and more general result see Corollary \ref{D.S.Dense}):

%\smallskip
\begin{corollary}{} \label{S.Dense}
  The collection of all simplified copulas is dense in $(\CC^3, d_{\infty})$. 
\end{corollary}{}
%\smallskip

Although every copula can be approximated arbitrarily well by simplified copulas a reasonable approximation from the same 
Fr{\'e}chet class might not be possible as the following example illustrates: 
%\smallskip

\begin{example}{} \label{SVC.PInd.2}
  (Class $\mathcal{F}^3_\Pi$, cont.) \\
For the non--simplified copula $C^{\rm Cube} \in \mathcal{F}^3_\Pi$ introduced in Example \ref{SVC.PInd} there exists 
some $\varepsilon > 0$ such that for every simplified copula $D \in \mathcal{F}^3_\Pi$ we have
$$
  d_\infty \big( C^{\rm Cube},D \big)	> \varepsilon,
$$
which can be shown as follows: Recall that every simplified copula $D$ from this class fulfills 
$D (\uuu,v) = A (\uuu) \, v$
for all $(\uuu,v) \in \I^2 \times \I$,
where $A$ is some bivariate copula.
  Furthermore recall that $C^{\rm Cube}$ fulfills  
$$
  C^{\rm Cube}(\uuu,v)
	%= \int\limits_{[0,v]} (C^{\rm Cube})_{12;3}^t \, \big( F_{1|3}(u_1|t), F_{2|3}(u_2|t) \big) \;\mathrm{d}\lambda(t)
	= \int\limits_{[0,v]} (C^{\rm Cube})_{12;3}^t \, (\uuu) \;\mathrm{d}\lambda(t)
$$
for all $(\uuu,v) \in \I^2 \times \I$,
where 
$(C^{\rm Cube})_{12;3}^t = A^1$ for almost all $t \in \big(0,\tfrac{1}{2}\big)$ and 
$(C^{\rm Cube})_{12;3}^t = A^2$ for almost all $t \in \big(\tfrac{1}{2},1\big)$,
and $A^1$ and $A^2$ are bivariate copulas with $A^1 \neq A^2$ (see Example \ref{SVC.PInd}).
  Thus, 
$$
  C^{\rm Cube}(\uuu,v) 
	= 
	\begin{cases}
	  A^1 (\uuu) \, v	 															
		& v \in \big[ 0, \tfrac{1}{2} \big]
		\\
		A^1 (\uuu) \, \tfrac{1}{2} + A^2 (\uuu) \, \big( v - \tfrac{1}{2}	\big)												
		& v \in \big( \tfrac{1}{2}, 1 \big]
	\end{cases}
$$
and hence
\begin{eqnarray*}
  \lefteqn{\big| C^{\rm Cube}(\uuu,v) - D(\uuu,v) \big|}
	\\
	& = & 
	\begin{cases}
	  \big| A^1 (\uuu) - A (\uuu) \big|	\, v								
		& v \in \big[ 0, \tfrac{1}{2} \big]
		\\
		\big| \big[ A^1 (\uuu) - A(\uuu) \big] \, \tfrac{1}{2} + \big[ A^2 (\uuu) - A(\uuu) \big] \, \big( v - \tfrac{1}{2}	\big)	\big|	
		& v \in \big( \tfrac{1}{2}, 1 \big]
	\end{cases}
\end{eqnarray*}
for all $(\uuu,v) \in \I^2 \times \I$.
  If $A=A^1$, then 
$$d_\infty \big( C^{\rm Cube}, D \big)
	\geq \big| C^{\rm Cube} \big( \tfrac{\bf 1}{\bf 2},\tfrac{3}{4} \big) - D \big( \tfrac{\bf 1}{\bf 2},\tfrac{3}{4} \big) \big|
	 	=  \tfrac{1}{4} \; \big| A^2 \big( \tfrac{\bf 1}{\bf 2} \big) - A^1 \big( \tfrac{\bf 1}{\bf 2} \big) \big|
		=  \tfrac{1}{4} \; \big| 0 - \tfrac{1}{2} \big|
		=  \tfrac{1}{8}.$$
  If $A \neq A^1$ then there exists some $\uuu^\ast \in \I^2$ and some $\varepsilon >0$ with
$|A(\uuu^\ast) - A^1(\uuu^\ast)| > 4 \, \varepsilon$ and hence
$$
  d_\infty \big( C^{\rm Cube}, D \big)
	\geq \big| C^{\rm Cube} \big( \uuu^\ast, \tfrac{1}{4} \big) - D \big( \uuu^\ast ,\tfrac{1}{4} \big) \big|
	  =  \tfrac{1}{4} \; \big| A^1 (\uuu^\ast) - A (\uuu^\ast) \big|
		>  \varepsilon
$$
  Thus $C^{\rm Cube}$ can not be approximated arbitrarily well by a simplified copula $D$ from the class $\mathcal{F}^3_\Pi$.
\end{example}{}

We now focus on the afore-mentioned stronger metrics or finer topologies on $\CC$ (or important subclasses). To simplify notation we will write $\CC^3_{{\rm ac}, >0}$ for the collection of all 
absolutely continuous copulas with positive density.
\begin{theorem}\label{thm:nowhere.dense} \leavevmode
\begin{enumerate}
\item
The collection of all simplified copulas is nowhere dense in $(\CC^3,D_1)$ and $(\CC^3,D_\infty)$.
\item The collection of all simplified copulas is nowhere dense in $\CC^3$ with respect to the topology induced by weak 
conditional convergence.
\item
The collection of all simplified copulas with positive density is nowhere dense in $(\CC^3_{{\rm ac}, >0},D_1)$ and $(\CC^3_{{\rm ac}, >0},D_\infty)$.
\end{enumerate}
\end{theorem}
\begin{proof}
To prove the first assertion assume that the $D_1$-closure of the family of all simplified copulas contains an open ball 
$O_{D_1}(C,r)=\{A \in \CC^3: D_1(A,C)<r\}$ with $C \in \CC$ and $r>0$. 
Since according to Lemma \ref{lem:dense} non-simplified checkerboard copulas are dense in $(\CC^3,D_1)$ 
we can find a non-simplified checkerboard copula $C^* \in O_{D_1}(C,r)$. 
Since, by assumption, the $D_1$-closure of the family of all simplified copulas contains $O_{D_1}(C,r)$ 
there exists a sequence $(C_n)_{n \in \mathbb{N}}$ of simplified copulas with 
$\lim_{n \rightarrow \infty} D_1(C_n,C^*) = 0$,
%According to Lemma \ref{lem:dense} we can find a sequence $(C_n^*)_{n \in \mathbb{N}}$ of simplified copulas in 
%$\CC^3_c$ with $\lim_{n \rightarrow \infty} D_1(C_n^*,C^*)=0$, 
a contradiction to Lemma \ref{lem:Delta(C)}.  
\\
Proceeding analogously yields the second and the third assertion.
\end{proof}

Theorem \ref{thm:nowhere.dense} and Theorem \ref{Thm.D1.TV} imply the following two striking results:\pagebreak
\begin{theorem} \label{Non.Dense.TV} \leavevmode
\begin{enumerate}
\item
The collection of all simplified copulas is nowhere dense in $(\CC^3,TV)$.
\item
The collection of all simplified copulas with positive density is nowhere dense in $(\CC^3_{{\rm ac}, >0},TV)$.
\end{enumerate}
\end{theorem}
\begin{theorem} \label{Non.Dense.KL}
The collection of all simplified copulas with positive density is nowhere dense in $(\CC^3_{{\rm ac}, >0},KL)$.
\end{theorem}

Theorems \ref{thm:nowhere.dense}, \ref{Non.Dense.TV} and \ref{Non.Dense.KL} answer the question
\enquote{How dense does the set of simplified densities lie in the set of all densities?} 
posed by \citet{nagler2016} in a complete and definitive manner.

In the same article the authors also pose the question on
\enquote{how far off can we be by assuming a simplified model?} - one of the main objectives of the subsequent sections is
to answer this very question. Notice that, for this purpose, we can restrict ourselves to the metric $d_\infty$ since 
(according to the afore-mentioned results) simplified copulas are nowhere dense w.r.t. $D_1$, $D_\infty$, TV and KL.

%%%%%%%%%%%%%%%%%%%%%%%%%%%%%%%%%%%%%%%%%%%%%%%%%%%%%%%%%%%%%%%%%%%%%%%%%%%%%%%%%%%%%%%%%%%%%%%%%%%%%%%%%%%%%%%%%%%%%%%%%%%%%%%%%%%%%%%%%%%%%%%%%%%%%%%%%%%%%%%%%%%%%%%%%%%%%%%%%%%%%%%%%%%%%%%%%%%%%%%%%%%%%%%%%%%%%%%%%%%%%%%%%%%%%%%%%%%%%%%%%%%%%%%%%%%%%%%%%%%%%%%%%%%%%%%%%%%%%%%%%%%%%%%%%%%%%%%%%%%%%%%%%%%%%%%%%%%%%%%%%%%%%%%%%%%%%%%%%%%%%%%%%%%%%%%%%%%%%%%%%%%%%%%%%%%%%%%%%%%%%%%%%%%%%%%%%%%%%%%%%%%%%%%%%%%%%%%%%%%%%%%%%%%%%%%%%%%%%%%%%%%%%%%%%%%%%%%%%%%%%%%%%%%%%%%%%%%%%%%%%%%%%%%%%%%%%%%%%%%%%%%%%%%%%%%%%%%%%%%%%%%%%%%%%%%%%%%%%%%%%%%%%%%%%%%%%%%%%%%%%%%%%%%%%%%%%%%%%%%%%%%%%%%%%%%%%%%%%%%%%%%%%%%%%%%%%%%%%%%%%%%%%%%%%%%%%%%%%%%%%%%%%%%%%%%%%%%%%%%%%%%%%%%%%%%%%%%%%%%%%%%%%%%%%%%%%%%%%%%%%%%%
%\newpage

\section{Simplified pair-copula constructions} 
\label{Sect.Pair}

Equation (\ref{Eq.Cop.1}) suggests the construction of a three-dimensional copula 
in terms of two families of (conditional) univariate marginal distribution functions characterizing the dependence structure between coordinates $1\&3$ and coordinates $2\&3$, respectively,
and (conditional) bivariate copulas representing the dependence structure between coordinates $1\&2$ conditional on 
the third variable. This just-mentioned construction principle is called \emph{vine decomposition} or 
\emph{pair-copula construction} (see \cite{aas2009,bedford2002}).
  In case the conditioning variable only enters indirectly through the conditional marginals
(as it is the case in Equation (\ref{Eq.GS.1}); see, e.g., \cite{joe1996} for an early reference),
the pair-copula construction is said to be \emph{simplified} (see \cite{hobaek2010}).

%%%%%%%%%%%%%%%%%%%%%%%%%%%%%%%%%%%%%%%%%%%%%%%%%%%%%%%%%%%%%%%%%%%%%%%%%%%%%%%%%%%%%%%%%%%%%%%%%%%%%%%%%%%%%%%%%%%%%%%%%%%%%%%%%%%%%%%%%%%%%%%%%%%%%%%%%%%%%%%%%%%%%%%%%%%%%%%%%%%%%%%%%%%%%%%%%%%%%%%%%%%%%%%%%%%%%%%%%%%%%%%%%%%%%%%%%%%%%%%%%%%%%%%%%%%%%%%%%%%%%%%%%%%%%%%%%%%%%%%%%%%%%%%%%%%%%%%%%%%%%%%%%%%%%%%%%%%%%%%%

\subsection{Construction principle} \label{SC.SPCC}

Simplified pair-copula constructions are used to approximate the data generating copula 
(from $\CC^3_{\rm c}$)
by a simplified copula (from $\CC^3_{\rm S}$)
using the following hierarchical bottom-up algorithm based on Equation (\ref{Eq.GS.1}):
\begin{onelist}
\item
  Estimation of the (conditional) univariate marginal distribution functions
$F_{1|3}(.|t)$ and $F_{2|3}(.|t)$ conditional on $t$;

\item
  Estimation of the (conditional) copula $A$ of coordinates $1\&2$ conditional on variable $3$ assuming that the conditioning variable enters only through the arguments of the conditional copula $A$ (simplifying assumption).
\end{onelist}
  The estimation is either done step-by-step or jointly, parametric or non-parametric, for more information we refer to 
\cite{aas2009,acar2012,hobaek2013,hobaek2010,kauermann2014,nagler2016,spanhel2019} and the references therein.
  For an additional discussion about estimating conditional copulas satis\-fying the simplifying assumption (step (2)),
we additionally refer to \cite{derumigny2017,gijbels2015b,gijbels2015,portier2018}.

  The $3$-dimensional copula resulting from this algorithm is simplified and is said to be a 
\emph{simplified vine copula} (SVC).
  Apparently, the above algorithm and thus its output, the SVC,
depend on the estimation method used and also on the suitable family of copulas from which the estimators are selected. 
%\textcolor{blue}{Ich habe den Begriff SVC nur f{\"u}r den Output des Algorithmus verwendet, denn streng genommen ist SVC nichts anderes als ein unspezifischer Sch{\"a}tzer.}
The above algorithm may certainly provide a reasonable estimator if the (data generating) copula is simplified.
The natural question arising at this point, however, is how well an SVC approximates the data generating copula if 
the latter fails to be simplified. We start with the following example also discussed in \cite[Section 5]{stoeber2013}:
%\smallskip

\begin{example}{} \label{SC.SPCC.EFGM}
  The EFGM copula $C^{\rm EFGM} \in \mathcal{F}^3_\Pi$ introduced in Example \ref{SVC.PInd} is non-simplified.
  Minimizing the Kullback-Leibler divergence between the conditional copula and its estimator 
selected from the family of all bivariate EFGM copulas in step (2) yields the bivariate independence copula 
as the optimal approximation. The SVC selected by a step-by-step algorithm hence equals the three-dimensional 
independence copula. 
\\
  Comparing the data generating copula with its selected SVC 
yields a $d_\infty$-distance of $1/64$;
this equals $6.25\%$ of the maximal $d_\infty$-distance of two copulas within the (Fr{\'e}chet) class of all 
copulas having pairwise independent marginals
(using the results in \cite[Section 3.3]{nelsen2012} it is straightforward to verify that 
the diameter of this class is $1/4$).
\end{example}{}
%\smallskip

We refer to \cite{acar2012, hobaek2010, stoeber2013} for more examples and comparisons of the data generating copula 
with its selected simplified vine copula whereby the quality of the approximations is judged quite differently. 

Aiming to obtain more general analytic results concerning the optimality of simplified pair-copula constructions,
in what follows we discuss the concept of partial vine copulas.

%%%%%%%%%%%%%%%%%%%%%%%%%%%%%%%%%%%%%%%%%%%%%%%%%%%%%%%%%%%%%%%%%%%%%%%%%%%%%%%%%%%%%%%%%%%%%%%%%%%%%%%%%%%%%%%%%%%%%%%%%%%%%%%%%%%%%%%%%%%%%%%%%%%%%%%%%%%%%%%%%%%%%%%%%%%%%%%%%%%%%%%%%%%%%%%%%%%%%%%%%%%%%%%%%%%%%%%%%%%%%%%%%%%%%%%%%%%%%%%%%%%%%%%%%%%%%%%%%%%%%%%%%%%%%%%%%%%%%%%%%%%%%%%%%%%%%%%%%%%%%%%%%%%%%%%%%%%%%%%%

\subsection{Partial vine copulas (PVCs)}
  The basic idea behind a partial vine copula is that the conditional bivariate copulas of the original 
  three-dimensional copula are averaged (see \cite{spanhel2016,spanhel2019}): \\
Considering that for every $C \in \CC^3_{\rm c}$ the copula $C_{12;3}^t$ is unique for almost every $t \in \I$ it follows that
the function $C_p: \I^2 \to \I$, given by 
$$
  C_p (\sss)
	:= \int\limits_{\I} C_{12;3}^t (\sss) \;\mathrm{d}\lambda(t)
$$
is well--defined. In the sequel we will refer to $C_p$ as the \emph{partial copula} of $C$ (also see \cite{Bergsma2011}).
Coinciding with the expected conditional copula, the partial copula is often used as an approximation of the conditional copula (see \cite{spanhel2016, spanhel2019} for more information).
Given $C_p$ in the above setting the mapping $\psi: \CC^3_{\rm c} \to \CC^3_{\rm c}$, given by
\begin{eqnarray*}
  \big( \psi (C) \big) (\uuu,v)
	& := & \int\limits_{[0,v]} C_{p} \big( F_{1|3}(u_1|t), F_{2|3}(u_2|t) \big) \;\mathrm{d}\lambda(t)
	%\\
	%&  = & \int\limits_{[0,u_2]} \int\limits_{[0,1]} C_{13;2}^w \big( C_{1|2}(u_1|v), C_{3|2}(u_3|v) \big) \; \mathrm{d} \lambda(w) \; \mathrm{d} \lambda(v)
  %\\
	%&  = & \int\limits_{[0,1]} \int\limits_{[0,u_2]} C_{13;2}^w \big( C_{1|2}(u_1|v), C_{3|2}(u_3|v) \big) \; \mathrm{d} \lambda(v) \; \mathrm{d} \lambda(w)
	%\\*
	%&  = & \int\limits_{[0,u_2]} \int\limits_{[0,1]} C_{13|2} \big( C_{1|2}^\leftarrow \big( C_{1|2}(u_1|v) | w \big), C_{3|2}^\leftarrow \big( C_{3|2}(u_3|v) | w \big) \, \big| \, w \big) \; \mathrm{d} \lambda(w) \; \mathrm{d} \lambda(v)
\end{eqnarray*}
is well--defined and assigns to every copula $C \in \CC^3_{\rm c}$ a simplified copula $\psi(C)$. 
The copula $\psi (C)$ is referred to as the \emph{partial vine copula} of $C$ (with respect to the third coordinate)
in the sequel.
It is obvious that every partial vine copula is simplified.

  The transformation $\psi$ preserves the dependence structure between coordinates $1 \& 3$ as well as between 
  coordinates $2 \& 3$. The following lemma gathers some additional properties of $\psi$: 

%\smallskip
\begin{lemma}{} \label{PVC.Lemma1}
  Suppose that $C \in \CC^3_{\rm c}$. Then the following assertions hold:
\begin{onelist}
\item
  The partial vine copula $\psi(C)$ of $C$ satisfies
$(\psi (C))_{13} = C_{13}$ as well as $(\psi (C))_{23} = C_{23}$.

\item 
  If $C$ is simplified then $\psi(C)=C$ holds.

\item
  The mapping $\psi: \CC^3_{\rm c} \to \CC^3_{\rm c}$ is not injective.
\end{onelist}
\end{lemma}{}

\begin{proof}{}
  Since $F_{1|3}(1|t)=1=F_{2|3}(1|t)$ for almost every $t \in \I$ we have 
\begin{eqnarray*}
  (\psi (C))_{13} (u_1,v)
	& = & \int\limits_{[0,v]} C_{p} \big( F_{1|3}(u_1|t), F_{2|3}(1|t) \big) \;\mathrm{d}\lambda(t)
	%= \int\limits_{[0,v]} C_{p} \big( F_{1|3}(u_1|t), 1 \big) \;\mathrm{d}\lambda(t)
	\\
	& = & \int\limits_{[0,v]} F_{1|3}(u_1|t) \; \mathrm{d} \lambda(t) = C_{13} (u_1,v) 
\end{eqnarray*}
%\begin{eqnarray*}
%  (\psi (C))_{13} (u_1,v)
%	& = & \int\limits_{[0,v]} C_{p} \big( F_{1|3}(u_1|t), F_{2|3}(1|t) \big) \;\mathrm{d}\lambda(t)
	%\\
	%& = & \int\limits_{[0,v]} \int\limits_{\I} C_{13;2}^z \big( F_{1|3}(u_1|t), F_{2|3}(1|t) \big) 
	%			\; \mathrm{d} \lambda(z) \mathrm{d} \lambda(t)
%	\\
%	& = & \int\limits_{[0,v]} C_{p} \big( F_{1|3}(u_1|t), 1 \big) \;\mathrm{d}\lambda(t)
%	\\
%	& = & \int\limits_{[0,v]} F_{1|3}(u_1|t) \; \mathrm{d} \lambda(t)
%	\\
%	& = & C_{13} (u_1,v) 
%\end{eqnarray*}
for all $(u_1,v) \in \I^2$. The identity $(\psi (C))_{23} = C_{23}$ follows in the same manner. 
  Assertion (2) is trivial and Assertion (3) follows from Example \ref{PVC.PInd.Ex2} below.
\end{proof}{}
%\smallskip

%\smallskip
\begin{example}{} \label{PVC.PInd.Ex1} 
(Class $\mathcal{F}^3_\Pi$, cont.) \\
For every $C \in \mathcal{F}^3_\Pi$ the identity
$$
  (\psi(C)) (\uuu,v)
	= C_{12} (\uuu) \, v 
$$
holds for all $(\uuu,v) \in \I^2 \times \I$.
\\
In fact, considering that 
$F_{1|3}(s_1|t) = s_1$ and 
$F_{2|3}(s_2|t) = s_2$ 
hold for all $\sss \in \I^2$ and almost all $t \in \I$ we get    
$$
  C_p (\sss)
	%= \int\limits_{\I} C_{12;3}^t (\sss) \; \mathrm{d} \lambda(t)
	= \int\limits_{\I} C_{12;3}^t \big( F_{1|3}(s_1|t), F_{2|3}(s_2|t) \big) 
		\; \mathrm{d} \lambda(t)
  = C (\sss,1)
	= C_{12} (\sss)
$$
for all $\sss \in \I^2$. 
Having this, the fact that $(\psi(C)) (\uuu,v) =  C_{p} (\uuu) \, v = C_{12} (\uuu) \, v $
%	the identity
%$$
%  \big( \psi(C) \big) (\uuu,v)
%	= \int\limits_{[0,v]} C_{p} \big( F_{1|3}(u_1|t), F_{2|3}(u_3|t) \big) 
%		\; \mathrm{d} \lambda(t)
%	= \int\limits_{[0,v]} C_{12} (\uuu) 
%		\; \mathrm{d} \lambda(t)
%  = C_{12} (\uuu) \, v 
%$$
holds for all $(\uuu,v) \in \I^2 \times \I$ follows immediately.
\end{example}{}
%\smallskip

As a consequence of Example \ref{PVC.PInd.Ex1}, if $C_{13} = \Pi = C_{23}$ and, additionally, $C_{12} = \Pi$, then $$\psi(C)=\Pi$$ follows although, in general, $C \neq \Pi$. This fact applies in particular to the following copulas: 

%\smallskip
\begin{example}{} \label{PVC.PInd.Ex2} \leavevmode
\begin{onelist}
\item
  The EFGM copula $C^{\rm EFGM} \in \mathcal{F}^3_\Pi$ introduced in Example \ref{SVC.PInd}
is non-simplified, 
satisfies 
$$
  C^{\rm EFGM}_{12} = C^{\rm EFGM}_{13} = C^{\rm EFGM}_{23} = \Pi
	\qquad \textrm{ and } \qquad 
	C^{\rm EFGM}_p = \Pi \textrm{ (also see \cite{spanhel2016})},
$$  
and hence $\psi(C^{\rm EFGM})= \Pi \neq C^{\rm EFGM}$.

\item
  The copula $C^{\rm Cube} \in \mathcal{F}^3_\Pi$ introduced in Example \ref{SVC.PInd}
is non-simplified,
satisfies 
$$
  C^{\rm Cube}_{12} = C^{\rm Cube}_{13} = C^{\rm Cube}_{23} = \Pi
	\qquad \textrm{ and } \qquad 
	C^{\rm Cube}_p = \Pi,
$$
and hence $\psi(C^{\rm Cube})= \Pi \neq C^{\rm Cube}$.

\item
  The copula $C^{\rm RCube} \in \mathcal{F}^3_\Pi$ whose mass is distributed uniformly within the cubes
\begin{eqnarray*}
  \big(0,\tfrac{1}{2}\big) \times \big(\tfrac{1}{2},1\big) \times \big(0,\tfrac{1}{2}\big)
	& \quad & \big(\tfrac{1}{2},1\big) \times \big(0,\tfrac{1}{2}\big) \times \big(0,\tfrac{1}{2}\big)
	\\*
	\big(0,\tfrac{1}{2}\big) \times \big(0,\tfrac{1}{2}\big) \times \big(\tfrac{1}{2},1\big)
	& \quad & \big(\tfrac{1}{2},1\big) \times \big(\tfrac{1}{2},1\big) \times \big(\tfrac{1}{2},1\big)
\end{eqnarray*}
and has no mass outside these cubes is non-simplified,
satisfies 
$$
  C^{\rm RCube}_{12} = C^{\rm RCube}_{13} = C^{\rm RCube}_{23} = \Pi
	\qquad \textrm{ and } \qquad 
	C^{\rm RCube}_p = \Pi,
$$
and hence $\psi(C^{\rm RCube})= \Pi \neq C^{\rm RCube}$.
\end{onelist}
  The copula in (3) is denoted as 'RCube' since it is a reflected version of the copula in (2); 
  both are related to each other via 
$ \mu_{C^{\rm RCube}} = (\mu_{C^{\rm Cube}})^T $
where $T: \I^2 \times \I \to \I^2 \times \I$ is the mapping given by $T(\uuu,v) := (\uuu,1-v)$ and $(\mu_{C^{\rm Cube}})^T$
denotes the push-forward of $\mu_{C^{\rm Cube}}$ via $T$.
\end{example}{}
%\smallskip

PVCs have been used in \cite{kurz2017} to test the simplifying assumption in vine copula models and in \cite{nagler2016}
 to construct a non-parametric estimator for multivariate distributions. 
 In \cite{spanhel2019} the authors showed that \grqq{}under regularity conditions, stepwise estimators of pair-copula constructions converge to the PVC irrespective of whether the simplifying assumption holds or not" 
 (see \cite[Corollary 6.1]{spanhel2019}). 
  Nevertheless, this does not need to be true if the estimation is done jointly in a non-simplified setting 
  (see \cite[Corollary 6.1]{spanhel2019}).
  The authors further proved that \grqq{}if one sequentially minimizes the Kullback-Leibler divergence 
  related to each tree then the optimal SVC is the PVC" (see \cite[Theorem 5.1]{spanhel2019}). 
  Since, again, this is not necessarily true if the estimation is done jointly in a non-simplified setting 
  (see \cite[Theorem 5.2]{spanhel2019}) the authors conclude that PVCs \grqq{}may not be the best 
  approximation in the space of SVCs" but are \grqq{}often the best feasible SVC approximation in practice."
	
  Motivated by these results in what follows we discuss analytic properties and 
  optimality of simplified pair-copula constructions 
  and focus mainly on partial vine copulas. In Section \ref{Section.PVC.Opt.} we calculate the 
  $d_\infty$-distance between non-simplified copulas and their unique partial vine copulas for different dependence 
  structures, in Section \ref{Section.PVC.Cont.} we discuss continuity of $\psi$ with respect to different notions of convergence.

%%%%%%%%%%%%%%%%%%%%%%%%%%%%%%%%%%%%%%%%%%%%%%%%%%%%%%%%%%%%%%%%%%%%%%%%%%%%%%%%%%%%%%%%%%%%%%%%%%%%%%%%%%%%%%%%%%%%%%%%%%%%%%%%%%%%%%%%%%%%%%%%%%%%%%%%%%%%%%%%%%%%%%%%%%%%%%%%%%%%%%%%%%%%%%%%%%%%%%%%%%%%%%%%%%%%%%%%%%%%%%%%%%%%%%%%%%%%%%%%%%%%%%%%%%%%%%%%%%%%%%%%%%%%%%%%%%%%%%%%%%%%%%%%%%%%%%%%%%%%%%%%%%%%%%%%%%%%%%%%%%%%%%%%%%%%%%%%%%%%%%%%%%%%%%%%%%%%%%%%%%%%%%%%%%%%%%%%%%%%%%%%%%%%%%%%%%%%%%%%%%%%%%%%%%%%%%%%%%%%%%%%%%%%%%%%%%%%%%%%%%%%%%%%%%%%%%%%%%%%%%%%%%%%%%%%%%%%%%%%%%%%%%%%%%%%%%%%%%%%%%%%%%%%%%%%%%%%%%%%%%%%%%%%%%%%%%%%%%%%%%%%%%%%%%%%%%%%%%%%%%%%%%%%%%%%%%%%%%%%%%%%%%%%%%%%%%%%%%%%%%%%%%%%%%%%%%%%%%%%%%%%%%%%%%%%%%%%%%%%%%%%%%%%%%%%%%%%%%%%%%%%%%%%%%%%%%%%%%%%%%%%%%%%%%%%%%%%%%%%%%%%
%\newpage

\section{Optimality of partial vine copulas}\label{Section.PVC.Opt.}
\label{Sect.Partial.Opt.}

Main objective of this section is to provide an answer to the question 
\grqq{}how far off can we be by assuming a simplified model?" posed by \citet{nagler2016}. We proceed as follows:
We first show that partial vine copulas are never the best simplified copula approximation 
(with respect to $d_\infty$) if the true copula is non-simplified (Theorem \ref{PVC.NonOpt}).
  We then compare non-simplified copulas $C$ with their unique partial vine copulas $\psi(C)$ in different settings and calculate their $d_\infty$-distance. 
	It turns out that the maximal distance within the family of all copulas with pairwise independent marginals
  is $1/8$ which corresponds to $50 \%$ of the diameter of this class w.r.t. $d_\infty$.
  Going even further, we provide an example of a copula $C \in \mathcal{C}^3$ fulfilling $d_\infty(C,\psi(C)) =3/16$ 
  which, in turn, corresponds to $28.125 \%$ of the diameter of ($\mathcal{C}^3,d_\infty)$. 
  In other words, $\psi(C)$ can be far away from $C$, so working with PVCs must be done with care. 

  Corollary \ref{S.Dense} implies that if $C$ does not fulfill the simplifying assumption  
  then the partial vine copula fails to be optimal with respect to $d_\infty$: 
%\smallskip

\begin{theorem}{} \label{PVC.NonOpt}
  Suppose that $C \in \CC^3_{\rm c}$ is non-simplified.
  Then there exists some simplified copula $D \in \CC^3_{\rm S}$ 
satisfying $ d_\infty (C,D) 	< d_\infty (C,\psi(C)) $. 
\end{theorem}{}

\begin{proof}{}
  Considering $C \in \CC^3_{\rm c} \backslash \CC^3_{\rm S}$ we have $C \neq \psi(C)$, so setting 
  $0 < d_\infty (C,\psi(C)) =: \varepsilon$ and using Corollary \ref{S.Dense} yields the desired result.
\end{proof}{}
%\smallskip

As next step we calculate 
$$
\sup_{C \in \mathcal{F}^3_\Pi} d_\infty(C,\psi(C)),
$$
show that the supremum is attained and then characterize all elements in $\mathcal{F}^3_\Pi$ attaining the maximum.
Afterwards we provide a lower bound for 
$$
\sup_{C \in \mathcal{C}^3_c} d_\infty(C,\psi(C)).
$$
The (dis)continuity results in Section 6 will make it clear why we can not simply use compactness of $(\mathcal{C}^3,d_\infty)$
to conclude that the supremum in the last expression is attained.

%%%%%%%%%%%%%%%%%%%%%%%%%%%%%%%%%%%%%%%%%%%%%%%%%%%%%%%%%%%%%%%%%%%%%%%%%%%%%%%%%%%%%%%%%%%%%%%%%%%%%%%%%%%%%%%%%%%%%%%%%%%%%%%%%%%%%%%%%%%%%%%%%%%%%%%%%%%%%%%%%%%%%%%%%%%%%%%%%%%%%%%%%%%%%%%%%%%%%%%%%%%%%%%%%%%%%%%%%%%%%%%%%%%%%%%%%%%%%%%%%%%%%%%%%%%%%%%%%%%%%%%%%%%%%%%%%%%%%%%%%%%%%%%%%%%%%%%%%%%%%%%%%%%%%%%%%%%%%%%%

\subsection{Worst case scenario for the class $\mathcal{F}^3_\Pi$}

The following theorem holds - notice that the set of maximizers includes the two copulas $C^{\rm Cube}$ and $C^{\rm RCube}$ introduced in Examples \ref{SVC.PInd} and \ref{PVC.PInd.Ex2}:

%\smallskip
\begin{theorem}{} \label{PVC.PInd.Dist.}
For every copula $C \in \mathcal{F}^3_\Pi$ the inequality $d_\infty \big(C,\psi(C)\big) \leq \frac{1}{8}$ holds. 
Moreover, for every $C \in \mathcal{F}^3_\Pi$ the following two conditions are equivalent: 
\begin{aaalist}
\item 
$d_\infty \big(C,\psi(C)\big) = \frac{1}{8}$.

\item
$C$ satisfies either
\begin{eqnarray*}
  \mu_C \big[ \big(0,\tfrac{1}{2} \big) \times \big(0,\tfrac{1}{2}\big) \times \big(0,\tfrac{1}{2}\big) \big] 
	&=& \tfrac{1}{4} =
	\mu_C \big[ \big(0,\tfrac{1}{2}\big) \times \big(\tfrac{1}{2},1\big) \times \big(\tfrac{1}{2},1\big) \big]
	\\*
	\mu_C \big[ \big(\tfrac{1}{2},1\big) \times \big(\tfrac{1}{2},1\big) \times \big(0,\tfrac{1}{2}\big) \big]
	&=& \tfrac{1}{4} =
	\mu_C \big[ \big(\tfrac{1}{2},1\big) \times \big(0,\tfrac{1}{2}\big) \times \big(\tfrac{1}{2},1\big) \big]
\end{eqnarray*}
or
\begin{eqnarray*}
  \mu_C \big[ \big(0,\tfrac{1}{2}\big) \times \big(0,\tfrac{1}{2}\big) \times \big(\tfrac{1}{2},1\big) \big]
	&=& \tfrac{1}{4} = 
  \mu_C \big[ \big(0,\tfrac{1}{2}\big) \times \big(\tfrac{1}{2},1\big) \times \big(0,\tfrac{1}{2}\big) \big]
	  \\
  \mu_C \big[ \big(\tfrac{1}{2},1\big) \times \big(\tfrac{1}{2},1\big) \times \big(\tfrac{1}{2},1\big) \big]
	&=& \tfrac{1}{4} = 
		\mu_C \big[ \big(\tfrac{1}{2},1\big) \times \big(0,\tfrac{1}{2}\big) \times \big(0,\tfrac{1}{2}\big) \big].
\end{eqnarray*}
\end{aaalist}
\end{theorem}{}

\begin{proof}{}
  Consider $C \in \mathcal{F}^3_\Pi$, fix $(\uuu,v) \in \I^2 \times (0,1)$ and set 
$$
  k
	:= \frac{1}{v} \; \int\limits_{[0,v]} C_{12;3}^t (\uuu) \; \mathrm{d} \lambda(t)
	\qquad \textrm{ and } \qquad
	l
	:= \frac{1}{1-v} \; \int\limits_{[v,1]} C_{12;3}^t (\uuu) \; \mathrm{d} \lambda(t).
$$
Then
$$
  C_{12} (\uuu) 
	= \int\limits_{\I} C_{12;3}^t (\uuu) \; \mathrm{d} \lambda(t)
	= k \, v + l \, (1-v)
$$
and
$$ 
  C(\uuu,v)
	= \int\limits_{[0,v]} C_{12;3}^t \big( F_{1|3}(u_1|t), F_{2|3}(u_2|t) \big) \; \mathrm{d} \lambda(t)
	= \int\limits_{[0,v]} C_{12;3}^t (\uuu) \; \mathrm{d} \lambda(t)	= k \, v 
$$
Having this and using Example \ref{PVC.PInd.Ex1} yields
$$	C(\uuu,v) - (\psi(C)) (\uuu,v)
	= k \, v - C_{12} (\uuu) \, v
	= k \, v - \big( k \, v + l \, (1-v) \big) \, v
	= v \, (1-v) \, (k - l).
$$
Since $W (\uuu) \leq k \leq M (\uuu)$ as well as $W (\uuu) \leq l \leq M (\uuu)$ we further have
$$
  \big| C(\uuu,v) - (\psi(C)) (\uuu,v) \big|
	= v \, (1-v) \, |k - l|
	\leq v \, (1-v) \, d_\infty (M,W)
	\leq \frac{1}{8}
$$
Considering  
$ d_\infty \big(C^{\rm Cube},\psi(C^{\rm Cube})\big)
	= C^{\rm Cube} \big( \tfrac{\bf 1}{\bf 2} \big) - \Pi \big( \tfrac{\bf 1}{\bf 2} \big)
	= \tfrac{1}{4} - \tfrac{1}{8}
	= \tfrac{1}{8} $
we finally obtain
$$
  \frac{1}{8}
    = d_\infty \big(C^{\rm Cube},\psi(C^{\rm Cube})\big)
  \leq \sup_{C \in \mathcal{F}_\Pi} d_\infty \big(C,\psi(C)\big) 
	\leq \frac{1}{8}
$$
which proves the first assertion. 

For proving the stated equivalence we proceed as follows: First suppose that (b) holds.
Considering that for $\uuu=\tfrac{\bf 1}{\bf 2}$ and $v = \tfrac{1}{2}$ we have  
$$
  |k - l|
	= \big| 2 \; \mu_C \big[ \big(0,\tfrac{1}{2} \big) \times \big(0,\tfrac{1}{2}\big) \times \big(0,\tfrac{1}{2}\big) \big] - 2 \; \mu_C \big[ \big(0,\tfrac{1}{2} \big) \times \big(0,\tfrac{1}{2}\big) \times \big(\tfrac{1}{2},1\big) \big] \big|
	= \frac{2}{4} 
$$
it follows that 
$$
  \frac{1}{8}
  \geq d_\infty \big(C,\psi(C)\big)
  \geq \big| C\big( \tfrac{\bf 1}{\bf 2} \big) - (\psi(C)) \big( \tfrac{\bf 1}{\bf 2} \big) \big|
	  =  \frac{1}{4} \; \big| k - l \big|
	  =  \frac{1}{4} \; \frac{2}{4} 
	  =  \frac{1}{8} 
$$
so (a) holds and it remains to show that (a) implies (b). First of all notice that 
$$
  \frac{1}{8}
	= d_\infty \big(C,\psi(C)\big)
	%= \sup_{(\uuu,v) \in \I^2 \times \I} \big| C(\uuu,v) - (\psi(C)) (\uuu,v) \big|
	= \sup_{(\uuu,v) \in \I^2 \times \I} \big| v \, (1-v) \, (k - l)\big|
$$
and that it is straightforward to show that $|k - l|$ is at most $1/2$ and that $1/2$ can only be attained 
by choosing $u_1 = 1/2 = u_2$ (irrespective of the value of $v$).
In this case either $k=1/2$ and $l=0$ or $k=0$ and $l=1/2$.
Thus, 
$$
  \frac{1}{8}
	  =  d_\infty \big(C,\psi(C)\big)
	\leq \; \sup_{v \in \I} \big| v \, (1-v) \big| \; \cdot 
			 \sup_{(\uuu,v) \in \I^2 \times \I} \big| k - l \big|
	  =  \frac{1}{2} \; \sup_{v \in \I} \big| v \, (1-v) \big| 
	  =  \frac{1}{8}		 
$$
and $v=1/2$. From the first part of this proof we get 
$$
  \mu_C \big[ \big( 0, \tfrac{1}{2} \big)^3 \big]
	 =  C \big( \tfrac{\bf 1}{\bf 2}, \tfrac{1}{2} \big)
	 =  k \, \tfrac{1}{2}
	\in \big\{ 0, \tfrac{1}{4} \big\}
$$
as well as 
\begin{eqnarray*}
	\mu_C \big[ \big( 0, \tfrac{1}{2} \big)^3 \big] 
	  + \mu_C \big[ \big(0,\tfrac{1}{2}\big) \times \big(0,\tfrac{1}{2}\big) \times \big(\tfrac{1}{2},1\big) \big] 
	& = & \mu_C \big[ \big(0,\tfrac{1}{2}\big) \times \big(0,\tfrac{1}{2}\big) \times \I \big] 
	\\
	& = & C_{12} \big( \tfrac{\bf 1}{\bf 2} \big) 
	\\
	& = & k \, v + l \, (1-v)
	\\
	& = & \tfrac{1}{4}
\end{eqnarray*}
Since for every $C \in \mathcal{F}^3_\Pi$ we have 
\begin{eqnarray*}
  \mu_C \big[ \big(0,\tfrac{1}{2}\big) \times \I \times \big(0,\tfrac{1}{2}\big) \big]
	= \tfrac{1}{4}		
	&& 
	\mu_C \big[ \I \times \big(0,\tfrac{1}{2}\big) \times \big(0,\tfrac{1}{2}\big) \big]
	= \tfrac{1}{4}
	\\*
	\mu_C \big[ \big(0,\tfrac{1}{2}\big) \times \I \times \big(\tfrac{1}{2},1\big) \big]
	= \tfrac{1}{4}	
	&& \mu_C \big[ \I \times \big(0,\tfrac{1}{2}\big) \times \big(\tfrac{1}{2},1\big) \big]
	= \tfrac{1}{4}	
	\\*
	\mu_C \big[ \big(\tfrac{1}{2},1\big) \times \I \times \big(0,\tfrac{1}{2}\big) \big]
	= \tfrac{1}{4}		
	&& \mu_C \big[ \I \times \big(\tfrac{1}{2},1\big) \times \big(0,\tfrac{1}{2}\big) \big]
	= \tfrac{1}{4}
	\\*
	\mu_C \big[ \big(\tfrac{1}{2},1\big) \times \I \times \big(\tfrac{1}{2},1\big) \big]
	= \tfrac{1}{4}		
	&& \mu_C \big[ \I \times \big(\tfrac{1}{2},1\big) \times \big(\tfrac{1}{2},1\big) \big]
	= \tfrac{1}{4}
\end{eqnarray*}
it suffices to distinguish the following two situations: \\
(i) If $\mu_C \big[ \big( 0, \tfrac{1}{2} \big)^3 \big] = \tfrac{1}{4}$ then 
$\mu_C \big[ \big(0,\tfrac{1}{2}\big) \times \big(0,\tfrac{1}{2}\big) \times \big(\tfrac{1}{2},1\big) \big]
	= 0$
and $C_{13} = \Pi = C_{23}$ yields
\begin{eqnarray*}
  \mu_C \big[ \big(0,\tfrac{1}{2}\big) \times \big(\tfrac{1}{2},1\big) \times \big(0,\tfrac{1}{2}\big) \big]
	= 0
	&&
	\mu_C \big[ \big(0,\tfrac{1}{2}\big) \times \big(\tfrac{1}{2},1\big) \times \big(\tfrac{1}{2},1\big) \big]
	= \tfrac{1}{4}
	\\
	\mu_C \big[ \big(\tfrac{1}{2},1\big) \times \big(0,\tfrac{1}{2}\big) \times \big(0,\tfrac{1}{2}\big) \big]
	= 0
	&&
	\mu_C \big[ \big(\tfrac{1}{2},1\big) \times \big(\tfrac{1}{2},1\big) \times \big(0,\tfrac{1}{2}\big) \big]
	= \tfrac{1}{4}
	\\
	\mu_C \big[ \big(\tfrac{1}{2},1\big) \times \big(\tfrac{1}{2},1\big) \times \big(\tfrac{1}{2},1\big) \big]
	= 0
	&&
	\mu_C \big[ \big(\tfrac{1}{2},1\big) \times \big(0,\tfrac{1}{2}\big) \times \big(\tfrac{1}{2},1\big) \big]
	= \tfrac{1}{4}
\end{eqnarray*}
(ii) If $\mu_C \big[ \big( 0, \tfrac{1}{2} \big)^3 \big] = 0$,
then 
$\mu_C \big[ \big(0,\tfrac{1}{2}\big) \times \big(0,\tfrac{1}{2}\big) \times \big(\tfrac{1}{2},1\big) \big]
	= \tfrac{1}{4}$
and $C_{13} = \Pi = C_{23}$ yields
\begin{eqnarray*}
  \mu_C \big[ \big(0,\tfrac{1}{2}\big) \times \big(\tfrac{1}{2},1\big) \times \big(0,\tfrac{1}{2}\big) \big]
	= \tfrac{1}{4}
	&&
	\mu_C \big[ \big(0,\tfrac{1}{2}\big) \times \big(\tfrac{1}{2},1\big) \times \big(\tfrac{1}{2},1\big) \big]
	= 0
	\\
	\mu_C \big[ \big(\tfrac{1}{2},1\big) \times \big(0,\tfrac{1}{2}\big) \times \big(0,\tfrac{1}{2}\big) \big]
	= \tfrac{1}{4}
	&&
	\mu_C \big[ \big(\tfrac{1}{2},1\big) \times \big(\tfrac{1}{2},1\big) \times \big(0,\tfrac{1}{2}\big) \big]
	= 0
	\\
	\mu_C \big[ \big(\tfrac{1}{2},1\big) \times \big(\tfrac{1}{2},1\big) \times \big(\tfrac{1}{2},1\big) \big]
	= \tfrac{1}{4}
	&&
	\mu_C \big[ \big(\tfrac{1}{2},1\big) \times \big(0,\tfrac{1}{2}\big) \times \big(\tfrac{1}{2},1\big) \big]
	= 0,
\end{eqnarray*}
which completes the proof.
\end{proof}{}
%\smallskip

Notice that Theorem \ref{PVC.PInd.Dist.} implies the following striking property: 
The maximal distance of a copula $C$ with pairwise independent marginals and its partial vine copula $\psi(C)$ corresponds to  
\begin{hylist}
\item
$50 \%$ of the diameter of the metric space of all copulas with pairwise independent marginals w.r.t. $d_\infty$;
the diameter of this class equals $1/4$ which can be calculated via \cite[Section 3.3]{nelsen2012}.
\item
$18.75 \%$ of the diameter of $(\mathcal{C}^3, d_\infty)$, which is given by $2/3$.
\end{hylist}

\begin{remark}
An equally striking result can be shown for the metric $D_1$: Again working with $C^{\rm Cube}$ it follows that 
$$
\sup_{C \in \mathcal{F}^3_\Pi} D_1(C,\psi(C)) \geq \frac{15}{64} = D_1(C^{\rm Cube},\psi(C^{\rm Cube}))
$$
holds. Using the results in \cite{trutschnig2015}
we therefore get that the maximal $D_1$-distance of a copula $C \in \mathcal{F}^3_\Pi$ and its partial vine copula 
$\psi(C)$ is greater than or equal to 42.1875\% of the diameter of the metric space $(\mathcal{C}^3,D_1)$;
the diameter of this class is at most $5/9$ which can be calculated via \cite[Lemma 2]{trutschnig2015}.
\end{remark}

\begin{remark}{} \label{rem.cond.}
At this point it is worth to mention that $C^{\rm Cube}$ is exchangeable and hence 
approximating $C^{\rm Cube}$ by $\psi(C^{\rm Cube})$ leads to equally poor results no matter which coordinate 
is chosen for the conditioning.
\end{remark}{}

%%%%%%%%%%%%%%%%%%%%%%%%%%%%%%%%%%%%%%%%%%%%%%%%%%%%%%%%%%%%%%%%%%%%%%%%%%%%%%%%%%%%%%%%%%%%%%%%%%%%%%%%%%%%%%%%%%%%%%%%%%%%%%%%%%%%%%%%%%%%%%%%%%%%%%%%%%%%%%%%%%%%%%%%%%%%%%%%%%%%%%%%%%%%%%%%%%%%%%%%%%%%%%%%%%%%%%%%%%%%%%%%%%%%%%%%%%%%%%%%%%%%%%%%%%%%%%%%%%%%%%%%%%%%%%%%%%%%%%%%%%%%%%%%%%%%%%%%%%%%%%%%%%%%%%%%%%%%%%%%

\subsection{Worst case scenario for the full class $\mathcal{C}^3_c$}

We are now going to show that the maximal $d_\infty$-distance of a copula $C \in \mathcal{C}^3_c$ 
and its assigned partial vine copula $\psi(C)$ is at least $3/16$ which corresponds to $28.125 \%$ of 
the diameter of the metric space $(\mathcal{C}^3_c,d_\infty)$.  

%\smallskip
\begin{example}{} \label{PVC.PConst.Ex} 
Consider the intervals $I_i := \big( \tfrac{i-1}{4}, \tfrac{i}{4} \big)$ for $i \in \{1,\dots,4\}$.
We use Equation (\ref{Eq.Cop.1}) in order to construct a three-dimensional non-simplified copula $C$
satisfying that its conditional copulas 
$C_{12;3}^t$, $t \in \I$, 
are identical for all $t$ within each of the four subintervals. To this end, set 
$$
  A^t := \sum_{i=1}^4 D^i \mathds{1}_{I_i}(t) + \Pi \, \mathds{1}_{\{0,\frac{1}{4},\frac{1}{2},\frac{3}{4},1\}}(t)
$$
where the bivariate copulas $D^1, \dots, D^4$ are the shuffles of $W$ depicted in Figure \ref{fig.cube.1} 
(for the definition of shuffles we refer to \cite[Definition 2.1]{durantefs2010} and \cite[Section 5]{fuchs2018}).
\\
\begin{figure}[h]
\centering
\begin{tikzpicture}[xscale=0.55,yscale=0.55] 
\node (AA) at (-0.3,4.) {$1$};
\node (BA) at (-0.3,0.) {$0$};
\node (CA) at (4.3,0.) {$1$};

\node (AB) at (5.7,4.) {$1$};
\node (BB) at (5.7,0.) {$0$};
\node (CB) at (10.3,0.) {$1$};

\node (AC) at (11.7,4.) {$1$};
\node (BC) at (11.7,0.) {$0$};
\node (CC) at (16.3,0.) {$1$};

\node (AD) at (17.7,4.) {$1$};
\node (BD) at (17.7,0.) {$0$};
\node (CD) at (22.3,0.) {$1$};

\draw [thin]  (0,0) -- (0,4);     \draw [thin]  (6,0) -- (6,4); 
\draw [thin]  (1,0) -- (1,4);     \draw [thin]  (7,0) -- (7,4); 
\draw [thin]  (2,0) -- (2,4);     \draw [thin]  (8,0) -- (8,4); 
\draw [thin]  (3,0) -- (3,4);     \draw [thin]  (9,0) -- (9,4); 
\draw [thin]  (4,0) -- (4,4);     \draw [thin]  (10,0) -- (10,4); 
\draw [thin]  (0,0) -- (4,0);     \draw [thin]  (6,0) -- (10,0); 
\draw [thin]  (0,1) -- (4,1);     \draw [thin]  (6,1) -- (10,1); 
\draw [thin]  (0,2) -- (4,2);     \draw [thin]  (6,2) -- (10,2); 
\draw [thin]  (0,3) -- (4,3);     \draw [thin]  (6,3) -- (10,3); 
\draw [thin]  (0,4) -- (4,4);     \draw [thin]  (6,4) -- (10,4); 
\draw [thick] (0,2) -- (1,1);     \draw [thick] (6,4) -- (7,3); 
\draw [thick] (1,4) -- (3,2);     \draw [thick] (7,1) -- (8,0); 
\draw [thick] (3,1) -- (4,0);     \draw [thick] (8,3) -- (10,1);    

\draw [thin]  (12,0) -- (12,4);     \draw [thin]  (18,0) -- (18,4); 
\draw [thin]  (13,0) -- (13,4);     \draw [thin]  (19,0) -- (19,4); 
\draw [thin]  (14,0) -- (14,4);     \draw [thin]  (20,0) -- (20,4); 
\draw [thin]  (15,0) -- (15,4);     \draw [thin]  (21,0) -- (21,4); 
\draw [thin]  (16,0) -- (16,4);     \draw [thin]  (22,0) -- (22,4); 
\draw [thin]  (12,0) -- (16,0);     \draw [thin]  (18,0) -- (22,0); 
\draw [thin]  (12,1) -- (16,1);     \draw [thin]  (18,1) -- (22,1); 
\draw [thin]  (12,2) -- (16,2);     \draw [thin]  (18,2) -- (22,2); 
\draw [thin]  (12,3) -- (16,3);     \draw [thin]  (18,3) -- (22,3); 
\draw [thin]  (12,4) -- (16,4);     \draw [thin]  (18,4) -- (22,4); 
\draw [thick] (12,3) -- (14,1);     \draw [thick] (18,4) -- (19,3); 
\draw [thick] (14,4) -- (15,3);     \draw [thick] (19,2) -- (21,0); 
\draw [thick] (15,1) -- (16,0);     \draw [thick] (21,3) -- (22,2);                                
\end{tikzpicture}
\caption{Shuffles $D^1, D^2, D^3, D^4$ of $W$ as considered in Example \ref{PVC.PConst.Ex}.}
\label{fig.cube.1}
\vspace{-0.3cm}
\end{figure}
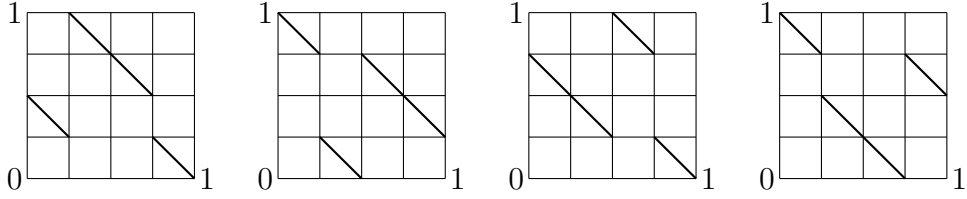
\begin{figure}[h]
\centering
\includegraphics[width=130mm]{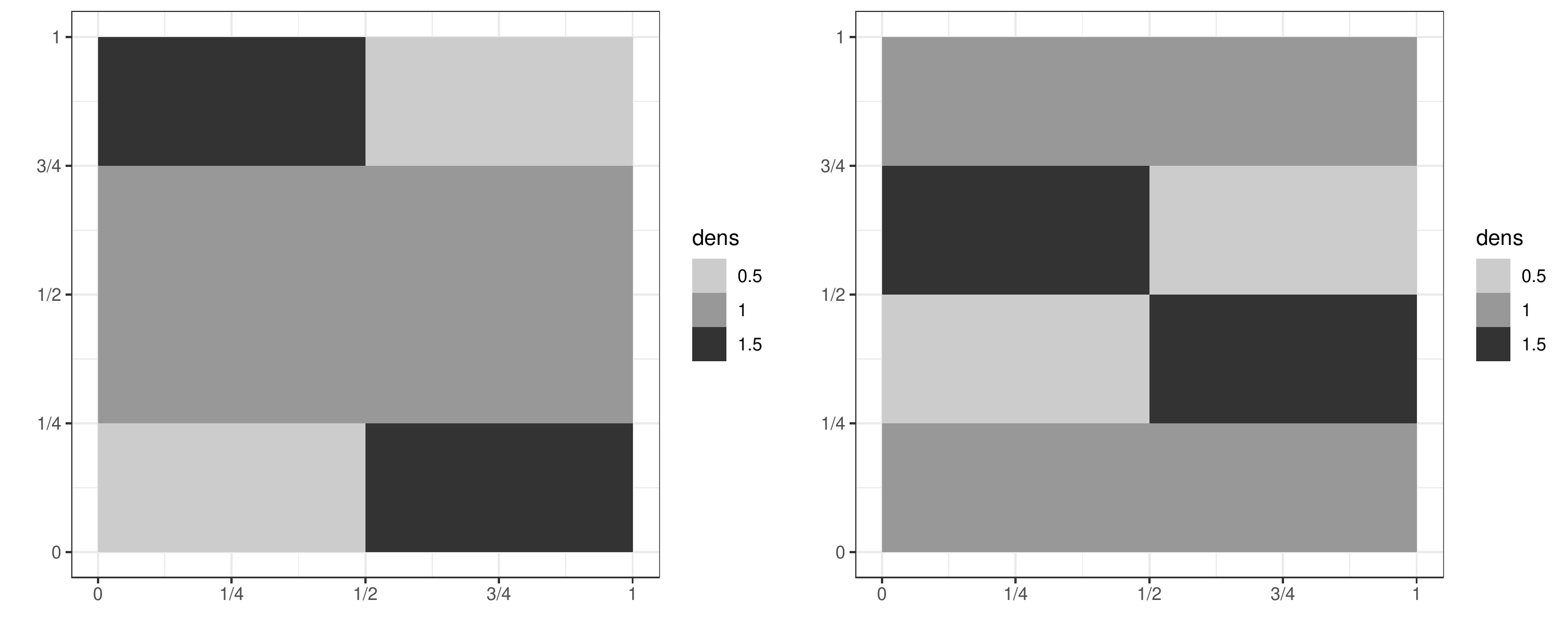}
\caption{Densities of the checkerboard copulas $B^\ast$ (left panel) and $B^{\ast\ast}$ (right panel).}
\label{fig:check}
\vspace{-0.3cm}
\end{figure}
As next step we construct the (conditional) univariate marginal distribution functions $F_{1|3}(.|t)$ and $F_{2|3}(.|t)$ (conditional on $t \in \I$) and proceed as follows: 
Let $B^\ast, B^{\ast\ast}$ denote bivariate checkerboard copulas (see \cite{trutschnig2015} for a definition) 
 whose densities $b^\ast, b^{\ast\ast}: \I^2 \to \R$ are given by
\begin{eqnarray*}
  b^\ast (u_1,t) 
	:= 
	\begin{cases}
	  \tfrac{1}{2} & (u_1,t) \in \big(0, \tfrac{1}{2}\big) \times I_1 \cup \big(\tfrac{1}{2},1\big) \times I_4
		\\
	  1            & (u_1,t) \in \I \times I_2 \cup \I \times I_3
		\\
	  \tfrac{3}{2} & (u_1,t) \in \big(\tfrac{1}{2},1\big) \times I_1 \cup \big(0,\tfrac{1}{2}\big) \times I_4 \\
	  0 & \textrm{ otherwise }
	\end{cases}
\end{eqnarray*}
and
\begin{eqnarray*}
	b^{\ast\ast} (u_2,t) 
	:= 
	\begin{cases}
	  \tfrac{1}{2} & (u_2,t) \in \big(0, \tfrac{1}{2}\big) \times I_2 \cup \big(\tfrac{1}{2},1\big) \times I_3
		\\
	  1            & (u_2,t) \in \I \times I_1 \cup \I \times I_4
		\\
	  \tfrac{3}{2} & (u_2,t) \in \big(\tfrac{1}{2},1\big) \times I_2 \cup \big(0,\tfrac{1}{2}\big) \times I_3 \\
	  0 & \textrm{ otherwise }
	\end{cases},
\end{eqnarray*}
respectively (see Figure \ref{fig:check}). Then the Markov kernels of $B^\ast$ and $B^{\ast\ast}$ obviously satisfy
\begin{eqnarray*}
  K_{B^\ast} (t,[0,0.5]) 
	= 
	\begin{cases}
	  \tfrac{1}{4} & t \in I_1
		\\
	  \tfrac{2}{4} & t \in I_2 
		\\
	  \tfrac{2}{4} & t \in I_3
		\\
	  \tfrac{3}{4} & t \in I_4
	\end{cases}
  \quad \textrm{ and } \quad
	K_{B^{\ast\ast}} (t,[0,0.5]) 
	= 
	\begin{cases}
	  \tfrac{2}{4} & t \in I_1
		\\
	  \tfrac{1}{4} & t \in I_2 
		\\
	  \tfrac{3}{4} & t \in I_3
		\\
		\tfrac{2}{4} & t \in I_4.
	\end{cases}
\end{eqnarray*}
 Completing the construction of $C$ we use the copulas $A^t$, $t \in \I$, as conditional copulas and 
 the Markov kernels $K_{B^\ast}$ and $K_{B^{\ast\ast}}$ as (conditional) univariate marginal distribution functions,
and set 
\begin{equation} 
  C(\uuu,v)
	:= \int\limits_{[0,v]} A^t \Big( K_{B^\ast} (t,[0,u_1]), K_{B^{\ast\ast}} (t,[0,u_2]) \Big) \;\mathrm{d}\lambda(t).
\end{equation}
 Then $C \in \mathcal{C}^3_c$ is non-simplified,
satisfies $C_{12;3}^t = A^t$ for all $t \in I^1 \cup I^2 \cup I^3 \cup I^4$,
$C_{13} = B^\ast$, $C_{23}=B^{\ast\ast}$, as well as 
\begin{eqnarray*}
  C \big(0.5,0.5,1\big) 
	& = & \int\limits_{\I} A^t \Big( K_{B^\ast} (t,[0,0.5]), K_{B^{\ast\ast}} (t,[0,0.5]) \Big) \;\mathrm{d}\lambda(t)
	\\*
	& = & \frac{1}{4} \; D^1 \left( \frac{1}{4}, \frac{2}{4} \right) 
			  + \frac{1}{4} \; D^2 \left( \frac{2}{4}, \frac{1}{4} \right) 
				+ \frac{1}{4} \; D^3 \left( \frac{2}{4}, \frac{3}{4} \right) 
				+ \frac{1}{4} \; D^4 \left( \frac{3}{4}, \frac{2}{4} \right)
	\\*
	& = & \frac{1}{4} \left(\frac{1}{4} + \frac{1}{4} + \frac{1}{2} + \frac{1}{2}\right) = \frac{3}{8}
\end{eqnarray*}
Considering that the partial copula $C_p$ of $C$ is given by $ C_p = \tfrac{1}{4} \big( D^1 + D^2 + D^3 + D^4 \big) $
the partial vine copula $\psi(C)$ of $C$ satisfies
\begin{eqnarray*}
		\lefteqn{\big(\psi(C)\big) \big(0.5,0.5,1\big)}
		\\
		& = & \int\limits_{\I} C_{p} \Big( K_{B^\ast} (t,[0,0.5]), K_{B^{\ast\ast}} (t,[0,0.5]) \Big) 
		      \;\mathrm{d}\lambda(t)
		\\
		& = & \frac{1}{4} \; C_{p} \left( \frac{1}{4}, \frac{2}{4} \right) 
			  + \frac{1}{4} \; C_{p} \left( \frac{2}{4}, \frac{1}{4} \right) 
				+ \frac{1}{4} \; C_{p} \left( \frac{2}{4}, \frac{3}{4} \right) 
				+ \frac{1}{4} \; C_{p} \left( \frac{3}{4}, \frac{2}{4} \right)
		\\
		& = & \frac{1}{16} \left( 
					D^1\left(\frac{1}{4},\frac{1}{2}\right) 
					+ D^2\left(\frac{1}{4},\frac{1}{2}\right) 
					+ D^3\left(\frac{1}{4},\frac{1}{2}\right) 
					+ D^4\left(\frac{1}{4},\frac{1}{2}\right) \right)
		\\
		&   & + \frac{1}{16} \left( 
					D^1\left(\frac{1}{2},\frac{1}{4}\right) 
					+ D^2\left(\frac{1}{2},\frac{1}{4}\right) 
					+ D^3\left(\frac{1}{2},\frac{1}{4}\right) 
					+ D^4\left(\frac{1}{2},\frac{1}{4}\right)\right)
		\\
		&   & + \frac{1}{16} \left(
					D^1\left(\frac{1}{2},\frac{3}{4}\right) 
					+ D^2\left(\frac{1}{2},\frac{3}{4}\right) 
					+ D^3\left(\frac{1}{2},\frac{3}{4}\right) 
					+ D^4\left(\frac{1}{2},\frac{3}{4}\right)\right)
		\\
		&   & + \frac{1}{16} \left(
					D^1\left(\frac{3}{4},\frac{1}{2}\right) 
					+ D^2\left(\frac{3}{4},\frac{1}{2}\right) 
					+ D^3\left(\frac{3}{4},\frac{1}{2}\right) 
					+ D^4\left(\frac{3}{4},\frac{1}{2}\right)\right)
		\\
		& = & \frac{1}{16} \left(\frac{1}{4} + \frac{1}{4} + \frac{5}{4} + \frac{5}{4}\right) =  \frac{3}{16},
\end{eqnarray*}	
  from which we get $ d_\infty(C, \psi(C)) \geq \tfrac{3}{16}$.
\end{example}{}
%\smallskip

We have therefore proved the following theorem:

%\smallskip
\begin{theorem}{} \label{PVC.Gen.Dist.}
There exists a copula $C \in \CC^3_{\rm c}$ fulfilling  $d_\infty(C, \psi(C)) \geq \tfrac{3}{16}$ and we have 
$$
  \sup_{C \in \CC^3_{\rm c}} d_\infty \big(C,\psi(C)\big) 	\geq \frac{3}{16}.
$$
\end{theorem}{}

%%%%%%%%%%%%%%%%%%%%%%%%%%%%%%%%%%%%%%%%%%%%%%%%%%%%%%%%%%%%%%%%%%%%%%%%%%%%%%%%%%%%%%%%%%%%%%%%%%%%%%%%%%%%%%%%%%%%%%%%%%%%%%%%%%%%%%%%%%%%%%%%%%%%%%%%%%%%%%%%%%%%%%%%%%%%%%%%%%%%%%%%%%%%%%%%%%%%%%%%%%%%%%%%%%%%%%%%%%%%%%%%%%%%%%%%%%%%%%%%%%%%%%%%%%%%%%%%%%%%%%%%%%%%%%%%%%%%%%%%%%%%%%%%%%%%%%%%%%%%%%%%%%%%%%%%%%%%%%%%%%%%%%%%%%%%%%%%%%%%%%%%%%%%%%%%%%%%%%%%%%%%%%%%%%%%%%%%%%%%%%%%%%%%%%%%%%%%%%%%%%%%%%%%%%%%%%%%%%%%%%%%%%%%%%%%%%%%%%%%%%%%%%%%%%%%%%%%%%%%%%%%%%%%%%%%%%%%%%%%%%%%%%%%%%%%%%%%%%%%%%%%%%%%%%%%%%%%%%%%%%%%%%%%%%%%%%%%%%%%%%%%%%%%%%%%%%%%%%%%%%%%%%%%%%%%%%%%%%%%%%%%%%%%%%%%%%%%%%%%%%%%%%%%%%%%%%%%%%%%%%%%%%%%%%%%%%%%%%%%%%%%%%%%%%%%%%%%%%%%%%%%%%%%%%%%%%%%%%%%%%%%%%%%%%%%%%%%%%%%%%%%
%\newpage

\section{Continuity of $\psi$}\label{Section.PVC.Cont.}
\label{Sect.Partial.Cont.}

In this section we discuss continuity properties of the mapping $\psi: \CC^3_{\rm c} \to \CC^3_{\rm c}$ assigning 
every $C \in \mathcal{C}^3_c$ its partial vine copula. Having in mind Lemma \ref{PVC.Lemma1} intuitively 
one might interpret $\psi$ as projection and therefore think that $\psi$ has to be continuous with respect to $d_\infty$. 
It turns out, however, that this interpretation is 
wrong, we will show that $\psi$ is \emph{not} continuous with respect to $d_\infty$. Considering stronger topologies than 
the one induced by $d_\infty$ changes the picture - we will prove that $\psi$ is continuous with respect to 
weak conditional convergence and with respect to the metric $D_1$ (under some mild regularity conditions).  

%%%%%%%%%%%%%%%%%%%%%%%%%%%%%%%%%%%%%%%%%%%%%%%%%%%%%%%%%%%%%%%%%%%%%%%%%%%%%%%%%%%%%%%%%%%%%%%%%%%%%%%%%%%%%%%%%%%%%%%%%%%%%%%%%%%%%%%%%%%%%%%%%%%%%%%%%%%%%%%%%%%%%%%%%%%%%%%%%%%%%%%%%%%%%%%%%%%%%%%%%%%%%%%%%%%%%%%%%%%%%%%%%%%%%%%%%%%%%%%%%%%%%%%%%%%%%%%%%%%%%%%%%%%%%%%%%%%%%%%%%%%%%%%%%%%%%%%%%%%%%%%%%%%%%%%%%%%%%%%%

\subsection{Uniform convergence}

The mapping $\psi$ is not continuous with respect to $d_\infty$ - the following result holds:
%\smallskip

\begin{theorem}{} \label{NonCont.dInf.Lemma}\label{NonCont.dInf}
Suppose that $C \in \CC_{\rm c}^3$ satisfies $d_\infty (C,\psi(C)) \neq 0$. Then $C$ is a discontinuity point of the 
the mapping $\psi: \CC^3_{\rm c} \to \CC^3_{\rm c}$. In other words: Every non-simplified $C \in \mathcal{C}_c^3$ is 
a discontinuity point of $\psi$.  
\end{theorem}{}%Es gilt sogar Gleichheit.
\begin{proof}{}
  Let $C$ be as in the theorem and set $\varepsilon:= d_\infty (C,\psi(C)) > 0$. Suppose that 
  $\mathbf{X}_1,\mathbf{X}_2, \ldots$ is an i.i.d. sample from $\mathbf{X} \sim C$ and let  
  $\widehat{C}_n$ denote the corresponding empirical copula. With probability one we have that 
  $\mathbf{X}_1,\mathbf{X}_2, \ldots$ has no ties and that $(\widehat{C}_n)_{n \in \mathbb{N}}$ converges to 
  $C$ with respect to $d_\infty$. Considering that empirical copulas are simplified 
  according to Theorem \ref{Emp.Subset.S} and using the triangle inequality it follows immediately that
\begin{eqnarray*}
   \varepsilon
	&   =  & d_\infty (C,\psi(C)) 
	\\
	& \leq & d_\infty \big( C, \psi \big(\widehat{C}_n\big) \big) + d_\infty \big( \psi \big(\widehat{C}_n\big), \psi(C) \big)
  \\
	&   =  & d_\infty \big( C, \widehat{C}_n \big) + d_\infty \big( \psi \big(\widehat{C}_n\big), \psi(C) \big)
\end{eqnarray*}
holds for every $n \in \N$. Consequently, since $\lim_{n \rightarrow \infty} d_\infty(\widehat{C}_n,C)=0$ 
$$
\liminf_{n \rightarrow \infty} d_\infty \big( \psi \big(\widehat{C}_n\big), \psi(C) \big) \geq \varepsilon 
$$
follows, implying that $\psi$ is not continuous at $C$. 
\end{proof}{}

Using convex combinations (of empirical copulas with a non-simplified co\-pula) 
it is straightforward to verify that the set of all $C \in \mathcal{C}^3_c$ that 
are non-simplified is dense in $(\mathcal{C}^3_c,d_\infty)$ - Theorem \ref{NonCont.dInf.Lemma} therefore has the 
following corollary: 
\begin{corollary}{} \label{discont.everywhere}
  The mapping $\psi: \CC^3_{\rm c} \to \CC^3_{\rm c}$ is discontinuous on a dense subset of $(\mathcal{C}^3_c,d_\infty)$. 
\end{corollary}{}

%%%%%%%%%%%%%%%%%%%%%%%%%%%%%%%%%%%%%%%%%%%%%%%%%%%%%%%%%%%%%%%%%%%%%%%%%%%%%%%%%%%%%%%%%%%%%%%%%%%%%%%%%%%%%%%%%%%%%%%%%%%%%%%%%%%%%%%%%%%%%%%%%%%%%%%%%%%%%%%%%%%%%%%%%%%%%%%%%%%%%%%%%%%%%%%%%%%%%%%%%%%%%%%%%%%%%%%%%%%%%%%%%%%%%%%%%%%%%%%%%%%%%%%%%%%%%%%%%%%%%%%%%%%%%%%%%%%%%%%%%%%%%%%%%%%%%%%%%%%%%%%%%%%%%%%%%%%%%%%%

\subsection{Weak conditional convergence}

Focusing on weak conditional convergence the mapping $\psi$ behaves more nicely:  

%\smallskip
\begin{theorem}{} \label{Cont.WCC}
Suppose that $C,C_1,C_2,\ldots$ are copulas in  $\CC_{\rm c}^3$. Then the following assertions hold:
\begin{onelist}
\item
$C_n \xrightarrow{\text{wcc}} C$
implies $(C_n)_{13} \xrightarrow{\text{wcc}} C_{13}$ and $(C_n)_{23} \xrightarrow{\text{wcc}} C_{23}$.

\item
$C_n \xrightarrow{\text{wcc}} C$
implies $(C_n)_p \xrightarrow{d_\infty} C_p$. 

\item
$C_n \xrightarrow{\text{wcc}} C$
implies $\psi(C_n) \xrightarrow{\text{wcc}} \psi(C)$.
\end{onelist}
\end{theorem}{}

\begin{proof}{}
  The first assertions follows from Theorem \ref{wcc.margins}. To prove the second one we proceed as follows:
  Since for almost all $v \in \I$ the marginal distribution functions of $K_{C_n}(v,.)$, $n \in \N$, and 
  of $K_{C}(v,.)$ are continuous, Lemma \ref{App.2} implies uniform convergence of the sequence
$((C_n)_{12;3}^v)_{n \in \N}$ to $C_{12;3}^v$. For $\sss \in \I^2$ we get 
\begin{eqnarray*}
  \big| (C_n)_{p} (\sss) - C_{p} (\sss) \big|
	&   =  & \biggl| \int\limits_{\I} (C_n)_{12;3}^t (\sss) - C_{12;3}^t (\sss) 
					 \; \mathrm{d} \lambda(t) \biggr|
	\\
	& \leq & \int\limits_{\I} \; \big| (C_n)_{12;3}^t (\sss) - C_{12;3}^t (\sss) \big|
					 \; \mathrm{d} \lambda(t) 
	\\
	& \leq & \int\limits_{\I} d_\infty \big( (C_n)_{12;3}^t, C_{12;3}^t \big)
					 \; \mathrm{d} \lambda(t) 
\end{eqnarray*}
and dominated convergence yields
\begin{eqnarray*}
  \lim_{n \to \infty} d_{\infty} \big( (C_n)_{p}, C_{p} \big) = 0
\end{eqnarray*}
  To prove the last assertion notice that for almost all $t \in \I$ 
  we have $(\psi(C))_{12;3}^t = C_p$ as well as $(\psi(C_n))_{12;3}^t = (C_n)_p$ for every $n \in \N$. Hence, using the 
  second assertion it follows that 
$$
  \lim_{n \to \infty} d_{\infty} \big( (\psi(C_n))_{12;3}^t, (\psi(C))^t_{12;3} \big)
	= \lim_{n \to \infty} d_{\infty} \big( (C_n)_{p}, C_{p} \big)
	= 0
$$
holds for almost all $t \in \I$. According to Lemma \ref{App.2} it now suffices to show that the marginal distribution 
functions of the Markov kernels converge weakly, which is, however, an immediate consequence of the fact that 
$(\psi(C))_{i3} = C_{i3}$ and $(\psi(C_n))_{i3} = (C_n)_{i3}$, $i \in \{1,2\}$ 
holds for every $n \in \N$ (see Lemma \ref{PVC.Lemma1}). 
\end{proof}{}

%%%%%%%%%%%%%%%%%%%%%%%%%%%%%%%%%%%%%%%%%%%%%%%%%%%%%%%%%%%%%%%%%%%%%%%%%%%%%%%%%%%%%%%%%%%%%%%%%%%%%%%%%%%%%%%%%%%%%%%%%%%%%%%%%%%%%%%%%%%%%%%%%%%%%%%%%%%%%%%%%%%%%%%%%%%%%%%%%%%%%%%%%%%%%%%%%%%%%%%%%%%%%%%%%%%%%%%%%%%%%%%%%%%%%%%%%%%%%%%%%%%%%%%%%%%%%%%%%%%%%%%%%%%%%%%%%%%%%%%%%%%%%%%%%%%%%%%%%%%%%%%%%%%%%%%%%%%%%%%%

\subsection{Convergence with respect to $D_1$}

  We finally discuss $D_1$-continuity. Similar to the proof of Theorem \ref{Cont.WCC},
we first relate $D_1$-convergence of copulas to uniform convergence of the corresponding partial copulas. 
The slightly technical (but straightforward) proof of the following useful lemma is deferred to the appendix:

%\smallskip
\begin{lemma}{} \label{Cont.D1.Lemma1}
Suppose that $C,C_1,C_2,\ldots$ are copulas in  $\CC_{\rm c}^3$. Then the following assertions hold:
\begin{onelist}
\item
  $C_n \xrightarrow{D_1} C$ implies
$(C_n)_{13} \xrightarrow{D_1} C_{13}$ and $(C_n)_{23} \xrightarrow{D_1} C_{23}$.

\item
  $(C_n)_p \xrightarrow{d_\infty} C_p$,
$(C_n)_{13} \xrightarrow{D_1} C_{13}$ and $(C_n)_{23} \xrightarrow{D_1} C_{23}$ imply $\psi(C_n) \xrightarrow{D_1} \psi(C)$.
%The inequality 
%$$ D_1 \big( \psi (C_n), \psi(C) \big)
%	\leq D_1 \big( (C_n)_{13}, C_{13} \big) + D_1 \big( (C_n)_{23}, C_{23} \big)
%           + d_{\infty} \big( (C_n)_{p}, C_{p} \big) $$
%holds for every $n \in \N$.
\end{onelist}
\end{lemma}
%\smallskip

  We now show $D_1$-continuity of the mapping $\psi$ on the subclass of absolutely continuous copulas satisfying 
  some integrability condition. The following lemma whose proof is deferred to the appendix will be key for proving this 
  result: 

%\smallskip
\begin{lemma}{} \label{Cont.D1.Lemma2}
  Suppose that $C,C_1,C_2,\ldots$ are copulas in  $\CC_{\rm c}^3$, that $C$ is absolutely continuous  
  and let $c_{13},c_{23}$ denote the densities of the marginal copulas $C_{13},C_{23}$ of $C$. 
  If there exist some constants $p_{13}, p_{23}, p_{123} \in (1,\infty)$ such that 
$$
  \| c_{13} \|_{p_{13}} < \infty,
	\qquad 
	\| c_{23} \|_{p_{23}} < \infty,
	\qquad 
	\| c_{13} \, c_{23} \|_{p_{123}} < \infty
$$
holds then $C_n \xrightarrow{D_1} C$ implies
$(C_n)_p \xrightarrow{d_\infty} C_p$.
\end{lemma}{} 
%\smallskip

Combining the previous two lemmata yields continuity of $\psi$ with respect to $D_1$ under some mild regularity conditions:
\begin{theorem}{} \label{Cont.D1}
  Consider a sequence of copulas $(C_{n})_{n \in \N}$ in $\CC_{\rm c}^3$ and an absolutely continuous copula 
  $C \in \CC_{\rm c}^3$, and let $c_{13},c_{23}$ denote the densities of the marginal copulas $C_{13},C_{23}$ of $C$, respectively. If there exist some constants $p_{13}, p_{23}, p_{123} \in (1,\infty)$ such that 
$$
  \| c_{13} \|_{p_{13}} < \infty,
	\qquad 
	\| c_{23} \|_{p_{23}} < \infty,
	\qquad 
	\| c_{13} \, c_{23} \|_{p_{123}} < \infty
$$
holds then $C_n \xrightarrow{D_1} C$ implies
$\psi(C_n) \xrightarrow{D_1} \psi(C)$.
\end{theorem}{}

%%%%%%%%%%%%%%%%%%%%%%%%%%%%%%%%%%%%%%%%%%%%%%%%%%%%%%%%%%%%%%%%%%%%%%%%%%%%%%%%%%%%%%%%%%%%%%%%%%%%%%%%%%%%%%%%%%%%%%%%%%%%%%%%%%%%%%%%%%%%%%%%%%%%%%%%%%%%%%%%%%%%%%%%%%%%%%%%%%%%%%%%%%%%%%%%%%%%%%%%%%%%%%%%%%%%%%%%%%%%%%%%%%%%%%%%%%%%%%%%%%%%%%%%%%%%%%%%%%%%%%%%%%%%%%%%%%%%%%%%%%%%%%%%%%%%%%%%%%%%%%%%%%%%%%%%%%%%%%%%%%%%%%%%%%%%%%%%%%%%%%%%%%%%%%%%%%%%%%%%%%%%%%%%%%%%%%%%%%%%%%%%%%%
%\newpage

\section{Results for arbitrary dimension}\label{Arb.Dim.}

To confirm that the case of dimension three is similar to higher dimension in this section we extend (slightly modified versions of) our main results 
(Theorem \ref{Emp.Subset.S}, Corollary \ref{S.Dense}, Theorem \ref{PVC.NonOpt}, Theorem \ref{PVC.Gen.Dist.},
Theorem \ref{NonCont.dInf.Lemma} and Corollary \ref{discont.everywhere}) to arbitrary dimensions. 

%%%%%%%%%%%%%%%%%%%%%%%%%%%%%%%%%%%%%%%%%%%%%%%%%%%%%%%%%%%%%%%%%%%%%%%%%%%%%%%%%%%%%%%%%%%%%%%%%%%%%%%%%%%%%%%%%%%%%%%%%%%%%%%%%%%%%%%%%%%%%%%%%%%%%%%%%%%%%%%%%%%%%%%%%%%%%%%%%%%%%%%%%%%%%%%%%%%%%%%%%%%%%%%%%%%%%%%%%%%%%%%%%%%%%%%%%%%%%%%%%%%%%%%%%%%%%%%%%%%%%%%%%%%%%%%%%%%%%%%%%%%%%%%%%%%%%%%%%%%%%%%%%%%%%%%%%%%%%%%%

\subsection{Simplified copulas.} 

Using disintegration for every copula $C \in \CC^d$, every $J \subseteq \{1,\dots,d\}$ with $2 \leq |J| \leq d$ and 
every $L \subseteq J$ with $1 \leq |L| \leq |J|-2$, 
there exists some Markov kernel $K_{C_J}$ such that the lower dimensional marginal copula $C_J$ of $C$ corresponding 
to the indices of the coordinates of $C$ belonging to $J$ can be expressed as
$$
  C_J(\uuu)
	= \int\limits_{[{\bf 0},\uuu_L]} K_{C_J} (\ttt, [\nul,\uuu_{J \backslash L}]) \;\mathrm{d}\mu_{C_L}(\ttt)
$$
for all $\uuu \in \I^{|J|}$. Thereby $\uuu_L \in \I^{|L|}$ denotes the vector of coordinates of $\uuu$ belonging to $L$,
and $\uuu_{J \backslash L} \in \I^{|J \setminus L|}$ the vector of coordinates of $\uuu$ belonging to $J \backslash L$.
 Since $K_{C_J}$ is a Markov kernel, for every $\uuu_{J \backslash L} \in \I^{|J \backslash L|}$ 
 the mapping $ \ttt \mapsto K_{C_J} (\ttt, [\nul,\uuu_{J \backslash L}]) $ is measurable and, 
for $\mu_{C_L}$-almost every $\ttt \in \I^{|L|}$, the mapping $ \uuu_{J \backslash L} \mapsto$ 
$K_{C_J} (\ttt, [\nul,\uuu_{J \backslash L}]) $ is a multivariate distribution function with (\emph{conditional}) univariate marginal distribution functions 
$F_{j|L}(.|\ttt)$, $j \in J \backslash L$, (conditional on $\ttt$). 
  By Sklar's theorem we get that for almost every $\ttt \in \I^{|L|}$ there exists some (\emph{conditional}) 
  copula $C_{J \backslash L;L}^\ttt$ (conditional on $\ttt$) satisfying
$$
  K_{C_J} (\ttt, [\nul,\uuu_{J \backslash L}])
	= C_{J \backslash L;L}^\ttt \big( F_{j_1|L} (u_{j_1}|\ttt), \dots, F_{j_{|J\backslash L|}|L} (u_{j_{|J\backslash L|}}|\ttt) \big) 
$$
for all $\uuu_{J \backslash L} = (u_{j_1}, \dots, u_{j_{|J\backslash L|}}) \in \I^{|J \backslash L|}$ 
such that the identity
\begin{equation*} %\label{Eq.Cop.1}
  C_J(\uuu)
	= \int\limits_{[{\bf 0},\uuu_L]} C_{J \backslash L;L}^\ttt \big( F_{j_1|L} (u_{j_1}|\ttt), \dots, F_{j_{|J\backslash L|}|L} (u_{j_{|J\backslash L|}}|\ttt) \big)   
	  \;\mathrm{d}\mu_{C_L}(\ttt)
\end{equation*}
holds for all $\uuu \in \I^{|J|}$ .

We will refer to a copula $C \in \CC^d$ as \emph{universally simplified} if for every 
$J \subseteq \{1,\dots,d\}$ with $2 \leq |J| \leq d$ and 
every $L \subseteq J$ with $1 \leq |L| \leq |J|-2$ the following properties hold:
\begin{onelist}
\item[(U1) ]
There exists some copula $A \in \CC^{|J \backslash L|}$ such that the identity 
\begin{equation} \label{D.Eq.Cop.1}
  C_J(\uuu)
	= \int\limits_{[{\bf 0},\uuu_L]} A \big( F_{j_1|L} (u_{j_1}|\ttt), \dots, F_{j_{|J\backslash L|}|L} (u_{j_{|J\backslash L|}}|\ttt) \big)
	  \;\mathrm{d}\mu_{C_L}(\ttt)
\end{equation}
holds for all $\uuu \in \I^{|J|}$.

\item[(U2) ]
The (conditional) univariate marginal distribution functions $F_{j|L}(.|\ttt)$, $j \in J \backslash L$,
are continuous for $\mu_{C_L}$-almost all $\ttt \in \I^{|L|}$.
\end{onelist}
Notice that every universally simplified three-dimensional copula is simplified in the sense studied in the last sections but not necessarily vice versa.
If $C \in \CC^d$ is universally simplified then Sklar's theorem implies that the (conditional) copulas 
$C_{J \backslash L;L}^\ttt$ are unique for $\mu_{C_L}$-almost all $\ttt \in \I^{|L|}$. 
In what follows we will let $\CC^d_{\rm c}$ denote the family of all $d$-dimensional copulas having continuous 
(conditional) univariate marginal distribution functions, $\CC^d_{\rm US}$ will denote the family 
of all $d$-dimensional universally simplified copulas. 
Notice that $\Pi \in \CC^d_{\rm US}$ and that the collection of all absolutely continuous copulas is 
contained in $\CC^d_{\rm c}$.

%\bigskip
As first step we now prove a sharper version of Theorem \ref{Emp.Subset.S} and show that all $d$-variate empirical copulas ($d$-linear interpolations) are universally simplified. 
%\smallskip
\begin{theorem}{} \label{D.Emp.Subset.S}
  Every $d$-dimensional empirical copula is universally simplified. 
\end{theorem}{}

\begin{proof}{} 
Suppose that $\XXX$ is a $d$-dimensional random vector with continuous univariate marginals and suppose that 
$\mathbf{X}_1,\ldots, \mathbf{X}_n$ is a sample from $\XXX$. W.l.o.g. assume that there are no ties.
Letting $\hat{C}_n$ denote the ($d$-linear interpolation of the) empirical copula there exists unique 
permutations $\sigma_1, \dots, \sigma_{d-1}$ of $\{1,\ldots,n\}$ such that the density 
$\hat{c}_n$ of $\hat{C}_n$ is given by (uniform distribution on $n$ $d$-dimen\-sional squares of volume $\frac{1}{n^d}$)
$$
  \hat{c}_n(\uuu) 
	= n^{d-1} \sum_{i=1}^n \left( \prod_{j=1}^{d} \mathds{1}_{I_i^j}(u_j) \right), \,\, \uuu=(u_1,\ldots,u_d) 
	\in \mathbb{I}^d,
$$
where 
$I^j_i=(\frac{\sigma_j(i)-1}{n},\frac{\sigma_j(i)}{n}]$, $j \in \{1, \dots, d-1\}$, and
$I_i^d=(\frac{i-1}{n},\frac{i}{n}]$ for every $i \in \{1,\ldots,n\}$. 
Since marginals of empirical copulas are empirical copulas too it suffices to prove the result for $J=\{1,\ldots,d\}$ and for $1 \leq l \leq d-2$ with $L=\{d-l+1,\ldots,d\}$. Considering that the $l$-dimensional marginal copula
$(\hat{C}_n)_L$ of $(\hat{C}_n)$ assigns full mass to the set
$$
\bigcup_{i=1}^n \underbrace{\left(\bigtimes_{j=d-l+1}^d I_i^j \right)}_{=:\Omega_i}
$$
it is enough to consider $\mathbf{u}_L \in \Omega_{i_0}$ for some $i_0 \in \{1,\ldots,n\}$. 
For such $\mathbf{u}_L$ the Markov kernel (conditioning on the coordinates in $L$) is given by (straightforward consequence 
of first considering the conditional density)
$$
  K_{\hat{C}_n}(\mathbf{u}_L,[0,u_1]\times \cdots \times [0,u_{d-l}]) 
	= n^{d-l} \left( \prod_{j=1}^{d-l} \int_{[0,u_j]} \mathds{1}_{I_{i_0}^j}(x_j) \; \mathrm{d} \lambda(x_j) \right)
$$
and the conditional univariate distribution functions $F_{j\vert L} (u_j \vert \mathbf{u}_L)$ for every $j \in J \setminus L$
can be expressed as
$$
F_{j\vert L} (u_j \vert \mathbf{u}_L) = n \int_{[0,u_j]} \mathds{1}_{I_{i_0}^j}(x_j)\; \mathrm{d} \lambda(x_j).
$$
Having this we have shown 
$$
K_{\hat{C}_n}(\mathbf{u}_L,[0,u_1]\times \cdots \times [0,u_{d-l}]) = \Pi_{d-l} \left(F_{1\vert L} (u_1 \vert \mathbf{u}_L),
 \ldots, F_{d-l\vert L} (u_{d-l} \vert \mathbf{u}_L) \right),
$$
which completes the proof.
\end{proof}{}

  Since the collection of all empirical copulas is dense in $(\CC^d, d_{\infty})$
(\cite[Proposition 3.2]{durantefs2010}) Theorem \ref{D.Emp.Subset.S} has the following immediate consequence:

%\smallskip
\begin{corollary}{} \label{D.S.Dense}
  The collection of all universally simplified $d$-dimensional co\-pulas is dense in $(\CC^d, d_{\infty})$. 
\end{corollary}{}
%\smallskip
 Thus, every copula can be approximated arbitrarily well by universally simplified ones. 
 Given a $d$-dimensional, non universally simplified copula $C$, a good uniform approximation by a universally 
simplified one from the same (Fr{\'e}chet) class might, however, not exist. 
The next example illustrates this fact: 

%\smallskip
\begin{example}{} \label{D.SVC.PInd.2}
(Family $\mathcal{F}^d_{\rm Ind}$ of all copulas $C$ satisfying 
$C(\uuu,\vvv) = B(\uuu) \, \Pi(\vvv)$ for all $(\uuu,\vvv) \in \I^3 \times \I^{d-3}$
and some $B \in \CC^3_{c}$ fulfilling $B_{12}=B_{13}=B_{23}=\Pi$.) \\
First, notice that for every universally simplified copula $D$ in $\mathcal{F}^d_{\rm Ind}$ there exists some 
co\-pula $A \in \CC^2$ such that, 
according to Equation (\ref{D.Eq.Cop.1})
($J = \{1,\dots, d\}$ and $L = \{3,\dots,d\}$),
the identity 
\begin{eqnarray*}
  D(\uuu,\vvv)
  & = & \int\limits_{[{\bf 0},\vvv]} A \big( F_{1|L} (u_{1}|\ttt), F_{2|L} (u_{2}|\ttt) \big)
	  \;\mathrm{d}\mu_{C_L}(\ttt)
	\\
	& = & \int\limits_{[{\bf 0},\vvv]} A (u_{1},u_{2})
	  \;\mathrm{d}\lambda^{d-2}(\ttt)
  \\
	& = & A(\uuu) \, \Pi(\vvv)
\end{eqnarray*}
holds for all $(\uuu,\vvv) \in \I^2 \times \I^{d-2}$. Notice that the second equality holds since in case of 
$\mathbf{X} \sim D$ we have that $(X_1,X_3)$ and $(X_4,\ldots,X_d)$ are independent, hence
\begin{eqnarray*}
\mathbb{P}\left(X_1 \leq u_1, X_3 \leq u_3,\ldots,X_d \leq u_d \right) &=&
\mathbb{P}\left(X_1 \leq u_1, X_3 \leq u_3\right) \\
&& \,\, \cdot \,\, \mathbb{P} \left(X_4 \leq u_4,\ldots,X_d \leq u_d \right) \\
&=& \mathbb{P}\left(X_1 \leq u_1\right) \mathbb{P} \left(X_3 \leq u_3,\ldots,X_d \leq u_d \right)
\end{eqnarray*}
for all $\mathbf{u} \in (0,1)^d$, so $X_1$ and $(X_3,\ldots,X_d)$ are independent and we get  
$F_{1|L} (u_{1}|\ttt) = u_1$ (the same reasoning applies to $F_{2|L}$). \\
%If we now consider $B=C^{\rm Cube} \in \CC_{\rm c}^3$ with $C^{\rm Cube}$ as introduced in 
%Example \ref{SVC.PInd.2} 
%Then 
%every $d$-dimensional universally simplified $D$ in $\mathcal{F}^d_{\rm Ind}$ 
%can be expressed in the following way:
%\begin{eqnarray}
% D(u_1,\ldots,u_d) &=& C^{\rm Cube}(u_1,u_2,u_3) \cdot \prod_{j=4}^d u_j \\
% D(u_1,\ldots,u_d) &=& A(u_1,u_2) \cdot \prod_{j=3}^d u_j 
%\end{eqnarray}
Setting $C(\uuu,\vvv) = C^{\rm Cube} (\uuu) \, \Pi(\vvv)$ for all $(\uuu,\vvv) \in \I^3 \times \I^{d-3}$ and 
considering Example \ref{SVC.PInd.2} it therefore follows that $d_\infty(C,D) > \varepsilon$ (with $\varepsilon$
 as in Example \ref{SVC.PInd.2}), i.e., it is not possible to 
approximate $C$ by universally simplified copulas in $\mathcal{F}^d_{\rm Ind}$ with an error smaller than $\varepsilon$. 
\end{example}{}

%%%%%%%%%%%%%%%%%%%%%%%%%%%%%%%%%%%%%%%%%%%%%%%%%%%%%%%%%%%%%%%%%%%%%%%%%%%%%%%%%%%%%%%%%%%%%%%%%%%%%%%%%%%%%%%%%%%%%%%%%%%%%%%%%%%%%%%%%%%%%%%%%%%%%%%%%%%%%%%%%%%%%%%%%%%%%%%%%%%%%%%%%%%%%%%%%%%%%%%%%%%%%%%%%%%%%%%%%%%%%%%%%%%%%%%%%%%%%%%%%%%%%%%%%%%%%%%%%%%%%%%%%%%%%%%%%%%%%%%%%%%%%%%%%%%%%%%%%%%%%%%%%%%%%%%%%%%%%%%%

\subsection{Partial vine copulas (PVC-D)}

We finally introduce \emph{partial vine copulas} (PVC) belonging to a $D$-vine structure.
%and follow \cite{spanhel2019}.

%\bigskip
Given $C \in\CC^d_{\rm c}$ the hierarchical construction of a partial vine copula $C^{\rm PVC}$ of $D$-vine structure 
may be build as follows:
\begin{onelist}
\item[(S1) ] 
the bivariate marginal copulas $\big(C^{\rm PVC}\big)_{i,i+1}$ are defined as  
$$\big(C^{\rm PVC}\big)_{i,i+1} := C_{i,i+1}$$
for all $i \in \{1,\dots,d-1\}$.

\item[(S2) ] 
In the second step the bivariate partial copulas $(C_{p})_{i,i+2}$, $i \in \{1,\dots,d-2\}$ are defined by 
$$
  (C_{p})_{i,i+2} (u_i,u_{i+2})
	:= \int\limits_{\I} 
				 C_{i,i+2;i+1}^{w_{i+1}} (u_i,u_{i+2}) \; \mathrm{d}\lambda(w_{i+1})
$$
and the three-dimensional marginal copulas $\big(C^{\rm PVC}\big)_{i,i+1,i+2}$ with \\
 $i \in \{1,\dots,d-2\}$, are 
constructed via 
\begin{eqnarray*}
  \lefteqn{\big(C^{\rm PVC}\big)_{i,i+1,i+2} (u_i,u_{i+1},u_{i+2})}
	\\
	& := & \int\limits_{[0,u_{i+1}]} 
				 (C_{p})_{i,i+2} \big( \big(F^{\rm PVC}\big)_{i|i+1} (u_{i}|t_{i+1}), \big(F^{\rm PVC}\big)_{i+2|i+1} (u_{i+2}|t_{i+1}) \big)  
				\; \mathrm{d}\lambda(t_{i+1})
\end{eqnarray*}
where the (conditional) univariate distribution functions 
$(F^{\rm PVC})_{i|i+1}$ and $(F^{\rm PVC})_{i+2|i+1}$ correspond to the copulas 
$(C^{\rm PVC})_{i,i+1}$ and $(C^{\rm PVC})_{i+1,i+2}$ from previous steps.

\item[(S3) ] 
The bivariate partial copulas $(C_{p})_{i,i+3}$ with $i \in \{1,\dots,d-3\}$ are defined by 
$$
  (C_{p})_{i,i+3} (u_i,u_{i+3})
	:= \int\limits_{\I^2} 
		 C_{i,i+3;i+1,i+2}^{(w_{i+1},w_{i+2})} (u_i,u_{i+3}) 
		 \; \mathrm{d} \mu_{(C^{\rm PVC})_{i+1,i+2}} (w_{i+1},w_{i+2})
$$
and
the $4$-dimensional marginal copulas $\big(C^{\rm PVC}\big)_{i,i+1,i+2,i+3}$ with \\
 $i \in \{1,\dots,d-3\}$ are 
constructed via 
\begin{eqnarray*}
  \lefteqn{\big(C^{\rm PVC}\big)_{i,i+1,i+2,i+3} (u_i,u_{i+1},u_{i+2},u_{i+3})}
	\\*
	& := & \int\limits_{[0,u_{i+1}] \times [0,u_{i+2}]} 
				 (C_{p})_{i,i+3} \big( \big(F^{\rm PVC}\big)_{i|i+1,i+2} (u_{i}|(t_{i+1},t_{i+2})), \big(F^{\rm PVC}\big)_{i+3|i+1,i+2} (u_{i+3}|(t_{i+1},t_{i+2})) \big)  
	\\*
	&    & \qquad \; \mathrm{d} \mu_{(C^{\rm PVC})_{i+1,i+2}} (t_{i+1},t_{i+2}),
\end{eqnarray*}
where the (conditional) univariate distribution functions 
$(F^{\rm PVC})_{i|i+1,i+2}$ and $(F^{\rm PVC})_{i+3|i+1,i+2}$ correspond to the three-dimensional copulas 
$(C^{\rm PVC})_{i,i+1,i+2}$ and $(C^{\rm PVC})_{i+1,i+2,i+3} $ from previous steps.

\item[(S4) ] 
The individual steps are continued until one obtaines the $d$-dimensional partial vine copula $C^{\rm PVC}$ 
of $D$-vine structure.
\end{onelist}
The mapping induced by the afore-mentioned procedure will be denoted 
by $\psi: \CC^d_{\rm c} \to \CC^d_{\rm c}$, by construction it fulfills 
$$
 \psi(C) = C^{\rm PVC}.
$$
Notice that, by definition, $C^{\rm PVC}$ is simplified with respect to the underlying $D$-vine structure but may fail to be universally simplified.

%\smallskip
\begin{example}{} \label{D.SVC.PInd.3} 
	We calculate $\psi(C)$ for the $d$-dimensional copula $C \in \CC^d_{\rm c}$ given by
$$
  C(\uuu,\vvv)
	:= C^{\rm Cube}(\uuu) \, \Pi(\vvv)
$$
for all $(\uuu,\vvv) \in \I^3 \times \I^{d-3}$ (see Example \ref{D.SVC.PInd.2}) and show that 
$\psi(C) = \Pi$ holds. \\
We start with the following observations: 
\begin{hylist}
\item 
$C$ satisfies $C_J = \Pi$ for all $J \subseteq \{1,\dots,d\}$ with $2 \leq |J| \leq d-1$ and $\{1,2,3\} \not \subseteq J$.

\item 
$C$ is absolutely continuous.

\item
the Markov kernel of $C_{\{1,\dots,j\}}$, $j \in \{4,\dots,d\}$, with respect to the coordinates $2,\dots,j-1$ 
satisfies %($c^{\rm Cube}$ denotes the density of $C^{\rm Cube}$)
\begin{eqnarray*}
  \lefteqn{K_{C_{\{1,\dots,j\}}} \big( (t_2, \dots, t_{j-1}), [0,u_1] \times [0,u_j] \big)}
	%\\
	%& = & c^{\rm Cube}(u_1,t_2,t_3) \, u_j
	\\
	& = & \begin{cases}
		\min{(2u_1,1)} \, u_j 	& \textrm{ if } (t_2,t_3) \in \big( 0,\tfrac{1}{2} \big)^2 \cup \big( \tfrac{1}{2},1 \big)^2
					\\
		\max{(2 u_1-1,0)} \, u_j 		& \textrm{ if } (t_2,t_3) \in \big( 0,\tfrac{1}{2} \big) \times \big( \tfrac{1}{2},1 \big) \cup \big( \tfrac{1}{2},1 \big) \times \big( 0,\tfrac{1}{2} \big)
				\end{cases}
\end{eqnarray*}
for almost all $(t_2, \dots, t_{j-1}) \in \I^{j-2}$, hence $C_{1,j;2,\dots,j-1} = \Pi$ and it follows that 
the partial copula $(C_p)_{1,j ; 2,\dots,j-1}$ coincides with $\Pi$. 
\end{hylist}
We now calculate the partial vine copula $\psi(C)$ step-by-step:
\begin{onelist}
\item[(S1) ]
In the first step we obtain
$$
  \big(C^{\rm PVC}\big)_{i,i+1}
  = \Pi		
$$
for all $i \in \{1, \dots, d-1\}$.

\item[(S2) ]
The bivariate partial copulas satisfy
$$
  (C_{p})_{1,3} (u_1,u_{3})
	= \int\limits_{\I} 
		C_{1,3;2}^{w_{2}} (u_1,u_{3}) \; \mathrm{d}\lambda(w_{2})
	= \tfrac{1}{2} A^1 (u_1,u_{3}) + \tfrac{1}{2} A^2 (u_1,u_{3})
	= \Pi (u_1,u_{3})
$$
for all $(u_1,u_{3}) \in \I^2$ and for every $i \in \{2,\dots,d-2\}$ we get 
\begin{eqnarray*}
  (C_{p})_{i,i+2} (u_i,u_{i+2})
	&=& \int\limits_{\I} C_{i,i+2;i+1}^{w_{i+1}} (u_i,u_{i+2}) \; \mathrm{d}\lambda(w_{i+1})
	= \int\limits_{\I} \Pi (u_i,u_{i+2}) \; \mathrm{d}\lambda(w_{i+1}) \\
	&=& \Pi (u_i,u_{i+2})
\end{eqnarray*}
for all $(u_i,u_{i+2}) \in \I^2$.
Therefore, the $3$-dimensional marginal copulas satisfy
\begin{eqnarray*}
  \lefteqn{\big(C^{\rm PVC}\big)_{i,i+1,i+2} (u_i,u_{i+1},u_{i+2})}
	\\
	& = & \int\limits_{[0,u_{i+1}]} 
		    (C_{p})_{i,i+2} \big( \big(F^{\rm PVC}\big)_{i|i+1} (u_{i}|t_{i+1}), \big(F^{\rm PVC}\big)_{i+2|i+1} (u_{i+2}|t_{i+1}) \big)  
		    \; \mathrm{d}\lambda(t_{i+1})
	\\*
	& = & \int\limits_{[0,u_{i+1}]} 
		    (C_{p})_{i,i+2} (u_{i},u_{i+2}) \; \mathrm{d}\lambda(t_{i+1})
	\\*
	& = & \Pi (u_{i},u_{i+1},u_{i+2}) 
\end{eqnarray*}
for all $(u_i,u_{i+1},u_{i+2}) \in \I^3$ %and
%$$
%  \big(C^{\rm PVC}\big)_{i,i+1,i+2} (u_i,u_{i+1},u_{i+2})
%	= (C_{p})_{i,i+2} (u_{i},u_{i+2}) \, u_{i+1}
%	%= \Pi (u_{i},u_{i+2}) \, u_{i+1}
%	= \Pi (u_i,u_{i+1},u_{i+2})
%$$
%holds for all $(u_1,u_2,u_{3}) \in \I^3$ 
and every $i \in \{1,\dots,d-2\}$.

\item[(S3) ]
In the third step the bivariate partial copulas satisfy
\begin{eqnarray*}
  (C_{p})_{i,i+3} (u_i,u_{i+3})
	& = & \int\limits_{\I^2} 
		    C_{i,i+3;i+1,i+2}^{(w_{i+1},w_{i+2})} (u_i,u_{i+3}) 
		    \; \mathrm{d} \mu_{(C^{\rm PVC})_{i+1,i+2}} (w_{i+1},w_{i+2})
	\\
	& = & \int\limits_{\I^2} 
		    \Pi (u_i, u_{i+3}) 
		    \; \mathrm{d} \lambda^2 (w_{i+1},w_{i+2})
	\\
	& = & \Pi (u_i, u_{i+3}) 
\end{eqnarray*}
for all $(u_i,u_{i+3}) \in \I^2$ and for every $i \in \{1,\dots,d-3\}$. Hence 
follows that the $4$-dimensional marginal copulas satisfy
\begin{eqnarray*}
  \lefteqn{\big(C^{\rm PVC}\big)_{i,i+1,i+2,i+3} (u_i,u_{i+1},u_{i+2},u_{i+3})}
	\\
	& = & \int\limits_{[0,u_{i+1}] \times [0,u_{i+2}]} 
				 (C_{p})_{i,i+3} \big( \big(F^{\rm PVC}\big)_{i|i+1,i+2} (u_{i}|(t_{i+1},t_{i+2})), \big(F^{\rm PVC}\big)_{i+3|i+1,i+2} (u_{i+3}|(t_{i+1},t_{i+2})) \big)  
	\\
	&   & \qquad \; \mathrm{d} \mu_{(C^{\rm PVC})_{i+1,i+2}} (t_{i+1},t_{i+2})
  \\
	& = & \int\limits_{[0,u_{i+1}] \times [0,u_{i+2}]} (C_{p})_{i,i+3} (u_{i}, u_{i+3})  
        \; \mathrm{d} \lambda^2 (t_{i+1},t_{i+2})
	\\
	& = & \Pi (u_i,u_{i+1},u_{i+2},u_{i+3})
\end{eqnarray*}
for all $(u_1,u_2,u_{3},u_4) \in \I^4$.

\item[(S4) ]
Continuing in the same manner we finally arrive at $\psi(C) = C^{\rm PVC} = \Pi$.
\end{onelist}
\end{example}{}
%\smallskip

\begin{remark}{} \label{Rem.SpanhelKurz}
Notice that the afore-mentioned construction principle of a partial vine copula of $D$-vine structure 
differs from the one introduced in \citet[p.1264]{spanhel2019}. It is, however, straightforward to verify that
constructing the partial vine copula version according to Spanhel and Kurz for the copula $C$ in Example \ref{D.SVC.PInd.3} yields the same output $\Pi$.
\end{remark}{}

%%%%%%%%%%%%%%%%%%%%%%%%%%%%%%%%%%%%%%%%%%%%%%%%%%%%%%%%%%%%%%%%%%%%%%%%%%%%%%%%%%%%%%%%%%%%%%%%%%%%%%%%%%%%%%%%%%%%%%%%%%%%%%%%%%%%%%%%%%%%%%%%%%%%%%%%%%%%%%%%%%%%%%%%%%%%%%%%%%%%%%%%%%%%%%%%%%%%%%%%%%%%%%%%%%%%%%%%%%%%%%%%%%%%%%%%%%%%%%%%%%%%%%%%%%%%%%%%%%%%%%%%%%%%%%%%%%%%%%%%%%%%%%%%%%%%%%%%%%%%%%%%%%%%%%%%%%%%%%%%

\subsection{Optimality of partial vine copulas and continuity of $\delta$}
\label{subsect.PVC}

Example \ref{D.SVC.PInd.3} allows to prove the following multivariate version of Theorem \ref{PVC.Gen.Dist.}: 

%\smallskip
\begin{theorem}{} \label{D.PVC.Gen.Dist.}
For every $d \geq 3$ there exists a copula $C \in \CC^d_{\rm c}$ fulfilling $d_\infty(C, \psi(C)) \geq \tfrac{1}{8}$ and we have 
$$
  \sup_{C \in \CC^d_{\rm c}} d_\infty \big(C,\psi(C)\big) \geq \frac{1}{8}.
$$
\end{theorem}{}
\begin{proof}{}
	Again consider the $d$-dimensional copula $C \in \CC^d_{\rm c}$ studied in Example \ref{D.SVC.PInd.3}. 
	In this case we have $\psi(C)=\Pi$ from which we get
\begin{eqnarray*}
  d_\infty \big(C,\psi(C)\big)
	&\geq& C \big( \tfrac{1}{2},\tfrac{1}{2},\tfrac{1}{2}, {\bf 1} \big) - \Pi \big( \tfrac{1}{2},\tfrac{1}{2},\tfrac{1}{2}, {\bf 1} \big) \\
	  &=&  C^{\rm Cube} \big( \tfrac{1}{2},\tfrac{1}{2},\tfrac{1}{2} \big) - 
	    \Pi \big( \tfrac{1}{2},\tfrac{1}{2},\tfrac{1}{2} \big)
	  = \frac{1}{4} - \frac{1}{8}
	  = \frac{1}{8},
\end{eqnarray*}
which implies the stated result.
\end{proof}{}

We conclude the paper with the multivariate versions of Theorem \ref{NonCont.dInf.Lemma} and Corollary \ref{discont.everywhere}:

\begin{theorem}{} \label{NonCont.dInf.Lemma.delta}\label{NonCont.dInf.delta}
Suppose that $C \in \CC_{\rm c}^d$ satisfies $d_\infty (C,\psi(C)) \neq 0$. Then $C$ is a discontinuity point of 
the mapping $\psi: \CC^d_{\rm c} \to \CC^d_{\rm c}$ assigning every copula its partial D-vine.  
\end{theorem}{}
\begin{proof}{}
  Proceeding analogous to the proof of Theorem \ref{NonCont.dInf.Lemma} and using the fact that empirical copulas are 
  invariant under $\psi$ and converge to the true copula with respect to $d_\infty$ yields the result. 
 \end{proof}{}

Again using convex combinations it is straightforward to verify that the set of all $C \in \mathcal{C}^d_c$ that 
are not universally simplified is dense in $(\mathcal{C}^d_c,d_\infty)$ - Theorem \ref{NonCont.dInf.Lemma.delta} has the following consequence: 
\begin{corollary}{} \label{discont.everywhere.delta}
  The mapping $\psi: \CC^d_{\rm c} \to \CC^d_{\rm c}$ is discontinuous on a dense subset of $(\mathcal{C}^d_c,d_\infty)$. 
\end{corollary}{}
%\smallskip

\begin{remark}{}
As a consequence of Remark \ref{Rem.SpanhelKurz}, 
all the results presented in Subsection \ref{subsect.PVC} 
(i.e., Theorems \ref{D.PVC.Gen.Dist.}, \ref{NonCont.dInf.Lemma.delta} and
Corollary \ref{discont.everywhere.delta}) 
remain true for the partial vine copula of $D$-vine structure as discussed in \citet[p.1264]{spanhel2019}.
It is worth mentioning that, although the latter construction principle sequentially minimizes the Kullback-Leibler 
divergence related to each tree, its outcome can be quite far away from the data generating 
copula (see Theorem \ref{D.PVC.Gen.Dist.}) which is in line with Theorem \ref{Non.Dense.TV}.
\end{remark}

%\bigskip
%  At this point it is worth to mention that the distance between copula $D$ in the proof of 
%  Theorem \ref{PVC.PInd.Dist.} and its partial vine copula $\psi(D)$ equals 
%\begin{hylist}
%\item
%$50 \%$ of the maximal distance between to copulas within the Fr{\'e}chet class of all copulas having %pairwise independent marginals
%\item
%and still $18.75 \%$ of the maximal distance between to arbitrary copulas.
%\end{hylist}

%%%%%%%%%%%%%%%%%%%%%%%%%%%%%%%%%%%%%%%%%%%%%%%%%%%%%%%%%%%%%%%%%%%%%%%%%%%%%%%%%%%%%%%%%%%%%%%%%%%%%%%%%%%%%%%%%%%%%%%%%%%%%%%%%%%%%%%%%%%%%%%%%%%%%%%%%%%%%%%%%%%%%%%%%%%%%%%%%%%%%%%%%%%%%%%%%%%%%%%%%%%%%%%%%%%%%%%%%%%%%%%%%%%%%%%%%%%%%%%%%%%%%%%%%%%%%%%%%%%%%%%%%%%%%%%%%%%%%%%%%%%%%%%%%%%%%%%%%%%%%%%%%%%%%%%%%%%%%%%%%%%%%%%%%%%%%%%%%%%%%%%%%%%%%%%%%%%%%%%%%%%%%%%%%%%%%%%%%%%%%%%%%%%

\appendix
\section{Supplementary material}

%\smallskip
\begin{lemma}{} \label{App.1} \leavevmode
\begin{onelist}
\item 
Suppose that $F,F_1,F_2,\ldots$ are univariate distribution functions and suppose that $F$ is continuous. 
Then weak convergence $F_n \to F$ implies uniform convergence. 
\item 
Suppose that $F,F_1,F_2,\ldots$ are $d$-dimensional distribution functions ($d\geq 2$) and suppose that $F$ is continuous. 
 Then weak convergence $F_n \to F$ implies uniform convergence. 
\end{onelist}
\end{lemma}{}

\begin{proof}{}
Since the first statement is well-known and straightforward to verify we focus on the second assertion. 
Considering that $F$ is continuous the sequence $(F_n)_{n \in \mathbb{N}}$ converges pointwise to $F$ and the same holds true for all univariate marginals. Using Sklar's Theorem, Lipschitz continuity of copulas and statement (1) we get 
$$
  \big| F_n(\xxx) - F(\xxx) \big|
	\leq \sum_{i=1}^d \big| (F_n)_i (x_i) - F_i (x_i) \big|
$$
for every $\xxx \in \R^d$, which completes the proof.
\end{proof}{}
%\smallskip

\begin{lemma}{} \label{App.2}
Suppose that $F,F_1,F_2,\ldots$ are $d$-dimensional distribution functions with continuous marginals $(F)_i,(F_1)_i,(F_2)_i,\ldots$ 
($i \in \{1, \ldots,d\}$) and copulas $C, C_1,C_2,\ldots$, respectively. Then the following assertions hold: 
\begin{onelist}
\item
  If $C_n \to C$ uniformly and $(F_n)_i \to (F)_i$ weakly then $F_n \to F$ uniformly.

\item
  If $F_n \to F$ weakly then $C_n \to C$ uniformly.
\end{onelist}
\end{lemma}{}

\begin{proof}{}
  Since the limits are continuous by assumption, according to Lemma \ref{App.1} weak and uniform convergence coincide. 
  We start with proving the first assertion and consider some $\xxx \in \R^d$. 
  Then Lipschitz continuity of copulas and the triangle inequality yield (we write ${\bf F}_n:=((F_n)_1,(F_n)_2,\ldots, (F_n)_d)$)
\begin{eqnarray*}
  \big| F_n(\xxx) - F(\xxx) \big|
	&   =  & \big| (C_n \circ {\bf F}_n) (\xxx) - (C \circ {\bf F}) (\xxx) \big|
	\\
	& \leq & \big| (C_n \circ {\bf F}_n) (\xxx) - (C \circ {\bf F}_n) (\xxx) \big|
	         + \big| (C \circ {\bf F}_n) (\xxx) - (C \circ {\bf F}) (\xxx) \big|
	\\
	& \leq & d_\infty \big( C_n, C\big)
	         + \sum_{i=1}^d \big| (F_n)_i (x_i) - (F)_i (x_i) \big|,
\end{eqnarray*}
  from which the first assertion follows immediately. \\
  To prove the second assertion fix $\uuu \in \I^d$. Letting $(F_n)_i^\leftarrow$ denote the quasi-inverse of $(F_n)_i$ and letting 
  ${\bf F}_n^\leftarrow$ accordingly denote the vector of quasi-inverses of the univariate marginals yields  
\begin{eqnarray*}
  \lefteqn{\big| C_n(\uuu) - C(\uuu) \big|}
	\\
	&   =  & \big| (F_n \circ {\bf F}_n^\leftarrow) (\uuu) - (F \circ {\bf F}^\leftarrow) (\uuu) \big|
	\\
	& \leq & \big| (F_n \circ {\bf F}_n^\leftarrow) (\uuu) - (F_n \circ {\bf F}^\leftarrow) (\uuu) \big|
					 + \big| (F_n \circ {\bf F}^\leftarrow) (\uuu) - (F \circ {\bf F}^\leftarrow) (\uuu) \big|
	\\
	& \leq & \sum_{i=1}^{d} \big| \big( (F_n)_i \circ (F_n)_i^\leftarrow \big) (u_i) - \big( (F_n)_i \circ (F)_i^\leftarrow \big) (u_i) \big|
					 + \big| F_n \big( {\bf F}^\leftarrow (\uuu) \big) - F \big( {\bf F}^\leftarrow (\uuu) \big) \big|
	\\
	&   =  & \sum_{i=1}^{d} \big| u_i - \big( (F_n)_i \circ (F)_i^\leftarrow \big) (u_i) \big|
					 + \big| F_n \big( {\bf F}^\leftarrow (\uuu) \big) - F \big( {\bf F}^\leftarrow (\uuu) \big) \big|
	\\
	&   =  & \sum_{i=1}^{d} \big| (F)_i \big( (F)_i^\leftarrow (u_i) \big) - (F_n)_i \big( (F)_i^\leftarrow (u_i) \big) \big|
					 + \big| F_n \big( {\bf F}^\leftarrow (\uuu) \big) - F \big( {\bf F}^\leftarrow (\uuu) \big) \big|
	\\
	&   \leq & \sum_{i=1}^d d_\infty \big( (F_n)_i, (F)_i \big)
					 + \sup_{\mathbf{x} \in \mathbb{R}^d} \vert  F_n(\mathbf{x}) - F(\mathbf{x}) \vert.
\end{eqnarray*}
  This completes the proof. 
\end{proof}{}

\begin{lemma}{} \label{App.3}
Suppose that $C \in \CC_{\rm ac}^3$ is an absolutely continuous copula, 
and let $c_{13},c_{23}$ denote the densities of the marginal copulas $C_{13},C_{23}$ of $C$, respectively. 
Then the following inequality holds for every $\tilde{C} \in \CC_{\rm c}^3$: 
\begin{eqnarray*}
  \lefteqn{\int\limits_{\I^2} \int\limits_{\I} 
				 \big| C_{12;3}^t (\sss) - \tilde{C}_{12;3}^t (\sss) \big|
				 \; \mathrm{d} \lambda(t) \mathrm{d} \lambda^2 (\sss)}
	\\
	& \leq & \int\limits_{\I} \int\limits_{\I^2} 
				\Big| K_{C} \big( t, [{\bf 0},\sss] \big) - K_{\tilde{C}} \big( t, [{\bf 0},\sss] \big) \Big| \;
				\Big( c_{13} (s_1,t) \, c_{23} (s_2,t) \Big) 
				\; \mathrm{d} \lambda^2 (\sss) \mathrm{d} \lambda(t) 
  \\
	&      & + \int\limits_{\I} \int\limits_{\I} 
				\big| F_{1|3} (s_1|t) - \tilde{F}_{1|3} (s_1|t) \big| \, c_{13} (s_1,t)
				   \; \mathrm{d} \lambda (s_1) \mathrm{d} \lambda(t) 
	\\
	&      & + \int\limits_{\I} \int\limits_{\I} 
				\big| F_{2|3} (s_2|t) - \tilde{F}_{2|3} (s_2|t) \big| \, c_{23} (s_2,t)
				   \; \mathrm{d} \lambda (s_2) \mathrm{d} \lambda(t) 
\end{eqnarray*}
\end{lemma}{}
\begin{proof}{}
For $\tilde{C} \in \CC_{\rm c}^3$ and $C \in \CC_{\rm ac}^3$ we have 
\begin{eqnarray*}
  \lefteqn{\int\limits_{\I^2} \int\limits_{\I} 
				 \big| C_{12;3}^t (\sss) - \tilde{C}_{12;3}^t (\sss) \big|
				 \; \mathrm{d} \lambda(t) \mathrm{d} \lambda^2 (\sss)}
	\\
	&   =  & \int\limits_{\I^2} \int\limits_{\I} 
				\Big| K_{C} \big( t, \big[0,F^{\leftarrow}_{1|3} (s_1|t)\big] \times \big[0,F^{\leftarrow}_{2|3} (s_2|t) \big] \big)
	\\
	&      & \qquad \qquad - \, K_{\tilde{C}} \big( t, \big[0,\tilde{F}^{\leftarrow}_{1|3} (s_1|t)\big] \times \big[0,\tilde{F}^{\leftarrow}_{2|3} (s_2|t) \big] \big) \Big| 
				\; \mathrm{d} \lambda(t) \mathrm{d} \lambda^2 (\sss)
	\\
	& \leq & \int\limits_{\I^2} \int\limits_{\I} 
				\Big| K_{C} \big( t, \big[0,F^{\leftarrow}_{1|3} (s_1|t)\big] \times \big[0,F^{\leftarrow}_{2|3} (s_2|t) \big] \big)
	\\
	&      & \qquad \qquad - \, K_{\tilde{C}} \big( t, \big[0,F^{\leftarrow}_{1|3} (s_1|t)\big] \times \big[0,F^{\leftarrow}_{2|3} (s_2|t) \big] \big) \Big| 
				\; \mathrm{d} \lambda(t) \mathrm{d} \lambda^2 (\sss)
  \\
	&      &	+ \int\limits_{\I^2} \int\limits_{\I} 
				\Big| K_{\tilde{C}} \big( t, \big[0,F^{\leftarrow}_{1|3} (s_1|t)\big] \times \big[0,F^{\leftarrow}_{2|3} (s_2|t) \big] \big)
	\\
	&      & \qquad \qquad - \, K_{\tilde{C}} \big( t, \big[0,\tilde{F}^{\leftarrow}_{1|3} (s_1|t)\big] \times \big[0,\tilde{F}^{\leftarrow}_{2|3} (s_2|t) \big] \big) \Big| 
				\; \mathrm{d} \lambda(t) \mathrm{d} \lambda^2 (\sss)
	\\*
	&   =: & I_1 + I_2.
\end{eqnarray*}
For every $t \in \I$ define $T^t: \I^2 \to \I^2$ by  
$$ 
T^t (\sss) := \big( F^{\leftarrow}_{1|3} (s_1|t), F^{\leftarrow}_{2|3} (s_2|t)\big). 
$$
Then $T^t$ is measurable, obviously satisfies $(T^t)^{-1} (\I^2)=\I^2$, and 
\begin{eqnarray*}
  (\lambda^2)^{T^t} \big( [{\bf 0},\uuu] \big)
	& = & \lambda^2 \big( \big\{ \sss \in \I^2 \, : \, T^t (\sss) \in [{\bf 0},\uuu] \big\} \big)
	\\
	& = & \lambda^2 \big( \big\{ \sss \in \I^2 \, : \, s_1 \leq F_{1|3} (u_1|t), s_2 \leq F_{2|3} (u_2|t) \big\} \big)
	\\
	& = & F_{1|3} (u_1|t) \, F_{2|3} (u_2|t)
	\\
	& = & \int\limits_{[{\bf 0},\uuu]} c_{13} (a_1,t) \, c_{23} (a_2,t) 
				\; \mathrm{d} \lambda^2 (\aaa)
\end{eqnarray*}
for every $\uuu \in \I^2$, implying that $(\lambda^2)^{T^t}$ is absolutely continuous with density 
$(a_1,a_2) \mapsto c_{13} (a_1,t) \, c_{23} (a_2,t)$. 
This yields
\begin{eqnarray*}
  I_1
	&   =  & \int\limits_{\I} \int\limits_{\I^2} 
				\Big| K_{C} \big( t, \big[0,F^{\leftarrow}_{1|3} (s_1|t)\big] \times \big[0,F^{\leftarrow}_{2|3} (s_2|t) \big] \big)
	\\
	&      & \qquad \qquad - \, K_{\tilde{C}} \big( t, \big[0,F^{\leftarrow}_{1|3} (s_1|t)\big] \times \big[0,F^{\leftarrow}_{2|3} (s_2|t) \big] \big) \Big| 
				\; \mathrm{d} \lambda^2 (\sss) \mathrm{d} \lambda(t) 
 \\
	&   =  & \int\limits_{\I} \int\limits_{\I^2} 
				\big| K_{C} \big( t, [{\bf 0},\sss] \big) - K_{\tilde{C}} \big( t, [{\bf 0},\sss] \big) \big| 
				\; \mathrm{d} (\lambda^2)^{T^t} (\sss) \mathrm{d} \lambda (t)
	\\
	&   =  & \int\limits_{\I} \int\limits_{\I^2} 
				\Big| K_{C} \big( t, [{\bf 0},\sss] \big) - K_{\tilde{C}} \big( t, [{\bf 0},\sss] \big) \Big| \;
				\Big( c_{13} (s_1,t) \, c_{23} (s_2,t) \Big) 
				\; \mathrm{d} \lambda^2 (\sss) \mathrm{d} \lambda(t). 
\end{eqnarray*}
Focusing on $I_2$, using Sklar's theorem, Lipschitz continuity, and a similar argument as before yields
\begin{eqnarray*}
  I_2
	&   =  & \int\limits_{\I} \int\limits_{\I^2} 
				\Big| K_{\tilde{C}} \big( t, \big[0,F^{\leftarrow}_{1|3} (s_1|t)\big] \times \big[0,F^{\leftarrow}_{2|3} (s_2|t) \big] \big)
	\\
	&      & \qquad \qquad - \, K_{\tilde{C}} \big( t, \big[0,\tilde{F}^{\leftarrow}_{1|3} (s_1|t)\big] \times \big[0,\tilde{F}^{\leftarrow}_{2|3} (s_2|t) \big] \big) \Big| 
				\; \mathrm{d} \lambda^2 (\sss) \mathrm{d} \lambda(t) 
	\\
	& \leq & \int\limits_{\I} \int\limits_{\I} 
				\big| \tilde{F}_{1|3} \big( F^{\leftarrow}_{1|3} (s_1|t) \big| t \big) 
							- \tilde{F}_{1|3} \big( \tilde{F}^{\leftarrow}_{1|3} (s_1|t) \big| t \big) \big| 
				   \; \mathrm{d} \lambda (s_1) \mathrm{d} \lambda(t) 
	\\*
	&      & + \int\limits_{\I} \int\limits_{\I} 
				\big| \tilde{F}_{2|3} \big( F^{\leftarrow}_{2|3} (s_2|t) \big| t \big) 
							- \tilde{F}_{2|3} \big( \tilde{F}^{\leftarrow}_{2|3} (s_2|t) \big| t \big) \big| 
				   \; \mathrm{d} \lambda (s_2) \mathrm{d} \lambda(t) 
	\\
	&   =  & \int\limits_{\I} \int\limits_{\I} 
				\big| \tilde{F}_{1|3} \big( F^{\leftarrow}_{1|3} (s_1|t) \big| t \big) - s_1 \big| 
				   \; \mathrm{d} \lambda (s_1) \mathrm{d} \lambda(t) 
	\\
	&      & + \int\limits_{\I} \int\limits_{\I} 
				\big| \tilde{F}_{2|3} \big( F^{\leftarrow}_{2|3} (s_2|t) \big| t \big) - s_2 \big| 
				   \; \mathrm{d} \lambda (s_2) \mathrm{d} \lambda(t) 
	\\
	&   =  & \int\limits_{\I} \int\limits_{\I} 
				\big| \tilde{F}_{1|3} \big( F^{\leftarrow}_{1|3} (s_1|t) \big| t \big) - F_{1|3} \big( F^{\leftarrow}_{1|3} (s_1|t) \big| t \big) \big| 
				   \; \mathrm{d} \lambda (s_1) \mathrm{d} \lambda(t) 
	\\
	&      & + \int\limits_{\I} \int\limits_{\I} 
				\big| \tilde{F}_{2|3} \big( F^{\leftarrow}_{2|3} (s_2|t) \big| t \big) - F_{2|3} \big( F^{\leftarrow}_{2|3} (s_2|t) \big| t \big) \big| 
				   \; \mathrm{d} \lambda (s_2) \mathrm{d} \lambda(t) 
	\\
	&   =  & \int\limits_{\I} \int\limits_{\I} 
				\big| \tilde{F}_{1|3} (s_1|t) - F_{1|3} (s_1|t) \big| \, c_{13} (s_1,t)
				   \; \mathrm{d} \lambda (s_1) \mathrm{d} \lambda(t) 
	\\
	&      & + \int\limits_{\I} \int\limits_{\I} 
				\big| \tilde{F}_{2|3} (s_2|t) - F_{2|3} (s_2|t) \big| \, c_{23} (s_2,t)
				   \; \mathrm{d} \lambda (s_2) \mathrm{d} \lambda(t), 
\end{eqnarray*}
and the proof is complete.
\end{proof}{}

As a direct consequence of Lemma \ref{App.3} we obtain the following result:
\begin{lemma}\label{lem_Jn}
Suppose that $C \in \CC_{\rm ac}^3$ is an absolutely continuous copula whose density $c$ 
fulfills $c \leq a \in [1,\infty)$. Then the inequality
$$
  J(C,\tilde{C})
	 :=  \int_{\I^3} \vert C^t_{12;3}(\sss) - \tilde{C}^t_{12;3}(\sss) \vert 
	     \; \mathrm{d} \lambda^3(\sss,t) 
	\leq a (2+a) D_\infty(C,\tilde{C})
$$
holds for every $\tilde{C} \in \CC_{\rm c}^3$.
\end{lemma}
\begin{proof}{}
Applying Lemma \ref{App.3} and Theorem \ref{Thm.D1.TV} yields
\begin{eqnarray*}
  \lefteqn{\int_{\I^3} \vert C^t_{12;3}(\sss) - \tilde{C}^t_{12;3}(\sss) \vert \; \mathrm{d} \lambda^3(\sss,t)}
	\\
	& \leq & \int\limits_{\I} \int\limits_{\I^2} 
				\Big| K_{C} \big( t, [{\bf 0},\sss] \big) - K_{\tilde{C}} \big( t, [{\bf 0},\sss] \big) \Big| \;
				\Big( c_{13} (s_1,t) \, c_{23} (s_2,t) \Big) 
				\; \mathrm{d} \lambda^2 (\sss) \mathrm{d} \lambda(t) 
  \\
	&      & + \int\limits_{\I} \int\limits_{\I} 
				\big| F_{1|3} (s_1|t) - \tilde{F}_{1|3} (s_1|t) \big| \, c_{13} (s_1,t)
				   \; \mathrm{d} \lambda (s_1) \mathrm{d} \lambda(t) 
	\\
	&      & + \int\limits_{\I} \int\limits_{\I} 
				\big| F_{2|3} (s_2|t) - \tilde{F}_{2|3} (s_2|t) \big| \, c_{23} (s_2,t)
				   \; \mathrm{d} \lambda (s_2) \mathrm{d} \lambda(t) 
	\\
	& \leq & a^2 \, D_1 (C,\tilde{C}) 
					 + a \, D_1(C_{13},\tilde{C}_{13})
           + a \, D_1(C_{23},\tilde{C}_{23})
	\\
	& \leq & a^2 \, D_\infty (C,\tilde{C}) 
					 + a \, D_\infty(C_{13},\tilde{C}_{13})
           + a \, D_\infty(C_{23},\tilde{C}_{23})
	\\
	& \leq & a^2 \, D_\infty (C,\tilde{C}) 
					 + 2 a \, D_\infty(C,\tilde{C})
\end{eqnarray*}
This proves the assertion.
\end{proof}{}

Suppose that $C \in \CC^3$ is a checkerboard copula. Then $C \in \CC^3_{\rm ac}$. 
We will say that $C$ has resolution $N \geq 2$ if $N$ is the smallest integer such that (there is a version of) its density 
$c$ of $C$ is constant on each square of the form 
$(\frac{i_x-1}{N},\frac{i_x}{N}) \times (\frac{i_y-1}{N},\frac{i_y}{N}) \times (\frac{i_z-1}{N},\frac{i_z}{N})$ 
with $i_x,i_y,i_z \in \{1,\ldots,N\}$. 
Notice that if $C$ is a checkerboard copula with resolution $N$ then its density $c$ fulfills $c(\mathbf{u},t) \leq N^2$ for
$\lambda^3$-almost all $(\mathbf{u},t) \in \I^3$.  
Given a checkerboard copula $C$ with resolution $N$ w.l.o.g. we may assume that the mapping 
$t \mapsto C_{12;3}^t$ is constant on each interval of the form $[\frac{i-1}{N},\frac{i}{N})$, $i \in \{1,\ldots,N\}$, 
and define the quantity $\Delta=\Delta(C)$ by
$$
  \Delta(C)
	:= \max_{i,j \in \{1,\ldots,N\}} \int_{\I^2} \vert C^{t_i}_{12;3}(\uuu) - C^{t_j}_{12;3}(\uuu) \vert
 \; \mathrm{d} \lambda^2(\uuu)
$$
whereby $t_i=\frac{i-1}{N}$ for every $i \in \{1,\ldots,N\}$.
\begin{lemma}\label{lem:Delta(C)}
Suppose that $C_1,C_2,C_3, \ldots \in \CC_{\rm S}^3$ are simplified copulas and that $C$ is a non-simplified checkerboard copula with resolution $N \geq 2$.
Then the quantity $J (C_n,C)$ from Lemma \ref{lem_Jn} fulfills 
\begin{equation}
J (C_n,C) \geq \Delta(C)/N > 0
\end{equation} 
for every $n \in \mathbb{N}$. 
As a direct consequence, there is no sequence $(C_n)_{n \in \mathbb{N}}$ in $\CC_{\rm S}^3$ that converges to 
$C$ w.r.t. $D_\infty$ ($D_1$) or weakly conditional. 
\end{lemma}
\begin{proof}
Under the assumptions of the lemma we obviously have
\begin{eqnarray*}
  J (C_n,C)
	&   =  & \sum_{i=1}^N \frac{1}{N} \int_{\I^2} \vert C^{t_i}_{12;3}(\uuu) - (C_n)_{12;3}(\uuu) \vert 
					 \; \mathrm{d} \lambda^2(\uuu) 
  \\
  & \geq & \frac{1}{N} \, \max_{i,j \in \{1,\ldots,N\}} \int_{\I^2} \vert C^{t_i}_{12;3}(\uuu) - C^{t_j}_{12;3}(\uuu) \vert
					 \; \mathrm{d} \lambda^2 (\uuu) 
	\\
	&   =  & \frac{\Delta(C)}{N} > 0.
\end{eqnarray*}
The second assertion now follows from Lemma \ref{lem_Jn} and the fact that $\Delta(C)$ only depends on $C$ and not on $n$, 
the assertion concerning weak conditional convergence from the fact that weak conditional convergence implies 
convergence w.r.t. $D_1$.
\end{proof}

\begin{lemma}\label{lem:dense} \leavevmode
\begin{enumerate}
\item 
The family of all non-simplified checkerboards is dense in $(\CC^3,D_\infty)$, in $(\CC^3,D_1)$, and 
dense in $\CC^3$ endowed with the topology induced by weak conditional convergence. 
%\item
%The family of all simplified copulas in $\CC^3_{\rm c}$ 
%is dense in the family of all generalized simplified copulas w.r.t. $D_\infty$ ($D_1$). 
\item 
The family of all non-simplified checkerboards with positive density 
is dense in the family of all absolutely continuous copulas with positive density w.r.t. $D_\infty$, wr.t. $D_1$, 
and w.r.t. the topology induced by weak conditional convergence. 
\end{enumerate}
\end{lemma}
\begin{proof}
To prove the first assertion let $C \in \CC^3$ be arbitrary but fixed. 
Since according to \cite{trutschnig2015} checkerboard copulas are dense in $(\CC^3,D_1)$ 
we can find a sequence $(B_n)_{n \in \mathbb{N}}$ of checkerboard copulas with 
$\lim_{n \rightarrow \infty} D_1(B_n,C)=0$.
For every $n \in \mathbb{N}$ 
let $E_n$ be a non-simplified checkerboard copula with the same resolution  
and the same $(1,3)$- and $(2,3)$-marginals as $B_n$. 
Setting $C_n:=(1-\frac{1}{n}) B_n + \frac{1}{n} E_n$ for every $n \in \mathbb{N}$ 
yields a sequence $(C_n)_{n \in \mathbb{N}}$ of non-simplified checkerboard copulas. Considering 
$$
  D_1(C_n,C) 
	\leq (1-\tfrac{1}{n}) D_1(C,B_n) + \tfrac{1}{n} D_1(C,E_n)
$$   
it follows that $\lim_{n \rightarrow \infty} D_1(C_n,C)=0$, which completes the proof of the first assertion concerning
$D_1$ and $D_\infty$. The assertion concerning weak conditional convergence can be shown analogously: in fact, it is straightforward to extend the bivariate proof in \cite[Theorem 3.2]{kasper2020} to the three-dimensional setting, hence reusing
the convex combination idea and considering $C_n:=(1-\frac{1}{n}) B_n + \frac{1}{n} E_n$ yields the desired result. 
%\\
%To show the second assertion we proceed as follows: 
%For every triplet $(A,B^\ast,B^{\ast\ast})$ of bivariate copulas define the simplified copula 
%$m(A,B^\ast,B^{\ast\ast}) \in \CC^3$ by 
%$$
%  \big( m(A,B^\ast,B^{\ast\ast}) \big) (\uuu,v)
%	:= \int_{[0,v]} A \big( K_{B^\ast}(t,[0,u_1]), K_{B^{\ast\ast}}(t,[0,u_2]) \big) 
%	   \; \mathrm{d} \lambda(t)
%$$ 
%Notice that $m(A,B^\ast,B^{\ast\ast})$ is not necessarily a checkerboard even though 
%$A,B^\ast,B^{\ast\ast}$ are checkerboards with the same resolution.
%It is straightforward to verify that the following inequality holds for all 
%$A_1,A_2,B^\ast_1,B^\ast_2,B^{\ast\ast}_1,B^{\ast\ast}_2 \in \CC^2$:
%$$
%  D_1 \big( m(A_1,B^\ast_1,B^{\ast\ast}_1),\, m(A_2,B^\ast_2,B^{\ast\ast}_2) \big) 
%	\leq d_\infty(A_1,A_2) + D_1(B^\ast_1,B^\ast_2) + D_1(B^{\ast\ast}_1,B^{\ast\ast}_2)
%$$
%Having this and using the fact that checkerboard copulas are absolutely continuous and that the family of all bivariate checkerboards is dense in $(\CC^2,D_1)$ the assertions follows immediately.
\\
To prove the second assertion suppose that $C \in \CC^3_{\rm ac}$ has positive density. 
According to the first assertion we can find a sequence $(B_n)_{n \in \mathbb{N}}$ of non-simplified checkerboard 
copulas with $\lim_{n \rightarrow \infty} D_1(B_n,C)=0$.
Setting $C_n:=(1-\frac{1}{n}) B_n + \frac{1}{n} \Pi$ for every $n \in \mathbb{N}$ 
yields a sequence $(C_n)_{n \in \mathbb{N}}$ of non-simplified checkerboard copulas with positive density. 
Considering 
$$
  D_1(C_n,C) 
	\leq (1-\tfrac{1}{n}) D_1(C,B_n) + \tfrac{1}{n} D_1(C,\Pi)
$$   
we get $\lim_{n \rightarrow \infty} D_1(C_n,C)=0$. Since the assertion for weak conditional convergence can be shown analogously, the proof is complete. 
\end{proof}

%%%%%%%%%%%%%%%%%%%%%%%%%%%%%%%%%%%%%%%%%%%%%%%%%%%%%%%%%%%%%%%%%%%%%%%%%%%%%%%%%%%%%%%%%%%%%%%%%%%%%%%%%%%%%%%%%%%%%%%%%%%%%%%%%%%%%%%%%%%%%%%%%%%%%%%%%%%%%%%%%%%%%%%%%%%%%%%%%%%%%%%%%%%%%%%%%%%%%%%%%%%%%%%%%%%%%%%%%%%%%%%%%%%%%%%%%%%%%%%%%%%%%%%%%%%%%%%%%%%%%%%%%%%%%%%%%%%%%%%%%%%%%%%%%%%%%%%%%%%%%%%%%%%%%%%%%%%%%%%%%%%%%%%%%%%%%%%%%%%%%%%%%%%%%%%%%%%%%%%%%%%%%%%%%%%%%%%%%%%%%%%%%%%

\section{Proofs}

{\bf Proof of Lemma} \ref{Cont.D1.Lemma1}: 
  For every $u \in \I$ we have
\begin{eqnarray*}
  \lefteqn{\int\limits_{\I}  
				   \big| K_{(C_n)_{13}} \big( v, [0,u] \big) - K_{C_{13}} \big( v, [0,u] \big) \big| 
					 \; \mathrm{d} \lambda(v)}
	\\
	&   =  & \int\limits_{\I} 
					 \big| K_{C_n} \big( v, [0,u] \times \I \big) - K_{C} \big( v, [0,u] \times \I \big) \big| 
					 \; \mathrm{d} \lambda(v)
	\\
	& \leq & \sup_{\uuu \in \I^2} \; \int\limits_{\I} 
					 \big| K_{C_n} \big( v, [{\bf 0},\uuu] \big) - K_{C} \big( v, [{\bf 0},\uuu] \big) \big| 
					\; \mathrm{d} \lambda(v)
\end{eqnarray*}
and hence 
$ D_\infty \big( (C_n)_{13}, C_{13}) \leq D_\infty (C_n,C) $.
  Since $D_1$-convergence is equivalent to $D_\infty$-convergence (see \cite{trutschnig2015}) this proves (1).
  We now prove the second assertion. Using Lipschitz continuity of copulas we obtain
\begin{eqnarray*}
  \lefteqn{D_1 \big( \psi (C_n), \psi(C) \big)}
	\\
	&   =  & \int\limits_{\I^2} \int\limits_{\I} 
					 \big| (C_n)_{p} \big( (F_n)_{1|3}(u_1|v), (F_n)_{2|3}(u_2|v) \big) 
								 - C_{p} \big( F_{1|3}(u_1|v), F_{2|3}(u_2|v) \big) \big| 
					 \; \mathrm{d} \lambda(v) \mathrm{d} \lambda^2(\uuu)
  \\
	& \leq & \int\limits_{\I^2} \int\limits_{\I} 
					 \big| (C_n)_{p} \big( (F_n)_{1|3}(u_1|v), (F_n)_{2|3}(u_2|v) \big) 
								 - (C_n)_{p} \big( F_{1|3}(u_1|v), F_{2|3}(u_2|v) \big) \big| 
					 \; \mathrm{d} \lambda(v) \mathrm{d} \lambda^2(\uuu)
	\\
	&      &  + \int\limits_{\I^2} \int\limits_{\I} 
					 \big| (C_n)_{p} \big( F_{1|3}(u_1|v), F_{2|3}(u_2|v) \big) 
								 - C_{p} \big( F_{1|3}(u_1|v), F_{2|3}(u_2|v) \big) \big| 
					 \; \mathrm{d} \lambda(v) \mathrm{d} \lambda^2(\uuu)
	\\
	& \leq & \int\limits_{\I^2} \int\limits_{\I}  
					 \big| (F_n)_{1|3}(u_1|v) - F_{1|3}(u_1|v) \big| 
					 + \big| (F_n)_{2|3}(u_2|v) - F_{2|3}(u_2|v) \big| 
					 \; \mathrm{d} \lambda(v) \mathrm{d} \lambda^2(\uuu)
	\\
	&      &  + \int\limits_{\I^2} \int\limits_{\I} 
					 \big| (C_n)_{p} \big( F_{1|3}(u_1|v), F_{2|3}(u_2|v) \big) 
								 - C_{p} \big( F_{1|3}(u_1|v), F_{2|3}(u_2|v) \big) \big| 
					 \; \mathrm{d} \lambda(v) \mathrm{d} \lambda^2(\uuu)
	\\
	& \leq & D_1 \big( (C_n)_{13}, C_{13} \big) + D_1 \big( (C_n)_{23}, C_{23} \big)
           + d_{\infty} \big( (C_n)_{p}, C_{p} \big)
\end{eqnarray*}
from which the assertion follows.\hfill$\Box$
\\
\\ 
{\bf Proof of Lemma} \ref{Cont.D1.Lemma2}: 
  We first have
\begin{eqnarray*}
  d_\infty \big( (C_n)_p, C_p \big) 
	&   =  & \sup_{\sss \in \I^2} 
					 \left| \int\limits_{\I} (C_n)_{12;3}^t (\sss) \;\mathrm{d}\lambda(t)
									- \int\limits_{\I} C_{12;3}^t (\sss) \;\mathrm{d}\lambda(t) \right|
	\\
	%&   =  & \sup_{s,t \in \I}  \;  \biggl| \int\limits_{[0,1]} (C_n)_{13;2}^v (s,t) - C_{13;2}^v (s,t) \; \mathrm{d} \lambda(v) \biggr|
	%\\
	& \leq & \sup_{\sss \in \I^2} \int\limits_{\I} 
					 \big| (C_n)_{12;3}^t (\sss) - C_{12;3}^t (\sss) \big| \; \mathrm{d} \lambda(t) 
	\\
	&   =  & D_\infty \big( B_n, B \big)
\end{eqnarray*}
where the copulas $B_n,B$ are given by
$$
  B_n (\uuu,v)
	:= \int\limits_{[0,v]} (C_n)_{12;3}^t (\uuu) \; \mathrm{d} \lambda(t) 
	\qquad \textrm{ and } \qquad
  B (\uuu,v)
	:= \int\limits_{[0,v]} C_{12;3}^t (\uuu) \; \mathrm{d} \lambda(t).	
$$
Since $D_\infty-$convergence is equivalent to $D_1-$convergence it suffices to prove 
that $(B_n)_{n \in \mathbb{N}}$ converges to $B$ w.r.t. $D_1$, which can be done as follows:
Applying Lemma \ref{App.3} and H{\"o}lder's inequality yields
\begin{eqnarray*}
  \lefteqn{D_1 \big( B_n, B \big)}
	\\
	&   =  & \int\limits_{\I^2} \int\limits_{\I} 
				 \big| (C_n)_{12;3}^t (\sss) - C_{12;3}^t (\sss) \big|
				 \; \mathrm{d} \lambda(t) \mathrm{d} \lambda^2 (\sss)
	\\
	& \leq & \int\limits_{\I} \int\limits_{\I^2} 
				\Big| K_{C} \big( t, [{\bf 0},\sss] \big) - K_{C_n} \big( t, [{\bf 0},\sss] \big) \Big| \;
				\Big( c_{13} (s_1,t) \, c_{23} (s_2,t) \Big) 
				\; \mathrm{d} \lambda^2 (\sss) \mathrm{d} \lambda(t) 
	\\
	&     & + \int\limits_{\I} \int\limits_{\I} 
				\big| F_{1|3} (s_1|t) - (F_n)_{1|3} (s_1|t) \big| \, c_{13} (s_1,t)
				\; \mathrm{d} \lambda(s_1) \mathrm{d} \lambda (t)
	\\*
	&      & + \int\limits_{\I} \int\limits_{\I} 
				\big| F_{2|3} (s_2|t) - (F_n)_{2|3} (s_2|t) \big| \, c_{23} (s_2,t)
				\; \mathrm{d} \lambda(s_2) \mathrm{d} \lambda (t)
	\\
	& \leq & \left( \;\, \int\limits_{\I^2 \times \I} 
				\Big| K_{C} \big( t, [{\bf 0},\sss] \big) - K_{C_n} \big( t, [{\bf 0},\sss] \big) \Big|^p 
				\; \mathrm{d} \lambda^3 (\sss,t)  \right)^{\frac{1}{p}}
	\\
	&      & \qquad \cdot \left( \;\, \int\limits_{\I^2 \times \I} 
				\Big( c_{13} (s_1,t) \, c_{23} (s_2,t) \Big)^{\frac{p}{p-1}} 
				\; \mathrm{d} \lambda^3 (\sss,t) \right)^{\frac{p-1}{p}}
	\\
	&      & + \left( \;\, \int\limits_{\I \times \I} 
				\big| F_{1|3} (s_1|t) - (F_n)_{1|3} (s_1|t) \big|^q 
				\; \mathrm{d} \lambda^2(s_1,t) \right)^{\frac{1}{q}}
				\cdot \left( \;\, \int\limits_{\I \times \I} 
				\big| c_{13} (s_1,t) \big|^{\frac{q}{q-1}}
				\; \mathrm{d} \lambda^2(s_1,t) \right)^{\frac{q-1}{q}}
	\\
	&      & + \left( \;\, \int\limits_{\I \times \I} 
				\big| F_{2|3} (s_2|t) - (F_n)_{2|3} (s_2|t) \big|^r 
				\; \mathrm{d} \lambda^2(s_2,t) \right)^{\frac{1}{r}}
	      \cdot \left( \;\, \int\limits_{\I \times \I} 
				\big| c_{23} (s_2,t) \big|^{\frac{r}{r-1}}
				\; \mathrm{d} \lambda^2(s_2,t) \right)^{\frac{r-1}{r}}
\end{eqnarray*}
for all $p,q,r \in (1,\infty)$.
The latter expressions are finite by assumption, and the former part is bounded by
\begin{eqnarray*}
  \lefteqn{\left( \;\, \int\limits_{\I^2 \times \I}
				\Big| K_{C} \big( t, [{\bf 0},\sss] \big) - K_{C_n} \big( t, [{\bf 0},\sss] \big) \Big|^p 
				\; \mathrm{d} \lambda^3 (\sss,t)  \right)^{\frac{1}{p}}}
	\\
	& \leq & \left( \;\, \int\limits_{\I^2 \times \I}
				\Big| K_{C} \big( t, [{\bf 0},\sss] \big) - K_{C_n} \big( t, [{\bf 0},\sss] \big) \Big|^1 
				\; \mathrm{d} \lambda^3 (\sss,t) \right)^{\frac{1}{p}} 	=  D_1 (C_n, C)^{\frac{1}{p}}
\end{eqnarray*}
by $D_1 ((C_n)_{13}, C_{13})^{1/q}$ and $D_1 ((C_n)_{23}, C_{23})^{1/r}$, respectively.
Thus we conclude that $\lim_{n \to \infty} D_1 \big( B_n, B \big) = 0$ 
and hence $ \lim_{n \to \infty} d_\infty \big( (C_n)_p, C_p \big) = 0 $.
This proves the assertion.\hfill$\Box$

\section*{Acknowledgement}
The second and the third author gratefully acknowledge the support of the WISS 2025 project 
'IDA-lab Salzburg' (20204-WISS/225/197-2019 and 20102-F1901166-KZP).

%%%%%%%%%%%%%%%%%%%%%%%%%%%%%%%%%%%%%%%%%%%%%%%%%%%%%%%%%%%%%%%%%%%%%%%%%%%%%%%%%%%%%%%%%%%%%%%%%%%
%\newpage

%%%%%%%%%%%%%%%%%%%%%%%%%%%%%%%%%%%%%%%%%%%%%%%%%%%%%%%%%%%%%%%%%%%%%%%%%%%%%%%%%%%%%%%%%%%%%%%%%%%

\bigskip
\vfill\hspace*{\fill}\today
\end{document}